\documentclass[a4paper,12pt]{article}
\usepackage{amsmath,amsthm,amssymb}
\usepackage{mathrsfs}
\usepackage{times}
\usepackage{enumerate}
\usepackage{bbm}
\usepackage[square,sort,comma,numbers]{natbib}
\usepackage[colorlinks,linkcolor=blue,citecolor=blue,urlcolor=blue]{hyperref}
\usepackage{latexsym}
\usepackage{color,graphicx}
\usepackage{bm} 
\usepackage{float}
\usepackage{tikz}
\usepackage{cases}
\usepackage{tabularx}
\usepackage{ltablex}
\usepackage{verbatim}

\allowdisplaybreaks

\tikzstyle{new edge style 0}=[line width=1.5pt]

\makeatletter
\def\namedlabel#1#2{\begingroup
	#2%
	\def\@currentlabel{#2}%
	\phantomsection\label{#1}\endgroup
}
\makeatother

\newcommand{\TYPE}{{\mathbf{t}}}

\def\titlerunning#1{\gdef\titrun{#1}}
\makeatletter
\def\author#1{\gdef\autrun{\def\and{\unskip, }#1}\gdef\@author{#1}}

\makeatother

\def\subjclass#1{{\renewcommand{\thefootnote}{}%
		\footnote{\emph{Mathematics Subject Classification (2010):} #1}}}
\def\keywords#1{\par\medskip
	\noindent\textbf{Keywords.} #1}

\newtheorem{theorem}{Theorem}[section]
\newtheorem{corollary}[theorem]{Corollary}
\newtheorem{lemma}[theorem]{Lemma}
\newtheorem{proposition}[theorem]{Proposition}

\theoremstyle{definition}
\newtheorem{definition}[theorem]{Definition}
\newtheorem{exploration}[theorem]{Exploration}
\newtheorem{remark}[theorem]{Remark}

\DeclareMathAlphabet{\mathscrbf}{OMS}{mdugm}{b}{n}

\numberwithin{equation}{section}

\newcommand{\bx}{{\bf x}}

\newcommand{\bc}{{\bf c}}
\newcommand{\bw}{{\bf w}}

\newcommand{\bW}{{\bf W}}

\newcommand{\cvd}{\ell^2_{\downarrow}}

\newcommand{\eps}{\varepsilon}
\newcommand{\R}{\mathbb{R}}

\newcommand{\N}{\mathbb{N}}
\newcommand{\p}{\mathbb{P}}

\newcommand{\Df}{{D}}
\newcommand{\Dfu}{{D}_0^{\uparrow\uparrow}(\R_+)}
\newcommand{\Dfuw}{{D}_0^{\uparrow}(\R_+)}
\newcommand{\DJfu}{{D}_{0,J}^{\uparrow}(\R_+)}
\newcommand{\Dfmr}{{D}_{\vec \rho}(\R^m_+)}
\newcommand{\Dfp}{{D}_0^+(\R_+)}

\newcommand{\G}{{\mathcal{G}}}
\newcommand{\sS}{{\mathscr{S}}}

\newcommand{\cdl}{c\`{a}dl\`{a}g}

\newcommand{\Jcal}{\mathcal{J}}

\newcommand{\cE}{\mathcal{E}}

\newcommand{\cC}{\mathcal{C}}

\newcommand{\bT}{{\bf T}}

\newcommand{\sM}{{\mathscr{M}}}

\newcommand{\bbX}{{\mathbb{X}}}
\newcommand{\bbT}{{\mathbb{T}}}
\newcommand{\bbx}{{\mathbbm{x}}}

\newcommand{\cL}{{\mathcal{L}}}

\newcommand{\tcirc}{{\,\tilde\circ\,}}

\newcommand{\len}{\operatorname{len}}
\newcommand{\id}{\operatorname{id}}

\usepackage[textwidth=23mm,textsize=tiny,color=yellow!70]{todonotes}

\usepackage[margin=1in]{geometry}

\title{Degree corrected stochastic block model: \\ excursion representation}
\titlerunning{Excursion representation of SBM}
\author{David Clancy, Jr.\footnote{Department of Mathematics, University of Wisconsin -- Madison, 480 Lincoln Dr, Madison, WI, USA 53706},\ \ Vitalii Konarovskyi\footnote{Faculty of mathematics, informatics and natural sciences, University of Hamburg, Bundesstraße 55, 20146 Hamburg, Germany; Institute of Mathematics of NAS of Ukraine, Kyiv, Tereschenkivska st. 3, 01024 Kyiv, Ukraine},\ \ Vlada Limic\footnote{Centre National de la Recherche Scientifique; IRMA, UMR 7501, Université de Strasbourg,
7 rue René-Descartes, 67084 Strasbourg Cedex, France}} 
\date{\today}

\def\subjclass#1{{\renewcommand{\thefootnote}{}%
\footnote{\emph{Mathematics Subject Classification (2020):} #1}}}
\def\emails#1{{\renewcommand{\thefootnote}{}%
\footnote{\emph{E-mails:} #1}}}

\begin{document}

\maketitle

\emails{\href{mailto:dclancy@math.wisc.edu}{dclancy@math.wisc.edu},\href{mailto:vitalii.konarovskyi@math.uni-bielefeld.de}{vitalii.konarovskyi@uni-hamburg.de},\href{mailto:vlada@math.unistra.fr}{vlada@math.unistra.fr}}

\begin{abstract} This is the first of two complementary works in which we analyze the connected components of the degree-corrected stochastic block model (DCSBM). Our model is a random graph with an underlying community structure and degree in-homogeneity. 
It belongs to a class of non-rank one models. The scaling limit of connected component sizes in the near-critical regime, obtained by Konarovskyi and Limic (2021) for a subfamily of DCSBM, is 
non-trivially different (although related to) the standard eternal multiplicative coalescent of Aldous (1997). 
 
The Aldous (1997) excursion representation combined with weak convergence approach to the scaling limits of connected components of random graphs
proved to be much more difficult (and therefore rare) for non rank-one models. In this work we show how to build a random field encoding for the connected component structure of DCSBM, in part relying on the theory of Chaumont and Marolleau (2020). We then show how one can, under additional assumptions, reformulate the minimization problem stated in terms of multidimensional first hitting times 
into an equivalent minimization problem stated for a single real-valued stochastic process.
This reformulation relies on a novel composition-like operator on pairs of compatible non-decreasing rcll functions, which might be of independent interest.

\end{abstract}

\newpage
\tableofcontents

\section{Introduction}

Over the past several decades, random graphs have become an indispensable tool for studying real-world networks \cite[Chapter 1]{vanderHofstad.17}. Real-world networks are frequently both large and complicated so that a precise description is near impossible in practice. In an attempt to understand finer properties of these large networks, one typically constructs a family of finite random graph models, and studies various structural properties of these  random graphs as $n$ gets large.

A fundamental question in this area is to understand the conditions under which the above large random graphs contain a connected component of size comparable to the size of the entire network. Ever since the fundamental work of Erd\H{o}s and R\'{e}nyi, one approaches this problem as follows. 
Let us denote by $G_n$ the $n$th element of our sequence of random graphs.
The edge density of $G_n$ is parametrized by $\theta\ge 0$ (for example, as in bond percolation), so that our growing family of random graph families is $((G_n^{\theta},\theta \in [0,1]); n\geq 1)$. The aim is to identify the ``giant component'' phase transition, or more precisely, to find $\theta_c$ such that, if $\theta> \theta_c$ then the largest connected component of $G_n^{\theta}$ is of size $\Theta(n)$ with overwhelming probability as $n\to \infty$, and otherwise if $\theta< \theta_c$ then largest connected component of $G_n^{\theta}$ is of size $o(n)$ with overwhelming probability as $n\to \infty$. 
Understanding the structure of large networks in the near-critical regime, is then naturally related to the problem of understanding the connected components in the parameter window $\theta = \theta_c \pm \eps_n$, for some vanishing model-dependent sequence $(\eps_n)_n$. 
Bollab\'as, Janson and Riordan in \cite{BJR.07} identify the critical threshold $\theta_c$
under rather general hypotheses on the random graph model, however they do not provide any insight into the connected component structure within the (near)-critical window.

For the near-critical Erd\H{o}s-R\'{e}nyi random graph, the latter analysis was carried out already by Aldous in \cite{Aldous.97}.
Aldous' approach was based on the breadth-first walk encoding of the connected components sizes, which can be summarized as follows:
1) the excursions of the breadth-first walk above the past infimum encode useful information about the connected components of the random graph, and 2) it is meaningful to take the limit as $n\to \infty$ in this coupling, which yields the scaling limit of the connected component sizes. This approach, together with the analysis of Aldous and Limic \cite{AL.98}, has proven valuable for understanding the critical window connected component structure for a number of related random graph models. A far from complete list of papers using these ideas is \cite{ Limic.19,Martin:2017,BDW.21,DvdHvLS.17,DvdHvLS.20, ABBG.12, CKG.20, BSW.17, Joseph.14,NP.07}. 
Until now, this ``encoding via a stochastic process'' approach has proved successful for \textit{rank-one} models of random graphs. For this class of models the expected adjacency matrix is approximately a rank-one matrix, or equivalently,  $\p(i\sim j) \approx \psi(i)\psi(j)$ for any pair of vertices $(i,j)$, where  $\psi$ is some model-dependent positive function. 

It is not surprising that many complex (real-world) networks are believed not to be of rank-one. Indeed, a fundamental problem in statistics and computer science is to meaningfully separate data into clusters which share certain important characteristics \cite{Hennig:2016}. When the data is a graph, this clustering involves partitioning of the vertex set $V$ into blocks $C_1,C_2,\dotsm, C_m$ for some $m\ge2$, in such a way that the edge density within blocks is high, and the edge density of links transcending blocks is low (or alternatively, the edge density within blocks is low, while the density of edges transcending blocks is high). 
A well-known random graph model which  exhibits non-trivial community structure (or equivalently, the expected adjacency matrix of higher-rank) is the stochastic block model (SBM). 
The SBM with $m$ blocks is a graph on $m n$ vertices where for each $i\in[m]:=\{1,2,\dotsm,m\}$ there are $n$ vertices of type $i$, and where an edge connects vertex $v$ of type $i$ and vertex $u$ of type $j$ with probability $p_{i,j} = p_{j,i}$, independently over different pairs of vertices. This graph has become an important model for rigorous analysis of network clustering algorithms. We refer an interested reader to the survey of Abbe \cite{Abbe:2017} for more information and precise statements on these theoretical results.

The second and the third author recently identified in \cite{KL.21} a new critical window for the stochastic block model, and carried out the scaling limit analysis akin to that of \cite{AL.98}.
The scaling limit of \cite{KL.21} is the so called {\em interacting multiplicative coalescent}.
The techniques used therein do not include an explicit encoding of SBM via a random walk.

 The main goal of this paper is to provide an excursion representation for SBM. Our encoding is quite general as it extends, under certain additional assumptions, to the so-called \textit{degree-corrected stochastic block model} \cite{Karrer:2011}, which incorporates degree inhomogeneity among the vertices of the same block. 
Our study relies on a  novel (composition alike) operator on (pairs of) real-valued functions on $[0, \infty)$. This construction is natural but somewhat technical, and it is crucial for the scaling limit analysis. The scaling limit for the sizes of connected components of the degree-corrected stochastic block models will be exhibited in a forthcoming work \cite{CKL.24+}.

\section{Model and Results}
\subsection{Graphical Models}
 Denote by
\begin{equation*}
    \cvd = \left\{\bx = (x_1,x_2,\dotsm): x_1\ge x_2\ge\dotsm\ge 0,\quad\sum_{j=1}^\infty x_j^2<\infty\right\}.
\end{equation*}
If $\bw = (w_1,w_2,\dotsm, w_{N},0,0,\dotsm)$ we say that $\bw$ is has finite length, and also that 
$N$ is the length of $\bw$, which we write as $\len(\bw) = N$.
Consider some weight vector ${\bw} = (w_1,w_2,\dotsm)\in \cvd$ of finite length. Given any square summable vector $\bx$ with non-negative entries, we write ${\rm ord}(\bx)$ for the decreasing re-ordering of the entries of $\bx$.

We recall the inhomogeneous multiplicative random graph of \cite{Aldous.97, AL.98}. The graph $\mathcal{G}({\bw},q)$ is a graph on $\len({\bw})$ vertices labeled by $l\in[\len({\bw})]$ where
\begin{equation*}
\p(l\sim r\text{ in }\mathcal{G}({\bw},q)) = 1-\exp(-q w_l w_r). 
\end{equation*}
We interpret the value $w_l$ as the {\em propensity} of the vertex $l$ to form edges. 
It is often called {\em the weight} (or \textit{mass}) of vertex $l$.
A natural coupling of $(\mathcal{G}({\bw},q))_{q\geq 0}$ can be realized in a usual  way (typical for all percolation processes): let the edge between 
$l$ and $r$ appear according to a Poisson process with parameter/rate $w_l w_r$, independently over all $l\neq r$. 
Note that in order to keep track of the connected component structure only at a fixed time $q$, one can equivalently construct the graph $\mathcal{G}({\bw}, q)$ by attaching a Poisson (with mean $qw_lw_r$) number of edges between vertex $l, r$. The original (continuous-time) graph is then obtained from this multi-graph by removing any duplicate edges. 
The Erd\H{o}s-R\'{e}nyi (binomial) random graph $G(n,p)$ is the special case, where ${\bw} = (1,1,\dotsm, 1,0,\dotsm)$ with $\len({\bw}) = n$ and $q = -\log(1-p)$.

The degree-corrected stochastic block model (DCSBM) can be constructed in a similar fashion, see \cite{Karrer:2011}. 
Here we fix $m$ finite length vectors ${\bw}^1,\dotsm, {\bw}^m\in \cvd$, and a symmetric $m\times m$ matrix $Q$ with non-negative real entries. 
Each vertex $v$ is of the form $v = (l,i)$, where $i\in[m]$ is its {\em type} (this means that $v$ is an element of the $i$th block) and $w_l^i$ is its assigned weight corresponding to the {\em propensity} of $v$ to form edges.  Let us denote by ${\bW}$ the vector $({\bw}^1,\dotsm, {\bw}^m)\in (\cvd)^m=\cvd \times \cvd \times \cdots \times \cvd$ listing all the propensities of all the vertices in a type-wise increasing (and propensity-wise non-increasing) ordering.
The random (multi-)graph $\mathcal{G}({\bW}, Q)$ is obtained after attaching $\operatorname{Poi}(Q_{i,j} w_l^i w_r^j)$ many edges between each pair of vertices $(l, i)$ and $(r,j)$, independently over different pairs.
Since we are concerned here with the sizes of connected components, all duplicate edges and all self-loops will be ignored.

The DCSBM has two kinds of parameters. Parameters of the first kind are the weight vectors $\bw^i$, which give rise to the degree inhomogeneity in the graph. The larger the value of $w^i_l$, the more neighbors will the corresponding vertex $(l,i)$ have on the average. This is analogous to the setting of the rank-one graph $\G(\bw,q)$.
Parameters of the second kind are the entries of $Q$, and they determine the block structure of the graph.
 The larger the value of $Q_{i,j}$, the more likely will an edge appear between a vertex of type $i$ and a vertex of type $j$.
 In fact, the matrix $Q$ is a multi-dimensional analogue of time $q$ in the rank-one model.

\begin{remark}
\label{rem:mass_time_scaling} 
One could incorporate the information on the diagonal of $Q$ within the weight data.
More precisely, define $\bW^Q = (\sqrt{Q_{1,1}}\bw^1,\dotsm, \sqrt{Q_{m,m}}\bw^m)$ and
$Q'_{i,j}:=Q_{i,j}/\sqrt{Q_{i,i}Q_{j,j}}$.
In this way $Q'$ has $1$ on the diagonal, and moreover it is easy to see that the open edges (i.e.~connections) in $\G(\bW,Q)$ and in $\G(\bW^Q,Q')$ have the same law. Indeed, in the latter model, the scaling of the block weights cancels out the scaling of the $m$-dimensional time.    
Note, however, that in this coupling (visually) the same connected components of 
$\G(\bW,Q)$ and $\G(\bW^Q,Q')$ have
completely different weights.
\end{remark}

\subsection{First hitting times of fields}

An encoding of the connected components weights of the naturally coupled family of random graphs $(\mathcal{G}({\bw},q); q > 0)$ is due to Limic \cite{Limic.19}. 
This construction relies on $\len(\bw)$ independent exponential random variables $(\xi_l)_{l\in [\len(\bw)]}$, where $\xi_l \sim {\operatorname{Exp}}(w_l)$.
Here and below $\operatorname{Exp}(c)$ denotes an exponential random variable with rate $c$ (mean $\frac{1}{c})$.
If $A$ is a collection of vertices in $\mathcal{G}({\bw},q)$, define the {\em weight of $A$} to be $\sum_{l\in A} w_l$. 
For each $q$, denote by $\sM_q(1)\ge \sM_q(2)\ge \dotsm $ the weights of the connected components of $\mathcal{G}({\bw},q)$, listed in non-increasing order. 

For $q>0$, let $X^q \equiv X^{q,\bw} = (X^q(t);t\ge 0)$ denote the random walk-like process
\begin{equation}\label{eqn:bfw1}
    X^q(t) = -t + \sum_{l=1}^{\len({\bw})} w_l 1_{[\frac{1}{q} \xi_l \le t]}.
\end{equation}
For $y\geq 0$,  let $T^q(y) = \inf\{t: X^q(t-) = -y\} \equiv \inf\{t: X^q(t-) \leq -y\}  $ and denote by $Y_1^q<Y_2^q<\dotsm$ the successive (finitely many) jump times of $T^q$ viewed as a process in $y$.
 A key result of  \cite{Limic.19} is its Proposition 5, which states (in a slightly different language) that
the processes 
 $ ({\rm ord}(T^q(Y_l+)-T^q(Y_l);l\ge 1))_{q>0}$ 
and 
$(\sM_q(1),\sM_q(2),\cdots)_{q>0}$
are identical in law.
The main advantage of this encoding over similarly looking ones in \cite{Aldous.97, AL.98}, and various analogues constructed in the meantime, is that it works on the level of processes.
A different full encoding for connected component sizes of random graphs with (or without) deletion was invented by Martin and R\'ath in \cite{Martin:2017}.

It was observed already in \cite{Aldous.97, AL.98} 
that, for each fixed $q$, the walk-based ordered encoding $(T^q(Y_l+)-T^q(Y_l);l\ge 1)$ is distributed as a size-biased copy of ${\rm ord}(T^q(Y_l+)-T^q(Y_l);l\ge 1)$. Let $(\sM^*(l);l\ge 1)$ be a size-biased reordering of $(\sM(l);l\ge 1)$, where the size of $\sM(l)$ is equal to its weight. \begin{corollary}[see also {\cite[Proposition 1]{Limic.19}}] \label{cor:coro_limic}  
For each $q>0$
\begin{equation*}
    \left(T^q(Y_l+)-T^q(Y_l);l\ge 1 \right) \overset{d}{=} \left(\sM^*(l); l\ge 1 \right).
\end{equation*}
\end{corollary}
\begin{proof}
Since  $ {\rm ord}(T^q(Y_l+)-T^q(Y_l);l\ge 1)$ 
and 
$(\sM_q(1),\sM_q(2),\cdots)$ are equally distributed, the same is true for their respective  size (weight)-biased lists. 
\end{proof} 

We now present a generalization of this representation, based on several ideas in the random tree and branching process literature \cite{Chaumont:2016, Hernandez:2020, Chaumont:2020}. 
The set-up is as follows: recall ${\bW}, Q$ fixed above, and provided that $w_l^i>0$, we let $\xi_l^i$ have $\operatorname{Exp}(w_l^i)$ distribution, where all the variables in the family  $(\xi_l^i)_{i\in [m], l\ge 1, w_l^i>0}$ are independent. When we refer to a vertex $v = (l,i)$ we will often simply write $\xi_v$ in place of $\xi_l^i$.
For all $i, j\in [m]$, let us define 
\begin{equation}
\label{D:matrix_R}
R_{i,j} := Q_{i,j}/Q_{i,i},
\end{equation} and
so that in particular $R_{i,i} \equiv 1$. 
In addition,
 for each $i,j\in [m]$ and all $t\geq 0$ we define
\begin{equation}\label{eqn:XfieldDef1}
    X_{i,j}(t) = \left\{ \begin{array}{ll}
         \displaystyle -t + \sum_{l=1}^{\text{len}({\bw}^j)} w_l^j 1_{[\frac{1}{Q_{j,j}}\xi_l^j \le  t]},&\mbox{if } i= j,  \\
         \displaystyle R_{i,j} \sum_{l=1}^{\text{len}({\bw}^j)} w_l^j 1_{[\frac{1}{Q_{j,j}}\xi_l^j \le  t]},&\mbox{if } i\neq j.
    \end{array}\right.
\end{equation}
The processes $X_{j} = (X_{1,j},\dotsm, X_{m,j})$, $j\in [m]$, clearly depend on both $Q$ and ${\bW}$; however, in the sequel this fact will be mostly suppressed from the notation. 
Observe that the vector-valued processes $X_j$ are independent over $j$. Also observe that for each fixed $j$, the off-diagonal processes depend deterministically on the diagonal $(X_{j,j}(t); \,t\geq 0)$. In particular, for any given $j\in [m]$, all the processes $(X_{i,j};\, i\in[m])$ have simultaneous jumps.

In order to state and prove an analogue of Corollary \ref{cor:coro_limic}, 
we need to define an analogue of the first hitting times process $(T^q(y);\, y\geq 0)$. Such processes were studied recently by Chaumont and Marolleau in 
\cite{Chaumont:2020, Chaumont:2021} in the context of random fields. 
We now recall the setting of \cite{Chaumont:2020, Chaumont:2021}, as well as some of their results which are fundamental for the present study.
Given (deterministic) c\`adl\`ag functions $x_{i,j}$ for $i,j\in[m]$ such that $x_{i,j}(0) = 0$ for all $i, j$, and such that $x_{i,j}$ is non-decreasing when $i\neq j$ and $x_{i,i}(t)-x_{i,i}(t-)\ge 0$ for all $t$ and $i\in [m]$, let us consider the following field
\begin{equation}\label{eqn:discField1}
    {\bbx}(\vec{t}) = {\bbx}(t_1,\dotsm, t_m) = \left( \sum_{j=1}^m x_{1,j}(t_j),\dotsm, \sum_{j=1}^m x_{m,j}(t_j)\right).
\end{equation} 

 It is proved in \cite{Chaumont:2020} that for each $\vec{y}\in \R^m_+$ there exists a unique solution to
\begin{numcases}{}
 x_i(\vec{t}-) =\sum_{j=1}^m x_{i,j}(t_j-) = -y_i,  & \text{$\forall i$  such that $t_i<\infty$}, \label{equ_equation_for_ti}\\
\vec{t} \to {\min}. & \nonumber 
\end{numcases}

Let us denote by 
\begin{equation}
\label{D:Tdeterministic}
\bbT(\bbx;\vec{y})= (T_1(\bbx;\vec{y}),T_2(\bbx;\vec{y}),\dotsm, T_m(\bbx;\vec{y}))\in[0,\infty]^m    
\end{equation}
this unique minimal solution. 
The condition $\vec{t} \to {\rm min}$ means that any other solution $\vec{t'}$ to~\eqref{equ_equation_for_ti} 
must be component-wise greater or equal to 
 $\bbT(\bbx;\vec{y})$, or equivalently that
 $T_j(\bbx;\vec{y})\le t_j'$ for all $j\in [m]$ (which we also write as $\bbT(\bbx;\vec{y})\le \vec{t'}$).

\begin{remark}
 Observe that $\vec{t} = (\infty,\infty,\dotsm,\infty)$ is always a solution to the equation in \eqref{equ_equation_for_ti}. As we will soon see, most of the random fields relevant for our present study will be such that $\bbT(\bbx;\vec{y})$ takes finite values in $\R_+^m$, for all $\vec{y}\in \R_+^m$ almost surely.
\end{remark}

In analogy to the deterministic setting, we now consider the $\R^m$-valued and $\R^m_+$-indexed field $\bbX = \bbX^{{\bW}, Q} = (\bbX(\vec{t}); \vec{t}\in \R^m_+)$, defined by
\begin{equation}\label{eqn:XfieldDef2}
    \bbX(\vec{t}) = \left(\sum_{j=1}^m X_{1,j}(t_j),\dotsm, \sum_{j=1}^m X_{m,j}(t_j) \right).
\end{equation}
By abuse of language we will henceforth refer to vectors $\vec{t}\in \R_+^m$ as ``time'', or less-frequently as ``time-lines''.

\subsection{Random field encoding of DCSBM}

Recall that ${\bW}$ and $Q$ are fixed as above.
Our next goal is to encode the weights of the connected components of the graph $\mathcal{G} = \mathcal{G}({\bW}, Q)$ in terms of its corresponding $\bbX$. 
Let us list the connected components $\cC(1),\cC(2),\dotsm$ 
 of $\mathcal{G}$ in some arbitrary (measurable) fixed way.

Recall that each vertex $v$ in $\mathcal{G}$ is identified with $(l,i)$ for some $l\ge 1$ and $i\in[m]$, where $i$ is the type of $v$, and $l$ is the ranking of $v$'s propensity or weight (specified as $w_l^i$) among all the type $i$ vertices in $\mathcal{G}$.
In forthcoming calculations it will often be convenient to write $\TYPE(v)$ to mean $i$, the type of $v$.
We can therefore define the total weight of type $j$ vertices in the $r$th connected component $\cC(r)$ of $\mathcal{G}$ by
\begin{equation*}
    \sM_j(r) := \sum_{l: (l,j)\in \cC(r)} w_l^j.
\end{equation*}
To keep track of this information we use a family of $m$-dimensional random vectors 
\begin{equation}
\label{D:masses_M}
\vec{\sM}(r) = (\sM_1(r),\dotsm, \sM_m(r)),\quad r\geq 1.
\end{equation}

We will encode the family $(\vec{\sM}(r); r\ge 1)$ via the family of the first hitting times $(\bbT(\bbx;\vec{y});\vec{y}\in L)$ along a {line} $L\subset \R_+^m$. 
More precisely, let us fix a vector $\vec{\rho}\in \R_+^m \setminus \{\vec{0}\}$ and consider the half-line $L = \{\vec{\rho}y; y\ge 0\}$ in the direction of $\vec{\rho}$. 
Define the vector-valued process 
$$\bT = \bT^{\vec{\rho}, {\bW}, Q} = (\bT(y); y\ge 0),$$  
by letting 
\begin{equation}
\label{D:TrhoWQ}
\begin{split}
\bT(y) =\bbT(\bbX;\vec{\rho} y) &= \inf\{\vec{t}: X_{i}(\vec{t}-)= -\rho_i y,\ \  \forall i \text{ s.t. }t_i<\infty\} \\
&= \inf\{\vec{t}: \bbX(\vec{t}-) = -\vec{\rho} y\} .
\end{split}
\end{equation}
\begin{remark}
For a fixed $y$, $\bT(y)$ is analogous to the above deterministic minimizer $\bbT(x;\vec{\rho}y)$, except that here we are (almost) sure that there is a finite random quantity $\vec{S^y}$ such that \sloppy $\sum_{j}X_{i,j}(\vec{S_j^y}-)= -\rho_i y$ for each $i\in [m]$, or equivalently, that 
${T}_i(\bbX;\vec{\rho} y)<\infty$ almost surely for each $i\in [m]$ and all $y\geq 0$. This is due to the fact that $X_{i,j}$ remains bounded for all $i\neq j$, while $X_{i,i}(t)\longrightarrow -\infty$ as $t\to\infty$. 
\end{remark}

It is easy to see that, with probability one, the process $y\mapsto \bT(y)$ is non-decreasing with left-continuous paths. 
By the construction of the minimal solution $\bbT$ in the proof of \cite{Chaumont:2020} Lemma~2.3, one can see that $\bT(y)$ is a (multi-dimensional) stopping time with respect to the filtration generated by $\bbX$. 
Furthermore, there are at most $\sum_{j=1}^m \text{len}({\bw}^j)$ many jumps of the process $\bT$. This is due to the construction \eqref{eqn:XfieldDef1}--\eqref{eqn:XfieldDef2} (in particular, there are $\text{len}({\bw}^j)$ many jumps of $X_{j,j}$, for each $j\in [m]$), joint with the fact that (in our discrete setting, analogously to the $m=1$ setting)
to each jump time $Y$ of $\bT$ corresponds a random index $J\in [m]$, and a unique jump of $X_{J,J}$ (on the $J$th timeline, say at time $S_J(Y)$) such that $X_{J,J}$ starts an excursion at $S_J(Y)$.
See also Remark \ref{R:m_rho_consequences}.
\begin{definition}
\label{def:Delta}
Let $(\vec{\Delta}(r);\,r \geq 1)$ denote the jump sizes of $y\mapsto \bT(y)$ listed in chronological order.
\end{definition}

Recall the matrix $R$ defined in \eqref{D:matrix_R},
and recall that $(\cC(r);r\ge 1)$ is an arbitrary ordering of the connected components of $\G$. Recall that $\TYPE(v)$ is the type of vertex $v$. 
Given a set of vertices $\mathcal{A}$ of $\G$, let us assign to 
$\mathcal{A}$ its $(\vec{\rho},Q)$-\textit{scaled mass}, or \textit{scaled mass} for short,  
as 
\begin{equation}\label{equ_scaled_weight_general}
    \sS(\mathcal{A})\equiv \sS(\mathcal{A};\vec{\rho},Q):=
    \sum_{v \in \mathcal{A}} \rho_{\TYPE(v)} Q_{\TYPE(v),\TYPE(v)} w_v .
\end{equation} 
Then, for each $r$, the scaled mass of $\cC(r)$ simplifies to 
\begin{equation}\label{equ_scaled_weight}
    \sS(\cC(r))\equiv \sS(\cC(r);\vec{\rho},Q):= \sum_{i=1}^m \rho_i Q_{i,i} \sM_i(r),
\end{equation}
where $\vec{\sM}(r)$ is from \eqref{D:masses_M}. It may be more accurate to write $\sS(V(\cC(r)))$ as this quantity depends on the \textit{vertex set} $V(\cC(r))$ for the component $\cC(r)$, but we think this is too cumbersome of notation. 
Let $(\cC^*(r);r\ge 1)$ denote a size-biased reordering of $\{\cC(r):\sS(\cC(r))>0\}$ by their scaled mass. 
For each $r$, we denote by $\vec{\sM}^*(r)$ the corresponding weight vector of the component $\cC^*(r)$. The following is our first main result.
\begin{theorem}\label{thm:discreteEncoding}
Let $\bW$, $Q$  be fixed as above. Then, for each  $\vec{\rho}\in\R_+^m \setminus\{\vec{0}\}$, we have the identity
\begin{equation}\label{eqn:THM1.5}
    \left(\vec{\Delta}(r);\, r \ge 1 \right) \overset{d}{=} \left(R \vec{\sM}^*(r) ;\, r\ge 1 \right).
\end{equation}
\end{theorem}
\noindent
Its proof is postponed until Section \ref{sec:BFW_discrete_proofs}.

Let us define 
\begin{equation}
\label{D:m_rho_restricted}
[m]_\rho:=\{k\in [m]:\rho_k>0\}.
\end{equation} Given a set $\mathcal A$ of vertices, and a set $I\subset[m]$ of indices we abuse notation and write
\begin{equation}
\label{E:vertices_intersect_indices}
    \mathcal A\cap I \equiv  \mathcal A\cap (\mathbb{N} \times I)= \{v\in\mathcal A: \TYPE(v) \in I\}.
\end{equation}

\begin{remark}
\label{R:m_rho_consequences}
On the event
$\{\cC(l) \cap [m]_{\rho}=\emptyset\}$, neither $\cC(l)$ nor its corresponding vector $\vec{\sM}(l)$ appear in the size-biased list above. 
The encoding via field $\mathbb{X}$ cannot access any such $\cC(l)$ (since the exploration is done only in the direction of $\vec{\rho}$), and therefore the scaled mass of $\cC(l)$ will not appear in the list 
 on the RHS of \eqref{eqn:THM1.5}. 
 Concerning the list on the LHS of \eqref{eqn:THM1.5}, we wish to point out that. 
 as the proof of Theorem \ref{thm:discreteEncoding} will show, the random process $\bT(y)$ can be written as 
\begin{equation*}
\vec{\rho} y + \sum_{x<y}(\bT(x+)-\bT(x))  = \vec{\rho} y + \sum_{r\in J_y} \vec{\Delta}(r),\qquad y\ge 0,
\end{equation*} for some finite and uniformly bounded set of jumps $J_y$.
Therefore,  $T_j(+\infty):=\lim_{y\to\infty} T_j(y)$
and it is almost surely finite, if and only if, $\rho_j = 0$. 
If $\rho_j = 0$, then the time $T_j(+\infty)$ may (and typically does) appear \textit{before} 
some of the excursions (above past infimum) of the process $X_{j,j}$ even begin. The information contained in $X_{j,j}$ on $[T_j(+\infty),+\infty)$, for all $j\in [m]\setminus [m]_{\vec{\rho}}$, could probably be used to reconstruct the connected component sizes of the DCSBM intersected with $[m]\setminus[m]_{\rho}$, however it is not clear if this extra effort would bring any significant benefits. We will exhibit an encoding of the connected component sizes in each probe direction $\vec{\rho}$. By varying $\vec{\rho}$, one can access all the connected components of the DCSBM.

\end{remark}

\subsection{From fields to processes}

Theorem \ref{thm:discreteEncoding} is a random field generalization of Corollary \ref{cor:coro_limic}. Let us first consider a restatement of Corollary \ref{cor:coro_limic} in terms of the excursions of the process $X^q$.

Recall that if $f:[0,\infty)\to \R$ is a c\`adl\`ag function, an interval $(l,r)$ is called an \textit{excursion (above past infima)} interval if
\begin{equation*}
    \inf_{t\le l} f(t) = \inf_{t\le r} f(t) \qquad\text{and}\qquad f(s-)> \inf_{t\le r} f(t) \text{ for all }s\in(l,r).
\end{equation*}
For a function $f$, we will denote by $\cE(f)$ the collection of excursions (above past infima), and by $\cL(f)$ the multiset of excursion lengths $\{r-l: (l,r)\in \cE(f)\}$. 
Finally, we let
\begin{center}
$\cL^{\downarrow}(f)$ be the non-increasing rearrangement of $\cL(f)$,
\end{center}
 provided that it is well-defined (if and only if there are at most finitely many excursions of $f$ longer than $\delta$, for each $\delta>0$).

As discussed in the introduction, the pioneering work of Aldous \cite{Aldous.97} was a base to a number of studies.
With a representation analogous to Corollary \ref{cor:coro_limic} and a corresponding scaling limit for the sequence of appropriately rescaled process $(X^q)_q$ one can often appeal to a quite general theory \cite[Lemma 7 and Proposition 15]{Aldous.97} to conclude relatively easily that the rescaled component weights (of the random graph under consideration) converge in distribution to a random element of $\cvd$ as the size of the graph diverges.
This general approach by Aldous \cite{Aldous.97} is based on relating the excursions of the prelimiting processes with excursions of the limiting stochastic process. 
While the construction of the first hitting times $\bT(y)$ in \cite{Chaumont:2020} is a very useful tool for our analysis, it does not give much insight into the behaviour of the field ``between'' $\bT(y)$ and $\bT(y+)$.
So it is not clear
what a reasonable definition of an excursion would even be in the present context.

A major contribution of this paper is a construction of a single curve $\vec{\gamma}:\R_+\to \R_+^m$ which \textit{combines the information on the $m$ time-lines in an appropriate way} so that the first hitting times of the $\R_+^m$-indexed and $\R^m$-valued field match the first hitting times of a $\R_+$-indexed real-valued process (as it turns out, there are several such processes). 
A precise statement is the following theorem. (We prove this result under weaker assumptions, which are cumbersome to state at this point.)
\begin{theorem}\label{thm:gammaExistence}
    Suppose that $\bbx$ is a field as in \eqref{eqn:discField1}, $\inf_{s\le t} x_{i,i}(s)<0$ for all $t>0$ and $\liminf_t x_{i,i}(t) = -\infty$ and that there exists some vector $\vec \rho=(\rho_1,\ldots,\rho_m)\in(0,\infty)^m$ such that
for each $l\in[m]$ and all $i,j \not= l$
\begin{equation}
\label{equ_symmetry_assumption}
\frac{ x_{i,l}(t) }{ \rho_i }=\frac{ x_{j,l}(t) }{ \rho_j } \text{ for all } t\geq 0.
\end{equation}
Let $\bT(y) = \bbT(\bbx;\vec{\rho}y)$ be the first hitting time of level $-\vec{\rho}y$ for $\bbx$. 
Then, there exists a Lipschitz curve $\vec{\gamma}:\R_+\to \R^m_+$ with non-decreasing coordinates such that
\begin{enumerate}
    \item[(1)] $\|\vec{\gamma}(s)\|_1 = s$ for every $s\ge 0$.
\end{enumerate}
Moreover, for all $y\ge 0$, if $\bT(y)\in \R_+^m$ then 
\begin{enumerate}
    \item[(2)] $
        \vec{\gamma}\left(\|\bT(y)\|_1\right) = \bT(y)$ and
    \item[(3)] for all $i\in[m]$
    \begin{equation*}
        \inf\left\{s\ge 0: \sum_{j=1}^m x_{i,j}\circ\gamma_j(s-) = -\rho_i y\right\} = \|\bT(y)\|_1.
    \end{equation*}
\end{enumerate}
\end{theorem}

An immediate consequence of Theorems \ref{thm:discreteEncoding} and \ref{thm:gammaExistence} is the following result. 
\begin{theorem}\label{thm:gammaAndField}
    Let $\bW$, $Q$, $R$ and $\bbX$ be as in Theorem \ref{thm:discreteEncoding}, and let $\G(\bW,Q)$ be the corresponding DCSBM. Suppose that, in addition, 
    \begin{equation}
    \label{equ_symmetry_stoch}
    R_{i,j} = \frac{Q_{i,j}}{Q_{j,j}} = \rho_i \nu_j, \ \text{ for all } i \neq j,
    \end{equation}
    for some vectors $\vec{\rho}, \vec{\nu}\in(0,\infty)^m$. 
    Recall that $(\vec{\sM}(r);r\ge 1)$ are the vector-valued component weights of $\G(\bW,Q)$ arranged in some arbitrary order. Then there exists a vector valued curve $\vec{\gamma}$ with non-decreasing coordinates such that
    \begin{enumerate}
        \item The ordered excursion lengths of $\sum_{j=1}^m X_{i,j} \circ \gamma_j(t)$ are equal in law to the reording of $(\|R\vec{\sM}(r)\|_1 ;r\ge 1)$; i.e.
        \begin{equation*}
            \cL^{\downarrow}\left( \sum_{j=1}^m X_{i,j}\circ \gamma_j\right) \overset{d}{=} {\rm ORD}\left(\|R\vec{\sM}(r)\|_1;r\ge 1\right).
        \end{equation*}
        \item If $((l_p,r_p);p\ge 1)$ are the excursion intervals of $\sum_{j=1}^m X_{i,j}\circ\gamma_j$ arranged chronologically, then 
        \begin{equation*}
            \left(\vec{\gamma}(r_p)-\vec{\gamma}(l_p); p\ge 1\right)\overset{d}{=} \left(R \vec{\sM}^*(p);p\ge 1\right).
        \end{equation*}
    \end{enumerate}
\end{theorem}

\subsection{Comments on models with condition \texorpdfstring{\eqref{equ_symmetry_stoch}}{(2.16)}}
\label{S:commentsONmodels}
\subsubsection{Restrictions with few blocks}\label{sec:fewblocks}

Let us begin by noting that whenever there are two blocks (i.e. $m = 2$) the condition \eqref{equ_symmetry_stoch} is always true provided that $Q_{i,j}>0$ for all $i,j\in[2]$. In fact, \eqref{equ_symmetry_assumption} in Theorem \ref{thm:gammaExistence} is always true in the case where there are just two types.

Assumption \eqref{equ_symmetry_stoch} starts to become more interesting in the case where $m = 3$. By first examining \eqref{equ_symmetry_assumption} in Theorem \ref{thm:gammaExistence} as well as the form of the field $\bbX$ in \eqref{eqn:XfieldDef1}, we see that for any $Q$ we can set
\begin{align*}
    \rho_{1} &= \frac{Q_{1,2}Q_{1,3}}{Q_{1,1}}&
    \rho_2 &= \frac{Q_{2,1}Q_{2,3}}{Q_{2,2}}&
    \rho_3 &=\frac{Q_{3,1}Q_{3,2}}{Q_{3,3}}
\end{align*}
and
\begin{align*}
    \nu_1&= \frac{1}{Q_{2,3}} & \nu_2 &= \frac{1}{Q_{1,3}} & \nu_3 &= \frac{1}{Q_{1,2}}.
\end{align*}
Indeed, looking at distinct $i,j,k\in[3]$ we have 
\begin{equation*}
\rho_i\nu_j = \frac{Q_{i,j}Q_{i,k}}{Q_{i,i}}\cdot\frac{1}{Q_{i,k}} = \frac{Q_{i,j}}{Q_{i,i}} = R_{i,j}.
\end{equation*}
In particular, provided that $Q$ is a symmetric matrix with strictly positive entries Theorem \ref{thm:gammaAndField} is always applicable for particular (and explicit) choices of $\vec{\rho}$ and $\vec{\nu}$.

A simple dimension counting argument implies that  \eqref{equ_symmetry_stoch} can not be satisfied in great generality for $m>3$. Indeed, the collection of symmetric $m\times m$ matrices $Q$ with positive entries forms a $\binom{m+1}{2}$ dimensional manifold, while the collection of matrices $Q$ that satisfy \eqref{equ_symmetry_stoch} is only of dimension $3m$ ($m$ for the diagonal entries of $Q$ and $m$ for each the vectors $\vec{\rho},\vec{\nu}$).

   \subsubsection{Link with \texorpdfstring{\cite{KL.21}}{[32]}} Condition \eqref{equ_symmetry_assumption} is equivalent to \eqref{equ_symmetry_stoch} in our stochastic setting, and we furthermore have an interesting probabilistic interpretation.

\begin{lemma}
The symmetric matrix $Q$ and the vector $\vec{\rho}\in(0,\infty)^m$ satisfy \eqref{equ_symmetry_stoch} if and only if there exist $q_i>0$, $i\in[m]\cup\{0\}$, such that vertices $(l,i)$ and $(k,j)$ of $\G(\bW,Q)$ are connected by an edge with probability 
$1-e^{-q_0 \rho_i w_l^i\rho_j w_k^j}$ if $i\neq j$  and with probability $1-e^{-q_i \rho_i w_l^i\rho_j w_k^j}$ if $i=j$.
\end{lemma}
\begin{proof}
    Since $Q$ is a symmetric matrix, due to  \eqref{equ_symmetry_stoch} we have for all $i\neq j$
    $$
    Q_{j,j} \rho_i \nu_j = Q_{i,i} \rho_j \nu_i 
    \Longleftrightarrow
    \frac{Q_{j,j}\nu_j}{\rho_j} = \frac{Q_{i,i}\nu_i}{\rho_i}.
    $$
    Let $q_0:= Q_{1,1}\nu_1/\rho_1= Q_{j,j}\nu_j/\rho_j$, $j\in [m]$, and $q_i:=Q_{i,i}/\rho_i^2$, $j\in[m]$. 
    It is straight-forward to see that $Q_{i,j} = q_0 \rho_i \rho_j$ for $i\neq j$ and $Q_{i,i}=q_i\rho_i^2$.
\end{proof}

The above lemma gives a connection with the setting of \cite{KL.21}.
The class of models satisfying \eqref{equ_symmetry_stoch} includes the SBMs analyzed in \cite{KL.21}, but it is much larger since it allows for varying intra-block connection probabilities over types. This connection will be important in our sequel paper \cite{CKL.24+}.

\subsubsection{An epidemiological interpretation}

Let us now describe possible epidemiological interpretation. We have a population of $N = \sum_{j=1}^m \len(\bw^{j})$ many individuals segmented into $m$ many sub-types. Each individual $(l,j)$ of type $j$ has some propensity $w_l^j$ of both catching or transmitting a disease to their neighbors. The factor $\nu_j$ represents the propensity of a type $j$ individual to transmit the disease to others, for example by not taking preventative measures to stop the spread of the disease. Finally, there is some likelihood that type $j$ individuals come into contact with type $i$ individuals which is represented by
\begin{equation*}
    \frac{\tilde{\rho}_i\tilde{\rho}_j}{\tilde{\rho}_i+\tilde{\rho}_j}\qquad\textup{where}\qquad \tilde{\rho}_{i} = Q_{i,i} \rho_i.
\end{equation*}

To model a disease spreading through the population we can use a direct graph where a direct edge from $u$ to $v$ means individual $u$ infected individual $v$. 
Moreover our graphs is built by independently adding an edge from $(l,i)$ to $(r,j)$ with probability
\[
1-e^{-w_{l}^{i}w_{r}^{j}\frac{\tilde{\rho}_{i}\tilde{\rho}_{j}}{\tilde{\rho}_{i}+\tilde{\rho}_{j}}\nu_{j}},\qquad \forall i\neq j.
\]

Forgetting direction of the edges, we see that an (undirected) edge between $(i,l)$ and $(j,r)$ appears with the probability $1-e^{-p},$
where
\begin{align*}
p & =w_{l}^{i}w_{r}^{j}\left(\frac{\tilde{\rho}_{i}\tilde{\rho}_{j}}{\tilde{\rho}_{i}+\tilde{\rho}_{j}}\nu_{j}+\frac{\tilde{\rho}_{j}\tilde{\rho}_{i}}{\tilde{\rho}_{i}+\tilde{\rho}_{j}}\nu_{i}\right)\\
 & =w_{l}^{i}w_{r}^{j}\left(\frac{Q_{ij}\tilde{\rho}_{j}}{\tilde{\rho}_{i}+\tilde{\rho}_{j}}+\frac{Q_{ji}\tilde{\rho}_{i}}{\tilde{\rho}_{i}+\tilde{\rho}_{j}}\right)=w_{l}^{i}w_{r}^{j}Q_{ij}\qquad \forall i\neq j.
\end{align*} Thus the (weakly) connected components in this disease model are equal in law to the connected components that Theorem \ref{thm:gammaAndField} can analyze.

\section{Discussion}

\subsection{Past and related work}
To the best of our knowledge, the ``stochastic process encoding'' for analyzing the connected components of critical random graphs which are not rank-1 appeared until now only a few times in the literature. The first such work  is by Dembo, Levit and Vadlamani \cite{DLV.19} on the so-called quantum Erd\H{o}s-R\'{e}nyi (QER) random graph. In this model, each vertex in the standard Erd\H{o}s-R\'{e}nyi graph is replaced by a copy of a circle $S^1$ cut into arcs according to a Poisson process, and these arcs then become the vertices of the QER random graph. Edges are included subsequently according to another independent Poisson process.

A more closely related model to ours appears in the works of Federico \cite{Federico.19} and Wang \cite{Wang.23} on the near-critical bipartite Erd\H{o}s-R\'{e}nyi random graph. While the actual explorations used in these papers differ from ours, their encodings correspond to a join of two separate explorations (one explores the left-vertex set, and the other the right-vertex set of the bipartite graph) into a single stochastic process, which can be analyzed via weak convergence techniques. As mentioned above in Section \ref{sec:fewblocks}, our encoding is completely general in the rank-2 case, whenever $Q_{i,j}>0$ for all $i,j$. Therefore our encoding does not technically cover the bipartite case where $Q_{i,i} = 0$ for both $i$ and so we cannot encode the graphs studied by \cite{Federico.19,Wang.23}. However, by taking the intra-block connection probabilities sufficiently small and using the result of Janson \cite[Corollary 2.12]{Janson.10}, one can see that the bipartite ER graph is asymptotically equivalent to a model with $Q_{i,i}>0$ and therefore one we can encode. See also \cite[Section 6.7]{vanderHofstad.17}. This approach is taken by DC in \cite{Clancy.24+} to analyze the general rank-2 multiplicative random graphs.

A different approach has been quite successful for analyzing other classes of non-rank-1 random graphs. A general method for proving that the connected components of certain critical random graphs,  viewed \textit{as metric measure spaces}, lie in the \textit{basin of attraction} of the continuum limit of critical Erd\H{o}s-R\'{e}nyi random graphs of Addario-Berry et al.~\cite{ABBG.10,ABBG.12}, was developed by Bhamidi et al.~in \cite{BBSW.14}. Roughly speaking, this method consists in showing that the barely subcritical random graph satisfies certain asymptotic properties (this gives the ``blobs'' of \cite{BBSW.14}), and that the evolution of the model from the barely subcritical  to the critical regime is approximately that of the Aldous standard multiplicative coalescent \cite{Aldous.97} (giving the ``blob-level superstructure'' of \cite{BBSW.14}) and converge to the continuum random graph \cite{BSW.17}. This program has more recently enabled Blanc-Renaudie et al.~\cite{Blanc_Renaudie:2024} (resp.~Bhamidi et al.~\cite{BBBSW.23}) to prove that the connected components of the near-critical percolation on the $d$-dimensional hypercube (resp.~on a graph converging to an $L^3$-graphon) converge to the continuum random graph of \cite{ABBG.10,ABBG.12}. It is not likely that this approach would apply in our setting, which is more closely related to the restricted multiplicative merging and the interacting eternal multiplicative coalescents of \cite{KL.21}, than to the Aldous standard multiplicative coalescent. 

In addition to the aforementioned papers, several works used exploration processes and their related height processes (constructed by Duquesne and Le Gall in \cite{DL.02}) for analyzing scaling limits of multi-type Galton-Watson trees. In \cite{Miermont.08}, Miermont introduces a ``reduction of types'' argument to show that (modulo some scaling) the height process of a critical multitype Galton-Watson forest with finite variance converges to a reflected Brownian motion (which also encodes the limit for a single type Galton-Watson forest). This result was generalized in the case of offspring distributions in the domain of attraction of an $\alpha$-stable random variable by Berzunza \cite{BerzunzaOjeda.18}, and in the case of infinitely many types by de Raph\'elis \cite{deRaphelis.17}.

\subsection{Future work}

As mentioned in the Introduction, in this report we initiate our study of the degree corrected stochastic block model, which is continued in our work in progress \cite{CKL.24+}. 
In this paper we lay out the encoding of the graph model via a random field, and (under additional assumptions) develop a technique for transforming the field encoding  into an encoding by a conventional real-valued stochastic process.

Our subsequent work \cite{CKL.24+} is concerned with scaling limits. 
Define $\sigma_r(\bx) := \sum_{p=1}^\infty x_p^r$.
More precisely, we study the behaviour of the sequence of graphs $\G(\bW^{(n)},Q^{(n)})$, as $n\to \infty$, under the following asymptotic conditions on $\bW = \bW^{(n)}$ and $Q = Q^{(n)}$: there exist sequences $0<a_n\to 0$, $\vec{\rho}^{(n)}\to \vec{\rho}\in(0,\infty)^m,\vec{\nu}^{(n)} \to \vec{\nu}\in(0,\infty)^m$, and for each $i\in[m]$ there exist  $\alpha_i>0,\beta_i\ge 0, \alpha_i, \beta_i,\lambda_i\in\R$, and $\bc_i = (c_{i,1},c_{i,2},\dotsm)\in \ell^3_\downarrow$, such that for each $i\in[m]$ and $j\neq i$
\begin{align}\label{eqn:asymptLimitsforWeights}
    &\frac{w_p^{(n),i}}{\sigma_2(\bw^{(n),i})} \to c_{i,p}, & &\text{ and } \ \sigma_2(\bw^{(n),i})\to0,\\ \label{eqn:asymptLimitsForWeights2} &\frac{\sigma_3(\bw^{(n),i})}{\sigma_2(\bw^{(n),i})^3} \to \beta_i + \sigma_3(\bc_i), & &\text{ where }\ \beta_i>0\ \textup{ or }\ \sigma_2(\bc^i) = +\infty,\\
  \label{eqn:asymptLimitsForWeights3}  &\frac{\sigma_2(\bw^{(n),i})}{a_n}\to \alpha_i,  &&\text{ and } \ \frac{R_{i,j}^{(n)}}{a_n} = \frac{Q_{i,j}^{(n)}}{a_n Q_{j,j}^{(n)}} = \frac{\rho_i^{(n)} \nu_j^{(n)}}{a_n},\\
  \label{eqn:asymptLimitsForWeights4}  &\text{and in addition }\ Q_{i,i}^{(n)} = \frac{1}{\sigma_2(\bw^{(n),i})}+ \lambda_{i}+o(1).
\end{align}
 Hypotheses \eqref{eqn:asymptLimitsforWeights}-\eqref{eqn:asymptLimitsForWeights2} and \eqref{eqn:asymptLimitsForWeights4} are the well-known conditions arising from \cite{AL.98}. Informally, the left-hand side of \eqref{eqn:asymptLimitsForWeights3} guarantees that the weights of all the type $i$ vertices are roughly of the same order, while the right-hand side is a technical condition which allows us to apply the results obtained in Sections \ref{sec:excursions_along_a_curve} and \ref{sec:generalexcursionrep} of the present work. One can also check that these assumptions are the natural inhomogeneous generalizations of \cite{KL.21}.

From hypotheses \eqref{eqn:asymptLimitsforWeights}--\eqref{eqn:asymptLimitsForWeights4} 
(without using the RHS in \eqref{eqn:asymptLimitsForWeights3})
it is not hard (applying results from \cite{AL.98,Limic.19}) to derive the scaling limit for each of the processes $X_{k,i}^{(n)}$ in \eqref{eqn:XfieldDef1},
and moreover the joint scaling limit for the matrix-valued process $(X_{i,j})_{i,j\in[m]}$.
In addition, we already obtained an encoding of the finite graph, so one might think that the scaling limit theorem for our DCSBM model would be a standard extension of \cite{Aldous.97, AL.98}.
However, we face several non-trivial technical obstacles when passing from the random fields $\bbX^{(n)}$ to the stochastic processes $\sum_{i=1}^m X_{k,i}^{(n)}\circ\gamma^{(n)}_i(t)$, $k\in [m]$. Without going into details, here we give an indication of our approach developed in \cite{CKL.24+}.
\begin{itemize}
    \item Given a deterministic 
 sequence of fields $\bbx^{(n)}=(x_{i,j}^{(n)})_{i,j\in[m]}$, $n\ge 1$, 
  converging to the field
 $\bbx$, where each  $\bbx^{(n)}$ satisfies the hypotheses of  Theorem \ref{thm:gammaExistence}, we establish convergence, under appropriate assumptions, for the sequence of curves $\vec{\gamma}^{(n)}$ (constructed via Theorem \ref{thm:gammaExistence}).
    \item Since there is no total order on the space of $\R_+^m$-valued vectors, we construct a Polish space $\widetilde{\ell^{2,m}}$, which is an analog of $\ell^2_\downarrow$ for sequences of vectors with square summable norms. This construction is analogous to the construction of the space of graphons \cite{Chatterjee.17}.
    \item We establish tightness in $\widetilde{\ell^{2,m}}$ for the sequence of vector-valued connected component weights of $\G(\bW^{(n)},Q^{(n)})$, $n\geq 1$. 
    \item We improve on the work of Dhara et al.~\cite{DvdHvLS.20}, by weakening  the conditions on the limit of encoding processes guaranteeing the convergence of the corresponding sequence of excursion length vectors. 
    
    \item 
    We show that the paths $t\mapsto \sum_{i=1}^m X_{k,i}^{(\infty)}\circ\gamma_i^{(\infty)}(t)$, where $X^{(\infty)}_{k,i}$ is the scaling limit of $X_{k,i}^{(n)}$, for all $k,i\in[m]$, satisfy the above conditions.
\end{itemize}

In Section \ref{sec:BFWsection}, we introduce two explorations of the DCSBM. The first Exploration \ref{explore:field} is via the random field $X_{i,j}$ from \eqref{eqn:XfieldDef1}, and the second one uses the random graph defining data. We prove in Proposition \ref{prop:Exprole} that these two explorations are equivalent, and soon after we prove Theorem \ref{thm:discreteEncoding}.

In Section \ref{sec:excursions_along_a_curve}, we recall and expand on various elementary properties of the multi-dimensional first hitting times of \cite{Chaumont:2020} for deterministic fields $\bbx$. We then analyze these first hitting times whenever the field $\bbx$ is sufficiently smooth, and prove Theorem \ref{thm:gammaExistence} in this special case. 
This analysis relies on several properties of homeomorphic inverses which fail in general.

In order to overcome this difficulty, in Section \ref{S:non_decreasing_inver_comp} we are lead to introducing a novel composition-like operator $\tcirc$, which has a remarkably good behaviour with respect to taking generalized inverses. For example, if $g$ is an unbounded non-decreasing rcll function with left limits, and if $g$ is strictly increasing at $0$ in addition, then
$g^{-1}\tcirc g = g \tcirc g^{-1} =\operatorname{id}$, where
$g^{-1}$ is the right-continuous generalized inverse of $g$.
It is easy to see that the above identity fails in general (see also examples in Section \ref{S:non_decreasing_inver_comp})if $\tcirc$ is replaced by $\circ$.

 Theorem \ref{thm:gammaExistence} is proved in full generality  in Section \ref{sub:construction_of_a__special_curve}.

\section{Breadth-first walk}\label{sec:BFWsection}
Fix ${\bW} = ({\bw}^1,{\bw}^2,\dotsm, {\bw}^m)$, a collection of $m$ finite length vectors 
${\bw}^i\in \ell^2_\downarrow$. Furthermore fix an $m\times m$ symmetric matrix $Q = (Q_{i,j};\, i,j\in[m])$  with strictly positive entries along the diagonal and non-negative off-diagonal entries. 
All the processes considered in this section  will depend on ${\bW}$ and $Q$, but this will be mostly suppressed from the notation. 

Let 
\begin{equation}\label{eqn:Rvec}
    \vec{R}_j = (R_{1,j},\dotsm, R_{m,j})^T\in \R_+^m,
\end{equation} so that the $i^\text{th}$ coordinate of $\vec{R}_j$ is $R_{i,j}$.

\subsection{Two Explorations} 

The exploration we construct will involve $m$ time-lines corresponding to $m$ different coordinates (or types). 

We will use the stochastic processes $X_{i,j}$ in \eqref{eqn:XfieldDef1} to construct the exploration.
We keep track of two sequences of sets $(\mathcal{U}_k)_k$, $(\mathcal{U}_k^*)_k$ of {\em unexplored} vertices.
The initial set $\mathcal{U}_0$ consists of all the vertices $\mathcal{U}_0  = \{(l,i):l\le \len(\bw^i),i\in[m]\}$. 
At step $k$ we maintain two stacks $\mathcal{A}_k$, $\mathcal{A}^*_k$ of {\em active} vertices.
The initial stack $\mathcal{A}_0 = ()$ is empty. 
In addition we keep track of the set of {\em dead} vertices, which we denote by $\mathcal{D}_k$ in step $k$.
A given vertex needs to be active in order to become dead in a later step, so $\mathcal{D}_0 = ()$ is empty as well.
Recall that we write $[m]_{\vec{\rho}} = \{i\in[m]:\rho_i>0\}$, and at step $k$ denote by $\mathcal{U}^{\vec{\rho}}_k$ all the unexplored vertices $\mathcal{U}_k\cap \N\times [m]_{\vec{\rho}}$
with strictly positive $(Q,\vec{\rho})$-scaled mass. 

We will inductively construct two sequences of $m$-dimensional stopping times \sloppy $\vec{S}^{\ast}_k= (S_{k;1}^{\ast},S_{k;2}^{\ast},\dotsm, S_{k;m}^{\ast})$ where $\ast\in \{L,R\}$, with respect to the filtration $\mathscr{F}(\vec{t})$ generated by $\bbX$. We set $\vec{S}^R_0 = \vec{0}$.  Let $\zeta_k$ denote the number of connected components that have been discovered up-to step $k$. Note that $\zeta_0 = 0$. In the following exploration algorithm, we itemize the steps in order to facilitate the understanding of the subsequent proof.
\begin{exploration}
[Field Exploration]\label{explore:field}
Let $k=1$. 
\begin{enumerate}
\item[(\namedlabel{X0}{X0})] \textbf{While} either $\mathcal{A}_{k-1}\neq ()$ or $\mathcal{U}^{\vec{\rho}}_{k-1} \neq \emptyset$ do as follows:
    \item [(\namedlabel{X1}{X1})] \textbf{Orientation:}\begin{enumerate}
        \item On $\{\mathcal{A}_{k-1} = ()\} \cap \{\mathcal{U}^{\vec{\rho}}_{k-1} \neq \emptyset\}$ our exploration recorded all the vertices of the initial $\zeta_{k-1}$ components intersecting $[m]_{\vec{\rho}}$, and there are still some unexplored vertices left in $V(\G)\cap [m]_{\vec{\rho}}$.
        The algorithm increments $\zeta_{k} = \zeta_{k-1}+1$, and defines 
        $$
        f(l,i;k):= \frac{\xi_{l}^i}{\rho_i Q_{i,i}} -\frac{S_{k-1;i}^R}{\rho_i},\quad (l,i) \in \mathcal{U}^{\vec{\rho}}_{k-1},
        $$
        and
    \begin{equation*}
        \varpi(k) = (l_k,i_k) = \operatornamewithlimits{argmin}_{(l,i) \in \mathcal{U}_{k-1}}  f(l,i;k), \qquad Y_{\zeta_k} = \min_{(l,i) \in \,\mathcal{U}_{k-1}} f(l,i;k).
    \end{equation*}
     Set $\mathcal{A}_{k-1}^* = (\varpi(k))$ and $\mathcal{U}^*_{k-1} = \mathcal{U}_{k-1} \setminus\{\varpi(k)\}$. 
     We call $\varpi(k)$ the \textit{root} of the $\zeta_k$th component. For each $i\in [m]$, set $S_{k;i}^L = S_{k-1;i}^R + \rho_i Y_{\zeta_k}$, and let
     $\vec{S}_{k}^R = \vec{S}_k^L + w_{\varpi(k)}\vec{R}_{i_k}$, where here and below $i_k$ is the above defined type of $\varpi(k)$.  
     We define $N_k = k$ in this case (the reasons for this will be clear soon).
     \item 
     Otherwise on $\{\mathcal{A}_{k-1} \neq ()\}$
      we have (by induction) \begin{equation*}\mathcal{A}_{k-1} = \Big(\varpi(k),\dotsm, \varpi(N_k)\Big),
     \end{equation*} for some $N_k\geq k$. Set $\mathcal{A}^*_{k-1} := \mathcal{A}_{k-1}$, $\mathcal{U}_{k-1}^* = \mathcal{U}_{k-1}$, and $\zeta_k = \zeta_{k-1}$. 
    \end{enumerate}
Note that $N_k-k+1$ equals the length of the stack $\mathcal{A}^*_{k-1}$ almost surely.
\item [(\namedlabel{X2}{X2})]\textbf{The (unexplored) neighbors of $\varpi(k)$}: \begin{enumerate}
    \item The (newly discovered) neighbors of $\varpi(k)$ are the vertices 
\begin{equation*}
   \mathcal{B}_k:= \left\{(l,i)\in \mathcal{U}_{k-1}^*: \frac{\xi_l^i}{Q_{i,i}} \in[S_{k;i}^L, S^R_{k;i}) .\right\}
\end{equation*} Let $\chi(k) = \#\mathcal{B}_k$ be the cardinality of $\mathcal{B}_k$.
\item On $\{\chi(k)=0\}$ the algorithm jumps to \eqref{X3}.
\item Otherwise on  $\{\chi(k)>0\}$, we use the jump times $\xi_l^i/Q_{i,i}\in [S_{k;i}^L, S_{k,i}^R)$ to order the elements of $\mathcal{B}_k$ as follows: Set
\begin{equation*}
    \varpi(N_k+1),\dotsm, \varpi(N_k+\chi(k)),
\end{equation*}
in the almost surely unique way such that 
$i\mapsto \TYPE(\varpi(N_k+i))$ is non-decreasing on $\{1,\ldots,\chi(k)\}$, and such that
ties are broken according to the rule
 $\{\TYPE(\varpi(N_k+i)) = \TYPE(\varpi(N_{k}+j)), i<j\}\subset \{\xi_{\varpi(N_k+i)}< \xi_{\varpi(N_k+j)} \}$, almost surely.
\item For $j = 1,2,\dotsm \chi(k)$ write (temporarily) $\TYPE_j = \TYPE(\varpi(N_k+j))$, and define
\begin{align*}
    \vec{S}_{N_k+j}^L &= \vec{S}_{N_k+j-1}^R & &\textup{and}&\vec{S}_{N_k+j}^R&= \vec{S}_{N_k+j}^L + w_{\varpi(N_k+j)} \vec{R}_{\TYPE_j}.
\end{align*}
\end{enumerate}
\item [(\namedlabel{X3}{X3})] \textbf{Update Sets and Stacks}: Set
   $ M_k := N_k + \chi(k),$
and 
\begin{align*}
\text{on}\  \{M_k>k\} \ \text{define} &\\
\mathcal{A}_k &:= \Big(\varpi(k+1),\varpi(k+2),\dotsm, \varpi(N_k + \chi(k))\Big), \text{ while }
\\
\text{ on}\  \{M_k=k \} \ \text{define } &\mathcal{A}_k := \emptyset ,&
\\
    \mathcal{U}_k&:= \mathcal{U}_{k-1}^*\setminus \mathcal{B}_k ,  
    \\
    \mathcal{D}_k&:= \mathcal{D}_{k-1}\cup\{\varpi(k)\}.
\end{align*}
Increment $k$ by 1 and go to step \eqref{X0}.

\end{enumerate}
\end{exploration}

We now describe another exploration algorithm, which is directly linked to the DCSBM graph. We use similar notation on purpose, except now every character will have an additional ``widetilde'' mark in the superscript. For example, we write $\widetilde{\mathcal{A}}_k$ to denote the analogue of $\mathcal{A}_k$, and we define $\widetilde{\zeta_0}=0$.
Moreover, all the initial values of stacks $\widetilde{\mathcal{U}}_0, \widetilde{\mathcal{A}}_0, \widetilde{\mathcal{D}}_0$ are equal to the values of their corresponding analogues in Exploration \ref{explore:field}.

We will use several times in the sequel the following notation. Let $\mathcal{F}$ be a $\sigma$-field, and let $W\ge 0$ be a non-negative $\mathcal{F}$-measurable  random variable. Then
\begin{equation*}
    X|\mathcal{F}\sim \operatorname{Exp}(W),
\end{equation*}
means that the conditional law of $X$ given $\mathcal{F}$ is exponential with rate $W$. 
\begin{exploration}[Graph Exploration]\label{explore:Graph}
    Let $k=1$. 
\begin{enumerate}
\item[(\namedlabel{G0}{G0})] \textbf{While} either  $\widetilde{\mathcal{A}}_{k-1}\neq ()$ 
 or $\widetilde{\mathcal{U}}^{\vec{\rho}}_{k-1} \neq \emptyset$ do as follows:
    \item  [(\namedlabel{G1}{G1})] \textbf{Orientation:}\begin{enumerate}
        \item 
         On $\{\widetilde{\mathcal{A}}_{k-1} = ()\} \cap \{\widetilde{\mathcal{U}}^{\vec{\rho}}_{k-1} \neq \emptyset\}$ our exploration recorded all the vertices of the initial $\widetilde{\zeta}_{k-1}$ components of $\mathcal{G}$ which intersect $[m]_{\vec{\rho}}$, and the unexplored part of $\G$ non-trivially intersects
$[m]_{\vec{\rho}}$.      
The algorithm increments $\widetilde\zeta_{k} = \widetilde\zeta_{k-1}+1$, and samples $\widetilde\varpi(k) = (\widetilde{l}_k,\widetilde{i}_k)$ from $\G\setminus \mathcal{D}_{k-1}$ according to
    \begin{align*}
       \p&\Big( \widetilde\varpi(k)  = (l,i)\big| \mathcal{H}_k \Big) \propto \rho_i Q_{i,i}w_{l}^i 1_{[(l,i)\in\widetilde{\mathcal{U}}_{k-1}]} , \\
\widetilde{Y}_{\widetilde{\zeta}_k} |&\mathcal{H}_k \sim \operatorname{Exp}\left(\sum_{(l,i)\in \mathcal{G}\setminus\mathcal{D}_{k-1}} \rho_iQ_{i,i}w_l^i\right),
    \end{align*}
    where $\mathcal{H}_k$ is the $\sigma$-field generated by the first $k-1$ steps of the algorithm, and where $\widetilde{Y}_{\widetilde{\zeta}_k}$ and $\widetilde{\varpi}(k)$ are conditionally independent given $\mathcal{H}_k$.
     Set $\widetilde{\mathcal{A}}_{k-1}^* = (\widetilde\varpi(k))$ and $\widetilde{\mathcal{U}}^*_{k-1} = \widetilde{\mathcal{U}}_{k-1} \setminus\{\widetilde\varpi(k)\}$. Call $\widetilde\varpi(k)$ the \textit{root} of the $
\displaystyle\widetilde{\zeta}_k$th component of $\mathcal{G}$, and define $\widetilde N_k = k$.
     \item Otherwise 
     on $\{\widetilde{\mathcal{A}}_{k-1} \neq ()\}$
      we have (by induction) \begin{equation*}\widetilde{\mathcal{A}}_{k-1} = \Big(\widetilde\varpi(k),\dotsm, \widetilde\varpi(\widetilde{N}_k)\Big),
     \end{equation*} for some $\widetilde{N}_k\geq k$. Set $\widetilde{\mathcal{A}}^*_{k-1} := \widetilde{\mathcal{A}}_{k-1}$, $\widetilde{\mathcal{U}}_{k-1}^* = \widetilde{\mathcal{U}}_{k-1}$, and $\widetilde\zeta_k = \widetilde\zeta_{k-1}$. 
    \end{enumerate}
Note that $\widetilde{N}_k-k+1$ equals the length of the stack $\widetilde{\mathcal{A}}^*_{k-1}$ almost surely. 

\item [(\namedlabel{G2}{G2})]\textbf{The (unexplored) neighbors of $\widetilde\varpi(k)$}: \begin{enumerate}
    \item Let $\widetilde{\mathcal{B}}_k$ be the neighbors of $\widetilde\varpi(k)$ contained in $\widetilde{\mathcal{U}}^*_{k-1}$. We set $\widetilde\chi(k) = \#\widetilde{\mathcal{B}}_k$ and call it  the number of children of $\widetilde\varpi(k)$.
    \item 
    On $\{\widetilde{\chi}(k)=0\}$ the algorithm jumps to \eqref{G3}.
\item
 Otherwise on $\{\widetilde{\chi}(k)>0\}$  the elements of $\widetilde{\mathcal{B}}_k$
are ordered as 
\begin{equation*}
\widetilde\varpi(\widetilde{N}_k+1),\dotsm, \widetilde\varpi(\widetilde{N}_k+\widetilde\chi(k)),
\end{equation*} first
non-decreasingly with respect to their type, and 
for each $i\in [m]$
using conditionally and mutually independent (of all the information which our exploration collected up to this point) auxiliary size-biasing (with respect to weight) of elements of type $i$.  

\end{enumerate}
\item [(\namedlabel{G3}{G3})] \textbf{Update Sets and Stacks}: 
 Set
   $ \widetilde{M}_k := \widetilde{N}_k + \widetilde\chi(k),$ and 
\begin{align*}
\text{on}\  \{ \widetilde{M}_k>k\} \ \text{define} & \\
    \widetilde{\mathcal{A}}_k &:= \Big(\widetilde\varpi(k+1), \widetilde\varpi(k+2),\dotsm, \widetilde\varpi(\widetilde{N}_k + \widetilde\chi(k))\Big), \text{ while}
\\
 \text{ on}\  \{\widetilde{M}_k=k \} \ \text{define } & \widetilde{\mathcal{A}}_k := \emptyset,&\\
    \widetilde{\mathcal{U}}_k&:= \widetilde{\mathcal{U}}_{k-1}^*\setminus \widetilde{\mathcal{B}}_k   \\
    \widetilde{\mathcal{D}}_k&:= \widetilde{\mathcal{D}}_{k-1}\cup\{\widetilde\varpi(k)\}.
\end{align*}
Increment $k$ by 1 and go to step \eqref{G0}.
\end{enumerate}
\end{exploration}

Define
$\zeta_\infty := \sup_k \zeta_k$ and 
$\widetilde{\zeta}_\infty := \sup_k \widetilde{\zeta}_k$.
Note that both 
$\zeta_\infty$ and $\widetilde{\zeta}_\infty$ are finite random variables, and more importantly that the numbers of steps in the above while loops, respectively
$$
K:=\inf\{k: \zeta_k= \zeta_\infty\} \ \text{ and } 
\widetilde{K}:=\inf\{k: \widetilde{\zeta}_k= \widetilde{\zeta}_\infty\} ,
$$
are finite almost surely.
Our key proposition is stated next.
\begin{proposition}\label{prop:Exprole}
Explorations \ref{explore:field} and \ref{explore:Graph} are equal in law.
More precisely,
\begin{equation}
 (\mathcal{A}_k , \mathcal{A}^*_k , \mathcal{U}_k,  \mathcal{U}^*_k,  \zeta_k, Y_{\zeta_k}, \mathcal{B}_k, N_k: k\leq K)  \overset{d}{=} 
  (\widetilde{\mathcal{A}}_k , \widetilde{\mathcal{A}}^*_k , \widetilde{\mathcal{U}}_k,  \widetilde{\mathcal{U}}^*_k,  \widetilde{\zeta}_k, \widetilde{Y}_{\widetilde{\zeta}_k}, \widetilde{\mathcal{B}}_k, \widetilde{N}_k: k\leq \widetilde{K}),
\end{equation}
and $(\varpi(k): k\leq K)\overset{d}{=}(\widetilde\varpi(k): k\leq \widetilde{K})$ in particular.
\end{proposition}

We delay the proof until Section \ref{sec:PropExploreProof}. 

\subsection{Preliminary Lemmas}

Denote by $\le$ (typically we write $\vec{s}\le \vec{t}$) the standard coordinate-wise comparison partial order on  $(-\infty,\infty]^m$.
In complete analogy with the one-dimensional time setting,  we can now define filtrations (and related notions) indexed by $\R_+^m \subset (-\infty,\infty]^m$.The reader is referred to  \cite[Section 2.8]{Ethier:1986} for concepts and results which we will typically use in the sequel without further mention. In particular, the filtration  $\mathscr{F} = (\mathscr{F}(\vec{t});\, \vec{t}\in \R_+^m)$ of our field $\bbX$ is defined by the standard completion of 
$\sigma\left\{ \bbX(\vec{s}): \vec{s}\le \vec{t}\right\}.$ It is clear that $\bbX$ is adapted to $\mathscr{F}$. Furthermore it is easy to see that  $\bbX$ has right-continuous paths with respect to (the above $m$-dimensional partial order in time and) the Euclidean topology on the state space $\R^m$. 
We conclude that $\bbX$ is $\mathscr{F}$-progressively measurable.

We will need the following analogue of the $m$-dimensional strong Markov property.
\begin{lemma}
\label{L:strongManalogue}
Let $\vec{S}$ be a $\mathscr{F}$-stopping time such that $\mathbb{P}(\vec{S}<\infty)=1$. Then 
\begin{equation}\label{eqn:memorylessField}
    \left.\left(\bbX(\vec{t}+\vec{S}) - \bbX(\vec{S}) ;\, \vec{t}\in \R_+^m\right) \,\right| \mathscr{F}(\vec{S}) \overset{d}{=} \left( \tilde{\bbX} (\vec{t});\, \vec{t}\in \R_+^m\right),
\end{equation}
where ${\tilde{\bbX}}$ is the random field distributed as $\bbX^{{\tilde{\bW}}, Q}$ with a random (and $\mathscr{F}(\vec{S})$-measurable) collection of weights $\tilde{\bW} = (\tilde{\bw}^1,\dotsm, \tilde{\bw}^m)$, such that 
$\tilde{\bw}^j$ is the unique vector in $\ell^2_\downarrow$ of finite length whose entries are the non-decreasingly ordered elements of $\{w_l^j: \frac{1}{Q_{j,j}}\xi_l^j> S_j\}$. 
\end{lemma}
\begin{proof}
The assumption is that $\vec{S}$ is a non-negative random vector such that $\{\vec{S}\le \vec{t}\}\in \mathscr{F}(\vec{t})$ for all $\vec{t}$.
The $\sigma$-field $\mathscr{F}(\vec{S})$ consists of all measurable $A$ such that $A \cap \{\vec{S} \leq \vec{s}\} \in \mathscr{F}(\vec{s})$ for all $\vec{s} \geq \vec{0}$.

As in the $1$-dimensional setting, we have $(\vec{S},\bbX(\vec{S})) \in \mathscr{F}(\vec{S})$.
We know that the field of increments $\bbX(\cdot + \vec{S})-\bbX(\vec{S})$ depends deterministically on the residual exponential random variables $(\xi_l^j-Q_{j,j}S_j)_{j\in [m], l\ge 1, w_l^j>0, \xi_l^j/Q_{j,j}>S_j}$ in the same way that the original field $\bbX$ depends on $(\xi_l^j)_{j\in [m], l\ge 1, w_l^j>0}$.
We need to check that, given $\mathscr{F}(\vec{S})$, $(\xi_l^j-Q_{j,j}S_j)_{j\in [m], l\ge 1, w_l^j>0, \xi_l^j/Q_{j,j}>S_j}$ are again independent exponentials, where $\xi_l^j-Q_{j,j}S_j$ has exponential (rate $w_l^j$) distribution.
For this we first note that $\{\xi_l^j/Q_{j,j}>S_j\}\in \mathscr{F}(\vec{S})$, for each
 $l\geq 1$ and $j\in [m]$ such that $w_l^j>0$,
and moreover that $\mathscr{F}(\vec{S})$ is in fact generated by $\vec{S}$ and the family of   
exponentials
$(\xi_l^k/Q_{k,k})_{k\in [m], l\ge 1, w_l^k>0, \xi_l^k/Q_{k,k}\leq S_j}$
which occur prior to $\vec{S}$.
Therefore, on $\{ \xi_l^j/Q_{j,j}>S_j\}$, we have
\begin{equation}
\label{E:simplifyingFF(S)}
\mathbb{P}\left( \left.\frac{\xi_l^j}{Q_{j,j}}-S_j >u \right|  \mathscr{F}(\vec{S})\right) = 
\mathbb{P}\left( \left.\frac{\xi_l^j}{Q_{j,j}}-S_j >u \right| \vec{S}, \left(\frac{\xi_l^k}{Q_{k,k}}\right)_{k\in [m], l\ge 1, w_l^k>0, \xi_l^k/Q_{k,k}\leq S_j} \right).
\end{equation}
Due to the independence of the original family 
$(\xi_l^j)_{j\in [m], l\ge 1, w_l^j>0}$
of exponentials, it is now particularly easy to check that the RHS in \eqref{E:simplifyingFF(S)} equals $e^{-uQ_{j,j}w_l^j}$ almost surely on $\{ \xi_l^j/Q_{j,j}>S_j\}$, provided that the stopping time $\vec{S}$ is a discrete random vector. 

The previous paragraph can be generalized in an obvious way in order to prove conditional independence (given $ \mathscr{F}(\vec{S})$) of residual quantities $\frac{\xi_l^j}{Q_{j,j}}-S_j$ over all relevant $l$ and $j$.
Together with the above discussion, this confirms \eqref{eqn:memorylessField} in the discrete $\vec{S}$ setting.
Deriving 
\begin{equation}
\label{E:residual_time_dist}
\mathbb{P}\left( \left.\frac{\xi_l^j}{Q_{j,j}}-S_j >u \right|  \mathscr{F}(\vec{S})\right)
=e^{-uQ_{j,j}w_l^j} \mbox{ a.s.}
\end{equation}
and its joint  distribution counterparts, or equivalently, deriving \eqref{eqn:memorylessField}  for any stopping time $\vec{S}$ is done in a standard way (by approximating $\vec{S}$ from above with a sequence of discrete stopping times). 
\end{proof}

The following corollary is immediate and will also be used frequently in the sequel. 
\begin{corollary}\label{cor:Continuity1}
    Let $\vec{S}$ be an $\mathscr{F}$-stopping time such that $\p(\vec{S}<\infty) = 1$. Then for any $\vec{t}\in (0,\infty)^m$, and all $i,j\in [m]$
    \begin{equation*}
     \p(X_{i,j}(t_j+S_j) = X_{i,j}((t_j+S_j)-) ) = 1.
    \end{equation*}
    More generally, suppose that $\vec{U}\in \R_+^m$ is an $\mathscr{F}(\vec{S})$-measurable random variable such that
    $\bigcup_j \left\{X_{j,j}(S_{j})\neq X_{j,j}(S_j-),U_j=0\right\}$ is an event of probability zero. Then, for all $i,j\in [m]$
    \begin{equation*}
        \p(X_{i,j}(U_j+S_j) = X_{i,j}((U_j+S_j)-) ) = 1.
    \end{equation*}
\end{corollary}

The following lemma is also easy.
\begin{lemma}\label{lem:SisStoppingInductionStep}
    Suppose that $\vec{S}$ is an almost surely finite $\mathscr{F}$-stopping time, $\vec{U}\in \mathscr{F}(\vec{S})$ is a random variable such that $\p(\vec{U}\ge \vec{S})=1$. 
    Let $(L,I)$ be an $\mathscr{F}(\vec{S})$-measurable random index. 
    Then $\vec{U}$ and $\vec{U}+w_{L}^{I} \vec{R}_{I} 1_{\{\xi_L^I/Q_{I,I} \leq  S_{I}\}}$ are also $\mathscr{F}$-stopping times. 
\end{lemma}

The above lemma implies that in \eqref{X2}(c) all the random vectors $\vec{S}_{N_k+j}^L$, where $\ast\in\{L,R\}$, are stopping times provided that both $\vec{S}_{N_k}^L$ and $\vec{S}_{N_k}^R$ are stopping times. The proof of the next lemma is again a consequence of Lemma \ref{L:strongManalogue}. 
\begin{lemma}\label{lem:Ystopping0}
    Let $\vec{S}$ be an $\mathscr{F}$-stopping time, and let $\mathcal{U} = \{(l,i): \xi_l^i/Q_{i,i}>S_i\}$. Define 
    \begin{equation*}
        Z = \min_{(l,i)\in \mathcal{U}} \left\{\frac{1}{\rho_i}\left(\frac{\xi_l^i}{Q_{i,i}} - S_{i}\right)\right\}.
    \end{equation*}
    Then $\vec{S}+Z\vec{\rho}$ is an $\mathscr{F}$-stopping time, and furthermore
    \begin{equation*}
        Z\big| \mathscr{F}(\vec{S}) \sim \operatorname{Exp}\left(\sum_{(l,i)\in \mathcal{U}} \rho_i Q_{i,i} w_l^i \right),
    \end{equation*}
    and
    \begin{align*}
    \p&\left(\operatornamewithlimits{arg min}_{(l,i)\in \mathcal{U}} \left\{\frac{1}{\rho_i}\left(\frac{\xi_l^i}{Q_{i,i}} - S_i\right) \right\}  = (l_0,i_0)\big| \mathscr{F}(\vec{S})\right)1_{[(l_0,i_0)\in \mathcal{U}]}\propto \rho_{i_0} Q_{i_0,i_0} w_{l_0}^{i_0} 1_{[(l_0,i_0)\in \mathcal{U}]}.
    \end{align*} 
\end{lemma}
\begin{proof}
    As already noted, we will apply Lemma \ref{L:strongManalogue}.
    More precisely, it suffices to prove the statements on the conditional law of $Z$ and $\operatornamewithlimits{arg min}_{(l,i)\in \mathcal{U}} \left\{\frac{1}{\rho_i}\left(\frac{\xi_l^i}{Q_{i,i}} - S_i\right)\right\}$ in the  setting where $\p(\vec{S} = \vec{0})=1$.
    Note that these are a clear consequence
    of the elementary properties of (conditionally) independent exponential random variables. 
    
    Furthermore, $Y_1 \vec{\rho}$ (where $Y_1$ is defined in \eqref{X1}(a)) is an $\mathscr{F}$-stopping time since \begin{align*}
        \{Y_1\vec{\rho} \le \vec{t}\} &= \left\{\exists (l,i): \frac{\xi_l^i}{\rho_i Q_{i,i}} \rho_i \le t_i  \right\} =  \left\{\exists (l,i): \frac{\xi_l^i}{Q_{i,i}} \le t_i  \right\} \\  &=  \left\{\exists i:X_{i,i}\textup{ has a jump in }[0,t_i] \right\}, 
    \end{align*} 
    and the event on the RHS is clearly $\mathscr{F}(\vec{t})$-measurable. 
    Applying Lemma \ref{L:strongManalogue}, we arrive to the conclusion that $Z\vec{\rho}$ is a stopping-time with respect to the natural filtration of $\bbX(\cdot + \vec{S}) - \bbX(\vec{S})$, and a little thought is needed to see that this is enough to conclude that $\vec{S} + Z\vec{\rho}$ is a stopping time with respect to $(\mathscr{F}(\vec{t}), \vec{t} \geq \vec{0})$.
\end{proof}

By induction, Lemmas \ref{lem:SisStoppingInductionStep} and \ref{lem:Ystopping0} imply the following:
\begin{lemma}\label{lem:SareStopping}
Recall Exploration \ref{explore:field}.
    The vectors $\vec{S}_k^L$ and $\vec{S}_k^R$ constructed 
    in \eqref{X1}(a)
    are $\mathscr{F}$-stopping times, for any relevant $k$. 
\end{lemma}

\begin{remark}
Lemma \ref{lem:Ystopping0} guarantees that the start of exploration (in \eqref{X1}(a)) of each connected component occurs at a stopping times.
    Lemma \ref{lem:SisStoppingInductionStep} implies that $\vec{S}^L_k$ and $\vec{S}^R_k$ are stopping times for each $k$ such that $\varpi(k)$ is not a root of a connected component (or equivalently, $\varpi(k)$ is discovered in \eqref{X2}(c) at some earlier step).
\end{remark}

The following result will be used to prove the required ``equivalence'' of \eqref{X2} in Exploration \ref{explore:field} and \eqref{G2} in Exploration \ref{explore:Graph}.
\begin{lemma}\label{lem:neighbors}
    Suppose that $\vec{S}$ is an $\mathscr{F}$-stopping time. Let $\mathcal{U}$ be the  $\mathscr{F}(\vec{S})$-measurable collection defined as
    $\{(l,i): i\in[m], l\in\len(\bw^{i}), \xi_l^i/Q_{i,i}> S_i\}.$ 
    Let $(L,I)$ be an $\mathscr{F}(\vec{S})$-measurable random index.

    Then
    \begin{equation*}
        \p\left(\frac{\xi_r^j}{Q_{j,j}} \in[S_j,S_j+ w_{L}^{I}R_{j,I})\Big| \mathscr{F}(\vec{S})\right)1_{[(r,j)\in \mathcal{U}, (L,I)\not\in \mathcal{U}]} = \left(1-e^{-  Q_{j,I} w_{L}^{I} w_r^j}\right) 1_{[(r,j)\in \mathcal{U}, (L,I)\not\in \mathcal{U}]}.
    \end{equation*}
\end{lemma}
\begin{proof}
    Due to Lemma \ref{L:strongManalogue}, we can suppose without loss of generality that $\vec{S} = \vec{0}$, where we replace $\bW$ with $\widetilde\bW$.
    
   Since $\xi_r^j\sim \operatorname{Exp}(w_r^j)$, we have
    \begin{equation*}
        \p\left(\xi_r^j/Q_{j,j} \in[0,w_{L}^{I} R_{j,I})\right) = 1 - \exp\left(-Q_{j,j} R_{j,I} w_{L}^{I} w_r^j\right).
    \end{equation*}
    The claim now follows from the identities $R_{j,I}Q_{j,j} \equiv Q_{j,I}$ (see \eqref{D:matrix_R}).
\end{proof}

\subsection{Proof of Proposition \ref{prop:Exprole}}\label{sec:PropExploreProof}
The argument relies on the inductive coupling of the two explorations. We will need the following notation. Let $\mathcal{E}_k$ (resp. $\widetilde{\mathcal{E}}_k$) denote the $\sigma$-algebra generated by the initial $k \wedge K$ (resp.~ $k \wedge \widetilde{K}$) steps of Exploration \ref{explore:field} (resp. \ref{explore:Graph}). 
We begin by noting that by induction
\begin{equation}\label{eqn:EkandF}
    \mathcal{E}_k  = \mathscr{F}\left(\vec{S}^R_{k \wedge K}\right).
\end{equation}
Indeed, note that the algorithm gathers all the information in $\mathcal{E}_k$ while passing through steps \eqref{X1}--\eqref{X3} in the initial $k \wedge K$ rounds (in fact, 
this information is gathered already by the end of \eqref{X2}(c) in round $k \wedge K$). In particular, all the jump times of $\bbX$ occurring before $\vec{S}_{k\wedge K}^R$ are included in $\mathcal{E}_k$,
hence $\mathcal{E}_k\supset \mathscr{F}(\vec{S}^R_{k \wedge K})$. 
In the opposite direction, note that using $\mathscr{F}(\vec{S}_{k \wedge K}^R)$ we can reconstruct the field exploration for the first $k \wedge K$ steps (or rounds), so that $\mathcal{E}_k\subset \mathscr{F}(\vec{S}^R_{ k \wedge K})$. 

Unless needed for additional clarity, we will abuse notation and write ``step $k$'' instead of ``step $k\wedge K$'' (or ``step $k\wedge \widetilde{K}$''). We will also write $\vec{S}_k^\ast$ instead of $\vec{S}_{k\wedge K}^\ast$, where as usual $\ast\in\{L,R\}$.

\subsubsection{The base of induction}
If $k = 1$, we  need to compare steps \eqref{X1}(a) and \eqref{G1}(a). In this case, Lemma \ref{lem:Ystopping0}  
(with $\vec{S}=\vec{0}$) implies that the respective outputs of steps \eqref{X1}(a) and \eqref{G1}(a) are equal in distribution. 
We couple them so that they are equal almost surely.

Recall that $\varpi(1) = (l_1,i_1)$ and similarly $\widetilde\varpi(1) = (\widetilde{l}_1,\widetilde{i}_1).$ 
As just explained, 
in our coupling
$(\widetilde{l}_1, \widetilde{i}_1) = (l_1,i_1)$. Continuing onto \eqref{X2} and \eqref{G2}, for each vertex $v\in \mathcal{U}^*_0 = \widetilde{\mathcal{U}}_0^* 
= \mathcal{U}_0\setminus 
\{\varpi(1)\}
$
\begin{equation*}
    \p(v\in \mathcal{B}_1\big| \mathcal{E}_0, \mathcal{U}_0^*) = 1-\exp\left({-Q_{i_1,\TYPE(v)} w_v w_{l_1}^{i_1}}\right),
\end{equation*}
 due to Lemma \ref{lem:neighbors}, and
similarly, 
\begin{equation*}
    \p\left(v\in \widetilde{\mathcal{B}}_1\big| \widetilde{\mathcal{E}}_0,\widetilde{\mathcal{U}}_0^*\right) = 1-\exp\left(-Q_{i_1,\TYPE(v)} w_v w_{l_1}^{i_1}\right),
\end{equation*}
due to the very definition of $\G$ (in particular, the independence of 
edge connections in DCSBM).  
Therefore, $\mathcal{B}_1 | (\mathcal{E}_0, \mathcal{U}_0^*)
\overset{d}{=} \widetilde{\mathcal{B}}_1 | (\widetilde{\mathcal{E}}_0, \widetilde{\mathcal{U}}_0^*)$,
and more importantly 
$\xi_l^i/Q_{i,i}\in [S_{1;i}^L, S_{1,i}^R)$ 
can be used (verbatim)  to generate the output of \eqref{G2}(c).
More precisely, we repeatedly apply 
Lemma \ref{L:strongManalogue} and Lemma \ref{lem:SareStopping} to see that the residual clocks in Step (c) of \eqref{X2} are still (conditionally) independent exponential random variables. There is no analogue of 
\eqref{X2}(d)
in Exploration \ref{explore:Graph}.
In particular, the sequences  $(\vec{S}_k^*)_k$ do not appear in the statement of Proposition \ref{prop:Exprole}, yet they play an important role in its proof. 
As before, we couple the outputs of
\eqref{X2}(c) and \eqref{G2}(c)
so that they are equal almost surely.

Since \eqref{X3} (resp.~\eqref{G3}) depends deterministically on the output of \eqref{X2} (resp.~\eqref{G2}), we conclude that \eqref{X3} and \eqref{G3} are clearly 
almost surely equal 
(at least, if $k=1$) in our coupling.

\subsubsection{Induction Step} 

Suppose that
on $\{K\geq k-1\}=\{\widetilde{K}\geq k-1\}$
we already have a coupling in which the outputs of the two explorations are identical almost surely in steps $1,2,\dotsm, k-1$. 
We need to extend it so that 
$\{K= k-1\}=\{\widetilde{K}= k-1\}$,
and furthermore that on 
$\{K\geq k\}=\{\widetilde{K}\geq k\}$
the outputs of the two explorations in step $k$ are again identical almost surely.

The induction step is essentially identical to the base case $k =1$, with additional applications of Lemma \ref{lem:SareStopping} whenever necessary. 
Here we 
only sketch the argument.

Recall that the two explorations have the same condition for entering the ``while loop'' in steps 
\eqref{X0} and \eqref{G0}, respectively.

On $\{K\geq k-1\} \cap \{\mathcal{A}_{k-1} \neq \emptyset\} = 
\{\widetilde{K}\geq k-1\} \cap \{\widetilde{\mathcal{A}}_{k-1} \neq \emptyset\}$ the outputs of \eqref{X1} and \eqref{G1} are a deterministic function of the outputs of step $k-1$ in the already constructed coupling, and therefore they are identical almost surely. 
On
$\{K\geq k-1\} \cap \{\mathcal{A}_{k-1} = \emptyset\} = 
\{\widetilde{K}\geq k-1\} \cap \widetilde{\mathcal{A}}_{k-1} = \emptyset\}$
we rely on $\mathcal{U}_{k-1}=\widetilde{\mathcal{U}}_{k-1}$ and therefore 
$\mathcal{U}_{k-1}^{\vec{\rho}}=\widetilde{\mathcal{U}}_{k-1}^{\vec{\rho}}$
almost surely.
Hence $\{K=k-1\}= $
$$
\{ K\geq k-1,
\mathcal{A}_{k-1} = \emptyset,
\mathcal{U}_{k-1}^{\vec{\rho}} =\emptyset
\}
= \{ \widetilde{K}\geq k-1,
\widetilde{\mathcal{A}}_{k-1} = \emptyset,
\widetilde{\mathcal{U}}_{k-1}^{\vec{\rho}} =\emptyset
\}
= \{\widetilde{K}=k-1\},
$$
up to null-sets, and furthermore on $\{K>k-1\}\overset{a.s.}{=}\{\widetilde{K}>k-1\}$ we apply (as in the base of induction)
Lemmas \ref{L:strongManalogue}, \ref{lem:Ystopping0} and \ref{lem:SareStopping} in order 
to couple the outputs of \eqref{X1}(a) and \eqref{G1}(a). 

The coupling in steps \eqref{X2} and \eqref{G2}, and in steps \eqref{X3} and \eqref{G3}, respectively, is done as in the base case.
This concludes the proof of Proposition \ref{prop:Exprole}.

\subsection{Exploration \ref{explore:field} and hitting times \texorpdfstring{$\bT$}{T}}
Recall the definition of $X_{i,j}$ in \eqref{eqn:XfieldDef1}.
Due to elementary properties of independent exponentials, for each fixed $j$, there is an almost surely uniquely defined permutation $\pi^j$ on $\text{len}({\bw}^j)$ letters such that
\begin{equation}
\label{D:pi_hat_j}
    \xi^j_{\pi^j(1)}< \xi^j_{\pi^j(2)}<\dotsm < \xi^j_{\pi^j(\text{len}({\bw}^j))}.
\end{equation}
We will sometimes write $\pi(l,j)$ in place of the value  $\pi^j(l)$ of the permutation $\pi^j$ at $l$.
We shall also sometimes write $\pi(l,j)$ instead of vertex $(\pi^j(l), j)$.

As in the proof of Lemma \ref{lem:Ystopping0}, observe that with probability $1$
\begin{equation}
\label{E:loadfree}
    \bbX(\vec{t}-) = -\vec{t} \quad \text{ if and only if }\quad 
    \vec{t} 
    \leq
    \left(
    \frac{1}{Q_{1,1}} \xi_{\pi(1,1)},\dotsm, 
    \frac{1} {Q_{m,m}}\xi_{\pi(1,m)} 
    \right)
\end{equation}
since the first jump time on the $j$th time-line is $\frac{1}{Q_{j,j}}\xi_{v}$, where $v = \pi(1,j)$,
and since for all $\vec{t}$ with $j$th component $t_j \ge \frac{1}{Q_{j,j}}\xi_{\pi(1,j)}$ we have by definition
$X_j(\vec{t}) \geq X_{j,j}(t_j) \ge -t_j+w_{\pi(1,j)} >-t_j$.

Recall the definition of $Y_1$ in \eqref{X1}(a) and recall that $Y_1\vec{\rho}$ is an $\mathscr{F}$-stopping time (see Lemma~\ref{lem:Ystopping0}).
Due to \eqref{E:loadfree} we have that
\begin{equation}\label{eqn::Tlinear}
    \{T(y) =  y \vec{\rho}\} = \left\{ Y_1\ge y\right\} \text{ a.s.}.
\end{equation}  
The above remarks together with standard properties of exponentials imply the following. 
\begin{lemma} \label{lem:Y1isFirstJump}
    $\bT(y) = y\vec{\rho}$ for $y\le Y_1$ and $\bT(Y_1+)\neq \bT(Y_1)$.
\end{lemma}

Now let $\varpi(1),\varpi(2),\dotsm, \varpi(N)$ be the vertices of the first connected component discovered in Exploration \ref{explore:field}. 

More precisely, define 
$N:=\inf\{k: \zeta_{k}=2\}-1$.
In the statement of the next result we use the above notation.
\begin{lemma}\label{lem:jump1}
We have
    \begin{equation}
    \label{E:lem_jump1_1}
        \bT(Y_1) = \vec{S}_1^L\qquad\text{and}\qquad\bT(Y_1+) = \vec{S}_N^R\qquad \text{ almost surely.}
    \end{equation}
    On $\{K>N\}$ the restriction of $\bT$ to $(Y_1, Y_1+Y_2]$ is an affine function 
    $$
    y \mapsto \bT(y)= \bT(Y_1+) +(y- Y_1) \vec{\rho}.
    $$
    
     \noindent
    More generally, for any $l\leq r$ and $p$ positive integers, on the event
    $ \{\zeta_{l-1}=p-1, \zeta_l=p, \zeta_r=p\} \cap (\{K=r\} \cup \{ K> r,\zeta_{r+1}=p+1\} )$
(where
     $\{\varpi(l),\dotsm, \varpi(r)\}$ is the $p^\text{th}$ connected component explored), 
     we have
    \begin{equation}
     \label{E:lem_jump1_2}
        \bT\left(\sum_{n=1}^p Y_n\right) = \vec{S}_l^L \qquad \text{and}\qquad \bT\left(\sum_{n=1}^p Y_n+\right) = \vec{S}_r^R\qquad \text{ almost surely, and}
    \end{equation}
    \noindent
    on the above event intersected with $\{K>r\}$ the restriction of
     $\bT$ to $(\sum_{n=1}^p Y_n, \sum_{n=1}^{p+1} Y_n]$ is an affine function 
    $$
    y \mapsto \bT(y)= \bT\left(\sum_{n=1}^p Y_n+\right) +\left(y- \sum_{n=1}^p Y_n\right) \vec{\rho}.
    $$
\end{lemma}

\begin{proof} We prove in detail the statements which concern the first encountered component, or more precisely the process $\bT$ before and after its first jump. The statements related to the $p$th jump of $\bT$ can be proven in a similar fashion, using the strong  Markov property of $\bbX$. We provide a sketch, and leave the details to an interested reader. 

    The first identity in \eqref{E:lem_jump1_1} is an easy consequence of  the definition of $\vec{S}_1^L$ in \eqref{X1}(a).

    In order to establish 
    the second identity in 
    \eqref{E:lem_jump1_1} 
    it suffices to show that almost surely
 \begin{eqnarray}
 & &  
  \label{E:lem_jump1_suff1}
    \bT(Y_1+)\leq S_N^R +\delta \vec{\rho},
     \text{ for any }\delta >0, \text{ and }\\
     & &
 \label{E:lem_jump1_suff2}
     \bT(Y_1+\eps)> S_N^R,
     \text{ for any }\eps >0. 
 \end{eqnarray}

As a preliminary calculation we show (see Corollary \ref{cor:Continuity1}) that almost surely 
 \begin{equation}
 \label{E:XofSNRisY1rho}
      X_j(S_{N;i}^R-) = 
   X_j(S_{N;i}^R) = \sum_{j=1}^m X_{j,i}(S_{N;i}^R) = -\rho_jY_1, \text{for each } j\in[m].
 \end{equation}
 Let $\vec{e}_k$ be the $k^\text{th}$ standard basis vector of $\R^m$, and note that by linearity
\begin{equation}\label{eqn:comp1Gap}
        \vec{S}_N^R-\vec{S}_1^L = \sum_{p=1}^N \vec{R}_{\TYPE(\varpi(p))} w_{\varpi(p)} = R \left(\sum_{p=1}^N w_{\varpi(p)} \vec{e}_{\TYPE(\varpi(p))} \right) = R \vec{\sM},
    \end{equation}
    where $\vec{\sM}$ is the total weight vector (broken by type, see \eqref{D:masses_M}) 
    of the first explored component.
    
    In addition, for $i\neq j$ we have by construction of the stopping time $\vec{S}_N^R$ (see again \eqref{X2}(d) and \eqref{eqn:XfieldDef1}) that
    \begin{equation*}
        X_{j,i}(S_{N;i}^R) = \sum_{\substack{p: \TYPE(\varpi(p)) = i\\ p\le N}} R_{j,i} w_{\varpi(p)}
    \end{equation*}
    as each jump of $X_{j,i}$ that occurs before time $S_{N;i}^R$ 
    corresponds either to $\varpi(1)$ or to
    a (type $i$) child of some vertex from $\varpi(1),\dotsm, \varpi(N)$.
    Similarly,
    \begin{equation*}
        X_{j,j}(S_{N;j}^R) = -S_{N;j}^R + \sum_{\substack{p: \TYPE(\varpi(p)) = j\\ p\le N}} w_{\varpi(p)}. 
    \end{equation*}
    Adding up the identities in the last two displays  and recalling that $R_{j,j} = 1$ (see \eqref{D:matrix_R}) implies
    \begin{align*}
        \sum_{i=1}^m& X_{j,i}(S_{N;i}^R) = -S_{N;j}^R + \sum_{i=1}^m \sum_{\substack{p: \TYPE(\varpi(p)) = i\\ p\le N}} R_{j,i} w_{\varpi(p)} =- S_{N;j}^R + \sum_{p=1}^N R_{j,\TYPE(\varpi(p))} w_{\varpi(p)}\\
        &= -S_{N;j}^R + \vec{e}_j^TR\vec{\sM} = -S_{1;j}^L = -\rho_j Y_1,
    \end{align*}
    where we used \eqref{eqn:comp1Gap}
    for the second to last identity.

We now know that $\bbX(\vec{S}_N^R) = -Y_1 \vec{\rho}$, and furthermore  recall that the stack of active vertices (of the first explored component) is exhausted at time $\vec{S}_N^R$.
The field $\bbX$ evolves deterministically for $t_l\in (S_{N;l}^R,S_{N;l}^R +  Y_2\rho_l)$ where $Y_2$ is again defined in \eqref{X1}(a), and where 
$\vec{S}_{N+1}^L:=\vec{S}_N^R +Y_2 \vec{\rho}$ on $\{Y_2 < \infty\}$.
More precisely, during $(\vec{S}_{N;l}^R,\vec{S}_{N;l}^R+Y_2 \rho_l)$, the $l$th coordinate of $\bbX$ deterministically decreases at rate $1$.
If $\rho_l=0$ the $l$th coordinate of $\bbX$ is not relevant in \eqref{X1}(a).
In particular, $\{\delta< Y_2\} = \{\bbX(\vec{S}_N^R+\delta \vec{\rho})= - (\delta+Y_1) \rho \}$ and therefore we  have
\begin{equation}
\label{E:interesting_upper_bd_bT}
\{\delta< Y_2\} \subset \{ \bT(Y_1+\delta) \leq \vec{S}_N^R+\delta \vec{\rho}\},
\end{equation}
which is enough for concluding 
\eqref{E:lem_jump1_suff1}.
Note that $\vec{S}_{N+1}^L$ is defined only on $\{Y_2<\infty\}$, but the rest of the argument is also valid 
on $\{Y_2=\infty\}$, the event that the exploration process ends upon exploring the first (and only) connected component which intersects $[m]_{\vec{\rho}}$.
    
We next show \eqref{E:lem_jump1_suff2}. Again due to Corollary \ref{cor:Continuity1} (or by the reasoning in the previous paragraph) we know
 that $\bT(Y_1+\eps)$ is a point of continuity for $\bbX$. We will show by induction that $\bT(Y_1+\eps)\ge \vec{S}_k^R$ for all $k=1,2,\dotsm,N$. 

    To do this, we observe that for each $p = 2,\dotsm,N$ the vertex $\varpi(p) = (l_p,i_p)$ is discovered strictly before the start of its corresponding observation window in \eqref{X2}(a). In terms of the exponential clock $\xi_{l_p}^{i_p}$, this 
    property can be written as
    \begin{equation*}
        \frac{\xi_{l_p}^{i_p}}{Q_{i_p,i_p}} < S^R_{p-1;i_p\color{black}} = S^L_{p;i_p\color{black}}, \ \ p\geq 2,
    \end{equation*} 
    
    where the last identity 
    \color{black}
    above follows form \eqref{X2}(d). 
    The root vertex $\varpi(1)$ is discovered at 
      $S_{1,i_1}^L$,
      the start of its observation window.

    Therefore, on $\{k\leq N\leq K\}$,
    before time $\vec{S}_k^L$ there are at least $\varpi(1),\ldots,\varpi(k)$ already discovered in Exploration \ref{explore:field}.
    So for each $k = 1,2,\dotsm, N$ and $i\neq j$ we have
    \begin{align}\label{eqn:BoundsForXij1}
        X_{j,i}(t_i) &\ge \sum_{p\le k: i_p = i} R_{j,i} w_{\varpi(p)} ,\   \forall  t_i\ge S_{k;i}^L,
        \text{ and moreover}
        \\
        X_{j,j}(t_j) &\ge -S_{k;j}^R + \sum_{p\le k: i_p = j} w_{\varpi(p)}, \  \forall t_j\in[S_{k;j}^L,S_{k;j}^R]\label{eqn:BoundsForXij2} , \text{ and }
     \\
        X_{j,j}(t_j)&\ge -s_j,\ \forall t_j\le s_j. \label{eqn:BoundsForXij2.a}
    \end{align}
    Similarly, if $k = 1$ our exploration algorithm guarantees that for $j\neq i_1$ 
    \begin{align}\label{eqn:BoundsForXij3}
        X_{j,i_1}(t_{i_1}) &\ge R_{j,i_1} w_{\varpi(1)}, \   \forall t_{i_1}\ge S_{1;i_1}^L.
    \end{align} 
    Recall the definition of $\vec{R}_j$ in \eqref{eqn:Rvec}.
   Recall again (see \eqref{X1}(a) and \eqref{X2}(d)) the recursion 
    \begin{equation}\label{eqn:SkiX2}
    S^L_{1;i} = \rho_i Y_1, \ 
    S_{k;i}^R = S_{k;i}^L + R_{i,i_{k}}w_{\varpi(k)},\ k\in [N],\ \text{ and } S_{p;i}^L = S_{p-1;i}^R,\  p\in [N]\setminus\{1\},
    \end{equation} 
    for each $i\in [m]$, where as usual $i_k=\TYPE(\varpi(k))$.
    \color{black}
    In particular, using \eqref{eqn:BoundsForXij2.a}, \eqref{eqn:BoundsForXij3}, and
    the middle identity in \eqref{eqn:SkiX2} we see that if $j\neq i_1$ then for any $\vec{t}$ such that $t_{i_1}\ge S_{1;i_1}^L$ and  $t_j\le S_{1;j}^R$
    \begin{align}
    \label{E:align1}
   \textup{we have }\  X_j(\vec{t})=  \color{black}\sum_{i=1}^m X_{j,i}(t_i)\color{black} &\color{black}\ge \underset{X_{j,i_1}(t_{i_1})}{\underbrace{R_{j,i_1}w_{\varpi(1)}}}  - \underset{X_{j,j}(t_j)}{\underbrace{S_{1;j}^R}} = -\rho_jY_1.
      \end{align}
      Similarly, if $j = i_1$  we apply \eqref{eqn:BoundsForXij2}, monontonicity of off-diagonal processes $X_{i_1,j}(\cdot)$, $j\neq i_1$, and again \eqref{eqn:SkiX2} to conclude that for any $\vec{t}$ such that$t_{i_1}\in[S_{1;i_1}^L, S_{1;i_1}^R]$
      \begin{align}
    \label{E:align2}
    \textup{we have }\  X_{i_1}(\vec{t})=  \color{black} \sum_{i=1}^m X_{i_1,i}(t_i) &\color{black} \ge -S_{k;i_1}^R + w_{\varpi(1)} = -\rho_{i_1}Y_1.
    \end{align}
    The estimates in \eqref{E:align1}--\eqref{E:align2}, combined with the fact that $\bT(Y_1)=^{a.s}\vec{S}_1^L$ immediately imply that $\bT(Y_1+\epsilon) > \vec{S}_1^R$, almost surely, which is the base of our induction argument for \eqref{E:lem_jump1_suff2}.

\color{black}
    Now suppose that we have shown 
    $\{k-1\leq N\leq K\}\subset \cap_{\eps>0} \{\bT(Y_1+\epsilon)>\vec{S}_{k-1}^R\}$.
    We wish to show that $\{k\leq N\leq K\}\subset  \cap_{\eps>0}  \{\bT(Y_1+\epsilon)>\vec{S}_{k}^R\}$.
 Recalling $\vec{S}_{k-1}^R=\vec{S}_k^L$,
we proceed (as in the above analysis on $[\vec{S}_1^L, \vec{S}_1^R]$) to estimate $\bbX$ from below on $[\vec{S}_k^L, \vec{S}_k^R]$.
More precisely, observe 
that if $\vec{t}\geq \vec{S}_k^L = \vec{S}_{k-1}^R$ is such that
$t_j\le S_{k;j}^R$ for some 
$j\in[m]$
then
    \begin{align*}
        \sum_{i=1}^m X_{j,i}(t_i) &\ge \underset{\eqref{eqn:BoundsForXij1}}{\underbrace{\sum_{i\neq j} \sum_{p\le k: i_p = i} R_{j,i_p} w_{\varpi(p)} }} + \underset{\eqref{eqn:BoundsForXij2}}{\underbrace{\sum_{p\le k: i_p = j} w_{\varpi(p)} - S_{k;j}^R}}\\
        &= {{\sum_{p\le k} R_{j,i_p} w_{\varpi(p)}}} - S_{k;j}^R\\
        &= -S_{1;j}^L = -\rho_jY_1.
    \end{align*}
    In the second line above we combined the two summations and used $R_{j,j} =1$.
    For the 
    final two identities we again applied \eqref{eqn:SkiX2} and used telescoping.
    We obtain $\{k\leq N\leq K\}\subset \cap_{\eps>0} \{\bT(Y_1+\epsilon)>\vec{S}_{k}^R\}$, which concludes the induction step and establishes \eqref{E:lem_jump1_1}.

Now consider
$$
\tilde{\bbX} (\vec{t}):= \bbX(\vec{t}+\vec{S}_N^R) - \bbX(\vec{S}_N^R), \ \vec{t}\geq \vec{0}.
$$
Stopping time $\vec{S}_N^R$ is the direct analogue of $\vec{0}$, while  $\vec{S}_{N+1}^L-\vec{S}_N^R$ (if finite, that is well-defined) is the direct analogue of
$\vec{S}_1^L$, where $\widetilde{\bbX}(\cdot)$ now takes the role of $\bbX(\cdot)$.
Therefore, the whole argument above can be repeated for the exploration of the second, and iteratively, of the $p$th connected component of DCSBM intersecting $[m]_{\vec{\rho}}$.
In addition, \eqref{eqn::Tlinear} 
and Lemma \ref{lem:Y1isFirstJump}
apply directly on the shifted process $\widetilde\bbX$ (and its iterations), and imply in particular that 
$$
\widetilde\bT(y) = y\vec{\rho},  \text{ iff } y\le \widetilde{Y}_1,
$$
where clearly 
on the event $\{N+1\leq K\}$ we have $\widetilde{Y}_1\equiv Y_2$.
In addition, note that 
for any $y>0$
(due to \eqref{E:XofSNRisY1rho}) we have
$\widetilde\bT(y):=\inf\{\vec{t}\geq \vec{0}: \widetilde{\bbX}(\vec{t}-)\leq -y\vec{\rho}\}\equiv \inf\{\vec{t}\geq \vec{S}_N^R: \bbX(\vec{t}-)\leq -(y+Y_1)\vec{\rho}\} - \vec{S}_N^R =: \bT(y+Y_1) - \vec{S}_N^R$,  almost surely. 

With these correspondencies, the above linear expression for $\widetilde\bT$ becomes $\bT(Y_1+y)= \vec{S}_N^R + y\vec{\rho}= \bT(Y_1+) + y\vec{\rho}$, where $y \in (0,Y_2]$, and this is clearly equivalent to the (affine map) expression in the statement of the lemma.
\end{proof}

\subsection{Proof of Theorem \ref{thm:discreteEncoding}}\label{sec:proofOfDiscrete}\label{sec:BFW_discrete_proofs}

In previous sections we did most of the work needed for completing the proof of Theorem \ref{thm:discreteEncoding}.

The fact that the connected components appear in a size-biased order follows  from the construction of Exploration \ref{explore:Graph} and the equivalence with Exploration \ref{explore:field} in Proposition \ref{prop:Exprole}. 
Indeed, let $\{\cC_j;j\ge1\}$ be the connected components of $\G$, and recall that 
$\mathscr{S}(\cC_j) = \sum_{(l,i)\in \cC_j} \rho_i Q_{i,i}w_l^i$.  The  size-biased ordering with respect to the weights $(\mathscr{S}(\cC_j))_j$ is conventionally done as follows: let 
\begin{equation}
\label{E:conv_size_biasing}
    E_l\,|\,\sigma(\cC_j,\ j\geq 1)\sim \operatorname{Exp}(\mathscr{S}(\cC_l)), \ \ l \geq 1,   
\end{equation}
be conditionally independent exponentials; the size-biased ordering
$(\cC_{\tau_k})_k$ is defined through a random permutation $(\tau_k)_k$,  which is specified via 
$$E_{\tau_k}<E_{\tau_{k+1}},\ \ k\geq 1, \ \text{almost surely}. $$  
The random ordering of components obtained in  Exploration \ref{explore:Graph} is different from the conventional ordering via \eqref{E:conv_size_biasing} in at least two ways: (a) we use fragments of connected component weights (when searching for the next ``root vertex'') instead of full connected component weights, and (b)  we draw (conditionally independent) exponential variables sequentially rather than simultaneously (this corresponds to our gradual exploration of $\mathcal{G}$).\\
Nevertheless, the two approaches can be easily linked due to these elementary properties: (i) the minimum of $n$ independent exponetial random variables is again an 
exponential variable with rate equal to the sum of $n$ individuals rates, and (ii) $  (E_{\tau_k})_k= (\min_j E_j, \min_{j \neq \tau_1} E_j, \min_{j\not\in \{\tau_1,\tau_2\}} E_j,\ldots)$. 

Now suppose that we are given another family $(F_v)_{v\in \mathcal{G}}$ of (conditionally and) mutually  independent exponentials given $(\cC_j)_{j\geq 1}$, such that
$$
  F_v \,|\, \sigma(\cC_j,\ j\geq 1)\sim \operatorname{Exp}
  (\rho_{\TYPE(v)}Q_{\TYPE(v),\TYPE(v)}w_v), \ \ v\in \mathcal{G}.
$$
Due to (i) above, we have 
$$
 E_l\,|\,\sigma(\cC_j,\ j\geq 1)\overset{d}{=} \min_{v\in \cC_l}  F_v|\,\sigma(\cC_j,\ j\geq 1).
$$
Let $\cC(v)$ denote the connected component of $\mathcal{G}$ which contains $v$.
If $V_1:=\operatornamewithlimits{argmin}_{v\in \mathcal G}  F_v $ let 
$V_2:=\operatornamewithlimits{argmin}_{v\in \mathcal G \setminus \cC(V_1)}  F_v $, 
$V_3:=\operatornamewithlimits{argmin}_{v\in \mathcal G \setminus (\cC(V_1)\cup \cC(V_2))}  F_v $, and continue inductively. The reader should note that here we confound again the graph with its vertex set, for ease of notation.

The main point of this discussion is the observation that, due to (i) and (ii) above, we have
\begin{equation}
\label{E:main_point_sb}
     (E_{\tau_k})_k \overset{d}{=} (\min_v F_v, \min_{v \in \mathcal G \setminus \cC(V_1)}  F_v, \min_{v \in \mathcal G \setminus (\cC(V_1)\cup \cC(V_2)) }F_v,\ldots). 
\end{equation}
We leave it to the interested reader to check that Exploration \ref{explore:Graph} in step \eqref{G1}(a) gradually produces a sequence of exponential random variables equally distributed as the right-hand-side in \eqref{E:main_point_sb}.

Let $\G = \G(\bW,Q)$ and 
recall the notation given before the statement of Theorem \ref{thm:discreteEncoding}.
 Given $\G$, let us generate conditionally independent exponential random variables $(E_r;\, r\ge 1)$ with respective rates $(\sS(r);\, r\ge 1)$.
Combining the formulae for $\bT$ in Lemma \ref{lem:jump1} with the size-biased ordering of connected components 
yields the following.
\begin{corollary}
    \label{coro:representatbT}
    We have
\begin{equation*}
\left(\bT(y);\, y\ge 0\right)\overset{d}{=} \left(\vec{\rho}y + \sum_{r} R\vec{\sM}(r) \, 1_{\{E_r<y\}};\, y\ge 0 \right).
\end{equation*} 
\end{corollary}
The strict inequality in the event of the $r$th indicator function is completely consistent with Lemma \ref{lem:jump1} and the left continuity of $\bT$.
While Lemma \ref{lem:jump1} does not concentrate on the behavior of $\bT$ after its final ($\zeta_K$th) jump, it is easy to see that on $(\sum_{i=1}^{\zeta_K} Y_i,  +\infty)$ $\bT$ evolves as a deterministic affine map, parallel to the line $y \mapsto y\vec\rho$.

\section{Excursion representation: preliminaries and special case}

\label{sec:excursions_along_a_curve}
This section is devoted to the analysis of fields and their hitting times in the deterministic setting where the fields are well-behaved. In the next section, we will explore the general case, but this special case is instructive for our later construction and motivation. We start by establishing some notation which will be useful throughout the sequel.

\subsection{Notation}
We will define several classes of functions with the following inclusions
\begin{equation*}
    \Dfu \subset\Dfuw \subset \Df_0^+(\R_+)\subset \Df_0(\R_+)\subset \Df(\R_+).
\end{equation*}
As usual, we denote by $\Df([0,\infty))=\Df(\R_+)$ the Skorohod space of real \cdl\ functions on $\R_+$. Let $\Df_0(\R_+)\subset \Df(\R_+)$ contain all  $f\in \Df(\R_+)$ such that $f(0) = 0$. 
Note that all $f\in \Df(\R_+)$ is necessarily continuous at $0$, and moreover 
if $f\in \Df_0(\R_+)$ then
$\lim_{t\downarrow 0} f(t)=0$.
Furthermore, $\Dfp$ denotes the collection of $f\in \Df_0(\R_+)$ with no negative jumps, or equivalently, such that $f(t)\ge f(t-)$ for all $t$.  A strictly smaller class $\Dfuw$ contains all the non-decreasing $f\in \Df_0(\R_+)$.
Finally, $\Dfu$ denotes the collection of $f\in \Dfuw$, such that $f(t)>0$ for all $t>0$ and $f(t)\to\infty$ as $t\to\infty$. Equivalently, $f\in \Dfu$ if and only if
it is non-decreasing, strictly increasing from the right at $0$, and unbounded on $\R_+$. Given an $f\in \Df(\R_+)$ we will write
\begin{equation*}
\Jcal(f) = \{t: f(t)\neq f(t-)\}
\end{equation*} as the collection of jump times of a function $f$.
 
We will focus on ``well-behaved'' fields $\bbx\in \Df(\R^m_+)$ where
\begin{equation*}
    \Df(\R_+^m):= \left\{\bbx = (x_1,\dotsm,x_m)\,\left|\ \substack{\displaystyle x_i (\vec{t}) = \sum_{j=1}^m x_{i,j}(t_j) ,\\
    \displaystyle 
    x_{i,j}\in \Dfuw, \ \forall j\neq i \textup{ and }
    x_{i,i}\in \Dfp}, \ \forall i\in[m] \right.\, \right\}.
\end{equation*}

Assume that we are given some $\bbx=(x_i)_{i\in [m]} \in \Df(\R^m _+)$ and 
$\vec \rho \in \R^m_+$,  $\vec \rho \not= \vec 0$.
Recalling \eqref{D:Tdeterministic},
for each $y\geq 0$ we write $\bT^{\vec \rho}(\bbx;y)$ or $\bT(y)$ (when $\bbx,\vec \rho$ are specified by context) 
for $\bT(\bbx;y\vec\rho)$, which is the {\em minimal solution} of
\begin{equation}
  \label{equ_equation_for_t_ext}
  x_i(\vec t-)=-\rho_i y, \quad \forall i\in[m] \ \mbox{s.t.}\ t_i<\infty.
\end{equation}
{Abusing notation, we will shortly write~\eqref{equ_equation_for_t_ext} in the form
\begin{equation}
  \label{equ_equation_for_t}
  \bbx(\vec t-)=-\vec\rho y.
\end{equation}}
Note that $\bT^{\vec \rho}(\bbx;y) \in [0,\infty]^m$ is well-defined, according to~\cite[Lemma~2.3]{Chaumont:2020}. {Moreover, if $\bT^{\vec \rho}(\bbx;y) \in \R^m _+$, then also $\bT^{\vec \rho}(\bbx;u)\in\R^m _+$ for all $u \in [0,y)$, since in fact $\bT^{\vec \rho}(\bbx;u)\leq \bT^{\vec \rho}(\bbx;y)$.}
We will frequently omit the symbols $\vec\rho$ and $\bbx$ from the notation if they are clear from context. We also write the coordinates 
\begin{equation*}
    \bT(y) = (T_1(y),T_2(y),\dotsm, T_m(y)) \in[0,\infty]^m.
\end{equation*}

Since our goal is to prove Theorem \ref{thm:gammaExistence}, we also recall the additional assumption \eqref{equ_symmetry_assumption}. This could be called the ``column-wise off-diagonal proportionality'' (for probabilistic interpretations see Section \ref{S:commentsONmodels}): there is some $\vec{\rho}\in(0,\infty)^m$ such that
for each $l\in[m]$ and all $i,j \not= l$
\begin{equation}
\label{E:offdiagproportion}
\frac{ x_{i,l}(t) }{ \rho_i }=\frac{ x_{j,l}(t) }{ \rho_j },\  \text{ for all } t\geq 0.
\end{equation}
Under this additional hypothesis,
we can and will introduce the following notation
\begin{equation} 
  \label{D:xstari}
 x_{*,l}(t):= \frac{ x_{j,l}(t) }{ \rho_j }, \quad t\geq 0,
\end{equation}
where $j$ is any element of $[m]\setminus l$. Let 
$$\Dfmr \subset \Df(\R^m _+)$$ denote the collection of all fields $\bbx \in \Df(\R^m _+)$ which also satisfy \eqref{E:offdiagproportion}.

If $f$ is a real function of a real variable, let us denote by $ \underline f$ the ``past infimum'' of $f$:
  $$
\underline f(t)=\inf\limits_{ r \in [0,t] }f(r), \quad  t\geq 0.
$$

In the sequel we will often work with $\underline x_{i,i}$ 
instead of $x_{i,i}$ for $i\in [m]$. 
The main advantage of $\underline x_{i,i}$ over $x_{i,i}$ is its continuity (here we use the fact that $x_{i,i}\in D_0^+(\R_+)$) and monotonicity.
The following analogue of $x_i$
\begin{equation}
\label{equ_definition_of_underline_x_i}
  \underline x_i(\vec t):=\underline x_{i,i}(t_i)+\sum_{ j \not= i } x_{i,j}(t_j), \quad \vec t \in \R^m _+,
\end{equation}
will be particularly useful to us.
We naturally write $\underline \bbx$ for $(\underline x_1,\ldots, \underline x_m)$.

We will be interested in curves $\vec{\gamma}:\R_+\to \R_+^m$ such that $\gamma_i$ is non-decreasing for each $i\in[m]$. 
For $\vec l,\vec r \in \R^m $ we recall that $\vec l<\vec r$  (resp. $\vec l\leq \vec r$)  if $l_i<r_i$ (resp. $l_i\leq r_i$) for each $i \in [m]$. 
We also set $\vec r\pm\delta=(r_i\pm\delta)_{i \in [m]}$ for $\vec r \in \R^m $ and any $\delta >0$ a scalar. Let $(\vec l,\vec r)= \prod_{ i=1 }^{ m } (l_i,r_i)$ for $\vec l<\vec r$.

\subsection{Preliminary properties of \texorpdfstring{$\bT(y)$}{T(y)}} 
\label{S:preliminaty_properties_T}

Recall \eqref{equ_equation_for_t_ext} 
and \eqref{equ_equation_for_t}. In particular, $\bT(y)\equiv \bT^{\vec \rho}(\bbx;y) =(T_1(y),\ldots,T_m(y))$ is the (component-wise) minimizer of \eqref{equ_equation_for_t}.
Hence if ${\rm sol}(y):=\{ \vec r=(r_1,r_2,\ldots,r_m)\geq \vec 0: \vec r \text{ solves } \eqref{equ_equation_for_t}\}$ then 
\begin{equation}\label{E:comp_wise_min}
T_i(y) = \min_{\vec r \in {\rm sol}(y)} r_i, \quad  \forall i\in [m].
\end{equation}

It is easy to see that the original minimization problem (with equality) is solution equivalent to the one (with inequalities) where
the component-wise minimum is chosen from 
\begin{equation}
\label{E:setsolnsineq}
\{ \vec r=(r_1,r_2,\ldots,r_m)\geq \vec 0: \bbx(\vec r -)\leq -y \vec\rho\}
\end{equation}
instead of 
${\rm sol}(y)$.
Furthermore, let us denote by $\|\vec t \|_1=\sum_i t_i$ the usual $\ell^1$-norm of $\vec t$.
In the setting where all the components of $\bT(y)$ are finite, we have the following useful equivalence.
\begin{lemma}
\label{lem:minim_T} Provided that $\bT(y)\in\R_+^m$, the above optimization (minimization) problem for $\bT(y)$ is equivalent to 
\begin{equation}
\label{E:minim_T}
\left\{
\begin{array}{c}
 \bbx(\vec t-)=-\vec\rho y,\\
\|\vec t \|_1 \to \min.
\end{array}
\right. 
\end{equation}
\end{lemma}
\begin{proof}
Due to  ${\rm sol}(y)\subset [0,\infty)^m$ and
\eqref{E:comp_wise_min},
 the component-wise minimum over 
${\rm sol}(y)$ is also the vector in ${\rm sol}(y)$ 
which minimizes the $l_1$-norm.
\end{proof}

It is straight-forward from \eqref{equ_equation_for_t} that
$$
\frac{x_i(\bT(y)-)}{\rho_i}  = -y, \quad i \in [m],
$$
if $\bT(y)\in\R_+^m$.
Moreover, recalling \eqref{equ_definition_of_underline_x_i}, we get the following.
\begin{lemma} 
  \label{lem_minimal_solution_via_underline_x}
Let $\bT(y)$, $y\geq 0$, be as defined above and suppose $\bT(y)\in \R_+^m$.
Then, for each $y\geq 0$, $\bT(y)$ is
  \begin{enumerate}
    \item[i)]  the component-wise minimizer of
\begin{equation}
\label{equ_underline_x_for_T}
	  \underline \bbx(\vec t-)=-\vec\rho y ,
	\end{equation}
and the solution to
\begin{equation}
\label{E:minim_T_underline}
\left\{
\begin{array}{c}
 \underline \bbx(\vec t-)=-\vec\rho y,\\
\|\vec t \|_1 \to \min,
\end{array}
\right. 
\end{equation}
      \item[ii)]  $\underline x_{i,i}(t)>\underline x_{i,i}(T_i(y))$, for each $i \in [m]$ and every $t<T_i(y)$.
  \end{enumerate}
\end{lemma}

\begin{proof} 
The first statement is a simple consequence of the definitions and the hypotheses. It also follows from~\cite[Lemma~2.3.4]{Chaumont:2020}.

In order to show ii), suppose that for some $t^*< T_i(y)$ we have $\underline x_{i,i}(t^*)\leq\underline x_{i,i}(T_i(y))$. Then it is easy to see that the vector $\vec{t}^*:=(T_1(y),\ldots,T_{i-1}(y), t^*,T_{i+1}(y),\ldots, T_m(y))$, which is strictly smaller than $\bT(y)$ in component $i$, satisfies
$$
\underline \bbx(\vec{t}^*-) \leq \underline{\bbx}(\bT(y)-) = -\vec \rho y, 
$$
contradicting the minimality of $\bT(y)$ (in the minimizing problem \eqref{E:setsolnsineq}).
\end{proof}

\begin{lemma}
\label{lem_T_is_lcrl}
The map $y \mapsto T_i(y)$ is strictly increasing and left-continuous for each $i \in [m]$.
\end{lemma}
\begin{proof}
As commented already, it is clear from the properties of $\bbx$ and $\underline \bbx$ (inherited from those of $x_{i,j}$ as $i,j$ range through $[m]$) that $y\mapsto T_i(y)$ is non-decreasing for each $i$.

Suppose that $y'>y$, so that $-\vec \rho y' <  -\vec \rho y$.
Assuming there would be at least one 
$i \in [m]$ such that $T_i(y)=T_i(y')$, we could expand (using the non-strict monotonicity of $T_j$ for each $j\in [m]$, together with the monotonicity of $x_{i,j}$ for all $i\neq j$)
\begin{align*}
  -\rho_i y&= \underline x_i(\bT(y)-)=\underline x_{i,i}(T_i(y))+
\sum_{ j \not= i } x_{i,j}(T_j(y)-)\\
&=\underline x_{i,i}(T_i(y'))+
\sum_{ j \not= i } x_{i,j}(T_j(y)-)\\
  &\leq \underline x_{i,i}(T_i(y'))+\sum_{ j \not= i } x_{i,j}(T_j(y')-)=\underline x_i(\bT(y')-)=-\rho_i y',
\end{align*}
which leads to a contradiction.

Note that \cite[Lemma~2.3~4.]{Chaumont:2020} implies the left-continuity of $\bT$.
This can be verified directly by taking a sequence $y_n \nearrow y$, defining
 $\bT^*(y) := \lim_n \bT(y_n) = \sup_n \bT(y_n)$ and using the monotonicity of $y \mapsto \bT(y)$ to get $\bT^*(y) \leq \bT(y)$, and the minimality of $\bT(y)$ to get the reversed inequality.
\end{proof}
\begin{remark}
It is not surprising that $\bT$ is left-continuous, since it is an $m$-dimensional analogue of the left-continuous generalized inverse (the inequality in \eqref{E:setsolnsineq} is not strict).
Since $\bT$ is also (component-wise) strictly increasing, it is easy to see that it is in particular a non-decreasing left-continuous map with right limits ({\em ndlcrl} for short) function.
\end{remark}
Lemma \ref{lem_minimal_solution_via_underline_x} is quite helpful when solving for $\bT(y)$.
Indeed, since
 $\bT(y)$ is the solution of \eqref{E:minim_T_underline} we now also know that
$$
\frac{\underline x_i(\bT(y)-)}{\rho_i} = - y, \quad i \in [m],
$$
and, therefore, that $\bT(y)$ solves
\begin{equation}
\label{E:important_identity_one}
\left\{
\begin{array}{l}
\displaystyle\frac{\underline x_1(\vec t-)}{\rho_1} = \frac{\underline x_i(\vec t-)}{\rho_i}=- y,  \quad \forall i \in [m],\\\\
\|\vec t \|_1 \to \min.
\end{array}
\right. \,
\end{equation}

\subsection{Solving for \texorpdfstring{$\bT(y)$}{T(y)} -- Special Case}\label{sub:solving_for_T}
Recall \eqref{D:xstari} and the definition of $\Dfmr$.
In this section and in the next section 
we assume that  $\bbx \in \Dfmr$ for some given $\vec\rho$. Our aim is to construct
a continuous curve $\vec\gamma:[0,\infty) \to \R^m $  whose values will contain $(\bT^{\vec \rho}(\bbx;y),\, y\geq 0)$.
This will enable us to encode the jumps of $\bT^{\vec \rho}=\bT$ by considering the excursions above past minima of real-valued functions $x_i(\vec\gamma(s))$, $s\geq 0$, $i\in [m]$. 
We will arrive to a suitable choice of $\vec\gamma$ by trying to solve for $\bT(y)$.

For each $i\in [m]$ let us define
\begin{equation} 
  \label{equ_function_g}
  g_i(t):=x_{*,i}(t)-\frac{\underline x_{i,i}(t)}{\rho_i}, \quad t\geq 0.
\end{equation}
It is easy to check from the definitions
(of $\Dfmr$ in particular)
that $g_i\in \Dfuw$ for each $i\in [m]$.

After elementary algebraic manipulations (including several cancellations due to \eqref{equ_symmetry_assumption}) the optimization problem  \eqref{E:important_identity_one} 
can be replaced by
\begin{equation}
\label{E:important_identity_two}
\left\{
\begin{array}{l}
x_1(\vec t-)=- {\rho_1} y,\\
g_1(t_1-)=g_i(t_i-),  \quad \forall i \in [m],\\
\|\vec t \|_1 \to \min.
\end{array}
\right. 
\end{equation}
Indeed, the constraint $\underline x_1(\vec t-)=- {\rho_1} y$ can be replaced by $x_1(\vec t-)=- {\rho_1} y$ due to \cite[Lemma~2.3.4]{Chaumont:2020}, or alternatively due to the fact that the component-wise minimal solution of 
$$ x_1(\vec t-) = -\rho_1 y, \ \
\underline x_i (\vec t-) = -\rho_i z, \ \forall i \in [m], \text{ and some }z \geq 0, 
$$
is again necessarily equal to $\bT(y)$ (the parameter $z$ is free, and we can use monotonicity of $z\mapsto \bT(z)$).
At this intermediate step we know that $\bT(y)$ is uniquely determined as the solution of  \eqref{E:important_identity_two}.

In order to circumvent several technical issues, we presently make the following additional assumption (this hypothesis is dropped in Section \ref{sub:construction_of_a__special_curve} due to a novel concept and a considerable amount of additional work):
\begin{align*}
\label{H}
\tag{SC}
    x_{i,j}(t)\text{ are strictly increasing and continuous for all } i\neq j.
\end{align*} 
In particular, for each $i$, $t\mapsto g_i(t)$ is a strictly increasing continuous function such that $g_i(0) = 0$. We also suppose that $g_i(t)\to\infty$ as $t\to\infty$ for all $i$, which can also by stated by a more complicated assumption placed on the field $\bbx$.
Consequently, each $g_i$ is a homeomorphism, and we let $g_i^{-1}$ denote its inverse. 
\begin{remark}
Anticipating analysis in Section \ref{sub:construction_of_a__special_curve}, it is practical for us to continue writing $\vec{t}-$ instead of $\vec{t}$ in constraints involving $\bbx$.  
\end{remark}

The above strict monotonicity and continuity implies that for any given $y$
\begin{equation}
\label{A1}
\tag{A1}
g_i(T_i(y)-)=g_i(T_i(y)), \text{ or equivalently, that } \ x_{*,i}(T_i(y)-)=x_{*,i}(T_i(y)), \quad \forall i \in [m],
\end{equation}
and also  that
\begin{equation}
\tag{A2}
\label{A2}
 g_i^{-1} \circ g_i(T_i(y))=T_i(y) , \quad \forall i \in [m].
\end{equation}
Since $\bT(y)$ solves $g_1(t_1-) = g_i(t_i-)$ according to \eqref{E:important_identity_two}, the two properties \eqref{A1}--\eqref{A2} yield
\begin{equation}
\label{E:important_identity_three}
T_i(y)= g_i^{-1} \circ  g_1(T_1(y)).
\end{equation}
In particular, we see that 
$$
\bT(y)=(T_1(y),g_2^{-1}\circ g_1(T_1(y)),\ldots, g_m^{-1}\circ g_1(T_1(y))),
$$
where $T_1(y)$ is the minimizer of
$$
\left\{
\begin{array}{l}
x_1(t_1-, g_2^{-1} \circ g_1(t_1)-, \ldots, g_m^{-1} \circ g_1(t_1)-)=- {\rho_1} y,\\
\displaystyle t_1 + \sum_{i\neq 1} g_i^{-1} \circ  g_1(t_1)  \to \min.
\end{array}
\right. 
$$
This analysis can be improved by
introducing a convenient reparametrization of \eqref{E:important_identity_two}:

\begin{equation}
\label{E:important_identity_four}
\left\{
\begin{array}{l}
x_1(\vec t -)=- {\rho_1} y,\\
g_1(t_1-)=g_i(t_i-),  \quad \forall i \in [m],\\
s=\sum_{i=1}^m t_i,\\
s \to \min.
\end{array}
\right. 
\end{equation}
Define
\begin{equation} 
  \label{equ_f_kappa}
  f(u):=\sum_{ i=1 }^{ m } g_i^{-1}(u), \ u\geq 0, \ \text{ and }\ 
\kappa:=f^{-1},
 \end{equation}
where $\kappa$ is the usual inverse of the homeomorphism $f$.

Define
\begin{equation}
\label{D:s_of_y}
s(y):= \sum_{i=1}^m T_i(y) = \| \bT(y)\|_1,
\end{equation}
and note that  
$y \mapsto s(y)$ is again left-continuous (in fact it is a ndlcrl map), admitting at most countably many points of discontinuity. 
Recalling \eqref{E:important_identity_three}, we see that if $u = g_1(T_1(y)-)$ then $g_{i}^{-1}(u) = T_i(y)$ and so $f(u) = s(y)$, at least in the case when $\bT(y)\in\R_+^m$. Hence we deduce that $(\bT(y),s(y))$ solves
\begin{equation}
\label{E:important_identity_four_a}
\left\{
\begin{array}{l}
x_1(\vec t-)=- {\rho_1} y,\\
t_i=g_i^{-1}\circ g_1 (t_1),  \quad \forall i \in [m],\\
s= f(g_1(t_1)),\\
s \to {\rm min},
\end{array}
\right.
\end{equation}

Note that $f$ is strictly increasing, since the auxilliary functions $(g_i)_{i\in [m]}$ are assumed to be strictly increasing and continuous. Using \eqref{A1} and the reasoning of the previous paragraph we conclude {that $f(g_1(T_1(y)-)) = f(g_1(T_1(y))) = s(y)$. Applying the  inverse  $\kappa = f^{-1}$ to the last equation we see that $g_1(T_1(y)) = \kappa(s(y))$. It is now immediate from \eqref{A2} or from \eqref{E:important_identity_four_a} that}
$$
T_1(y)=g_1^{-1} \circ \kappa (s(y)) = g_1^{-1} \circ \kappa (\| T(y) \|_1),
$$
where $s(y)$ is the minimal solution of \eqref{E:important_identity_four_a}.
Note that there is nothing special with $i=1$, which we initially took in \eqref{E:important_identity_two} as the reference index. 
The same reasoning as above leads to the following conclusion: under assumptions 
\eqref{A1}--\eqref{A2} 
we get 
\begin{equation}
\label{E:T_via_s_sc}
\bT(y)=(\vec\gamma_{sc}(s(y))),
\end{equation}
where
$$
\vec\gamma_{sc}(s):=(g_1^{-1} \circ \kappa (s),\ldots, g_m^{-1} \circ \kappa (s )),
$$
and where $ s(y)$ is the minimal solution to
\begin{equation}
\label{E:s_via_x_j_circ}
x_j (\vec\gamma_{sc}(s)-) = -\rho_j y,
\end{equation}
for any (and every) $j\in [m]$. Here the subscript ``$sc$'' just stands for special case.

\section{Excursion representation: smooth composition and the general case}\label{sec:generalexcursionrep}

Assumptions \eqref{A1}--\eqref{A2} are cumbersome to check and false in general. {Indeed, \eqref{H} is  almost surely false for the fields constructed in \eqref{eqn:discField1}.}
We will soon construct a generalization of the ``composed with'' operator which proves quite useful in the sequel, and might be of general interest. As far as we are aware, this novel concept, which could be studied on the level of undergraduate calculus, has not yet appear in the literature.

\subsection{Smooth compositions}\label{S:non_decreasing_inver_comp}

Recall briefly the technical issues of our construction of $\vec \gamma_{sc}$ in the previous section. 
More precisely, these are the steps in the previous construction of the solution $\bT(y)$ which would fail (in the sense of mathematical rigor) without \eqref{H} or \eqref{A1}--\eqref{A2}. 

 In this section, we no longer assume \eqref{H}.
The auxiliary functions $(g_i)_{i\in [m]}$ are no longer homeomorphisms, however they are still non-decreasing and right-continuous functions.
In the next section we will assume that $g_i\in \Dfu$ for each $i\in [m]$.
\color{black}
We henceforth write \textit{ndrcll} for non-decreasing right-continuous functions with left-limits. 
From now on $g_i^{-1}$ will denote the \textit{generalized inverse} of $g_i$.
Recall that if $h$ is ndrcll, its (right-continuous generalized) inverse $h^{-1}$ is defined as follows
\begin{equation}
\label{D:gen_inv}
h^{-1}(s):= \inf\{u>0\,:\,h(u)>s\}.
\end{equation}

The construction of $\vec \gamma_{sc}$ relied on properties \eqref{A1} and \eqref{A2}. The former requires some smoothness, and the latter is trickier to mimic since $h^{-1}\circ h$ need not equal the identity. 
Our composition-like operator $\tcirc$ constructed below is such that
\begin{equation}\label{equ:inverse}
g\tcirc g^{-1} = g^{-1}\tcirc g = \textup{id}
\end{equation}
for each $g\in \Dfu$.

Let us recall a few basic useful inequalities, which relate a ndrcll map and its inverse:
  \begin{equation} 
  \label{equ_h_larger_u}
  h(u)>s \quad \implies \quad  h^{-1}(s) \leq u,
  \end{equation}
  \begin{equation}
  \label{equ_h_smaller_u}
  h(u)\leq s \quad \implies \quad  h^{-1}(s) \geq u,
  \end{equation}
  and by contraposition of \eqref{equ_h_larger_u}
  \begin{equation} 
  \label{equ_equ_h_inv_larger_s}
  h^{-1}(s)>u \quad \implies \quad h(u)\leq s.    
  \end{equation}
It is also easy to see that 
\begin{equation*}
    (h^{-1})^{-1} = h.
\end{equation*} One can see \cite[Chapter 0]{RY.99} for more information.

Before turning to the construction of  $\tcirc$, let us recall some standard properties of the inverse, most of which will be used without further mention in the sequel.
These are all elementary consequences of definitions and inequalities \eqref{equ_h_larger_u}--\eqref{equ_equ_h_inv_larger_s}, and their proofs are left to the reader. We recall that $\Jcal(f)$ is the collection of jump times for a function $f$.
 
\begin{lemma}
  \label{lem_properties_of_right_inverse}
  Let $h \in \Dfu$ and $h^{-1}$ be as in \eqref{D:gen_inv}. 
  \begin{enumerate}
    \item [i)] If $h$ is strictly increasing from the right at some $u\geq 0$, then $h^{-1}(h(u))=u$.

    \item [ii)] If $h(h^{-1}(s))=s$ for some $s\geq 0$, then $h^{-1}$ is strictly increasing from the right at $s$. 

  \item [iii)]  $h^{-1}(h(u-))\geq u$ and $h^{-1}(h(u)-)\leq u$ for all $u\geq 0$. 
In particular, if $h(u-)<h(u)$ (i.e. $u\in \Jcal(h)$) for some $u\geq 0$, then $h^{-1}(h(u-))=h^{-1}(h(u)-)=u$.

  \item [vi)] If $h(h^{-1}(s)-)=h(h^{-1}(s))$ for some $s\geq 0$, then $h(h^{-1}(s))=s$.
  \end{enumerate}
\end{lemma}

\begin{remark}
Clearly for $g\in \Dfu$, $g(u)>g(u-)$ (i.e. $u\in \Jcal(g)$) if and only if the inverse image $(g^{-1})^{-1}(\{u\})$ of $g^{-1}$ at $u$ is the (positive length) interval $[g(u-),g(u)]$, provided $g$ is strictly increasing from the right at $u$, or $[g(u-),g(u))$ otherwise.
Note that $g \circ g^{-1} = \id$ except on the union of  $(g^{-1})^{-1}(\{u\})$  (in some cases, taken without the right boundary point),
over all $u\in \Jcal(g)$.
The set of jump points for $g$ is at most countable.
However the above set of exceptions can be quite large.
Indeed, for ``pure jump'' functions $g \in \Dfu$ we have that $\bigcup_{u\in \Jcal(g)} (g^{-1})^{-1}(\{u\})=(0,\infty)$. 
It will be convenient to use below an example of such a function
\begin{equation}
\label{D:gtwo}
g_e(u):= \sum_{k=1}^\infty \frac{1}{k} \, 1_{[1/(k+1),1/k)} +\sum_{j=1}^\infty j \, 1_{[j,j+1)}.
\end{equation}
\end{remark}

Let $\DJfu$ consist of all $g \in \Dfu$ such that $g$ is both strictly increasing from the left and from the right at any jump point $u$ of $g$. If $g\in \DJfu$ has a jump at $u>0$ then $(g^{-1})^{-1}(\{u\}) = [g(u-),g(u)]$ and moreover 
$g \circ g^{-1} (s)=s$ for $s=g(u)$.
Therefore it is natural to define $g\tcirc g^{-1} \equiv g\circ g^{-1}=\id$ on the good set
$$
G(g):= \left(\bigcup_u \, [g(u-),g(u))\right)^c,
$$
and to define $g\tcirc g^{-1}$ as a {\em linear spline} on $G(g)^c$.

In particular $g\tcirc g^{-1}$ is defined on $[g(u-),g(u))$
by linearly interpolating through \sloppy $(g(u-),\lim_{\{v \nearrow g(u-): v\in G(g)\}} g\tcirc g^{-1}(v))=(g(u-),g(u-))$ and $(g(u), g\tcirc g^{-1}(g(u)))=(g(u), g(u))$. As a result we get $g\tcirc g^{-1}=\id$.

\begin{remark}
\label{rem:DJfu}
\begin{enumerate}
    \item[(a)] The good set $G(g)$ is dense in any left neighborhood of $g(u-)$, using the fact that $g$ jumps at $u$ and the assumption $g \in \DJfu$. Therefore, it is possible to take the left limit above through points in $G(g)$. 
    \vspace{-3mm}
    \item[(b)] We could have just set $g\tcirc g^{-1} := \id$ instead of going through the above ``construction by linear approximation'', but the point here is that (under mild and natural ``compatibility'' assumptions, see Definition \ref{D:compatible} below) the final step in the above construction can be repeated in the context where $g$ is an element of $\Dfu$, and where $g^{-1}$ is replaced by another function $\kappa\in \Dfu$.
\end{enumerate}
\end{remark}

Function $g_e$ defined in \eqref{D:gtwo} is an element of $\Dfu\setminus\DJfu$ and, moreover, it is not strictly increasing from the right at any of its (countably many) jumps. 
Furthermore, the good set $G(g_e)$ consists of a single point $0$. 
Consider any three of its consecutive jumps, these could be for example $u_i=i$, with $i=1,2,3$
(the conclusion is the same in general).
If $s_i:=g_e(u_i)$, then $g_e\circ g_e^{-1} (s_i)=g_e(u_{i+1}) = s_{i+1}$, for $i=1,2$.
One could be tempted to define of $g_e\tcirc g_e^{-1}$ 
 by linearly interpolating 
through the points $(s_i,g_e\circ g_e^{-1} (s_i))$ and $(s_{i+1}, g_e\circ g_e^{-1} (s_{i+1}))$
on each $[s_i,s_{i+1}]$.
The result is again a continuous function, but clearly different from the identity map.

However, one could extend $\tcirc $ 
in a different and better way.

\begin{definition}
\label{D:compatible}
We say that two elements $g$ and $\kappa$ of $\Dfu$ are {\it compatible} if 
\begin{equation}
\label{H1}\tag{H1}
\kappa^{-1}(\{ u \})= \{ r\geq 0:\ \kappa(r)=u \} 
  \text{ has positive length whenever } u\in \Jcal(g),
\end{equation}
\begin{equation}
\label{H2}\tag{H2}
\text{ and }  g(\kappa(s-))=g(\kappa(s)-) 
      \text{ whenever } \kappa(s-)<\kappa(s).
\end{equation}
\end{definition}

Assumption \eqref{H1} means precisely that the inverse $\kappa^{-1}$ of $\kappa$ also jumps at each jump point $u$ of $g$. In this case, the segment $\kappa^{-1}(\{ u \})$ equals $[\kappa^{-1}(u-),\kappa^{-1}(u)]$ (resp. $[\kappa^{-1}(u-),\kappa^{-1}(u))$) provided $\kappa^{-1}$ is strictly increasing from the right at $u$ (resp.~constant on $[u, u+\delta)$ for some $\delta>0$).

We proceed by analogy to the construction of $g\tcirc g^{-1}$ for $g\in \DJfu$.
\begin{definition}\label{def_tilde_compose} Suppose that the pair of functions $(g,\kappa)$ satisfies \eqref{H1} and \eqref{H2}. Define a new function $g\tcirc \kappa$ by 
\begin{itemize}
\item[(i)]
if $s\not\in \bigcup_{u\in \Jcal(g)}  [\kappa^{-1}(u-), \kappa^{-1}(u)]$ let $g\tcirc \kappa(s)=g \circ \kappa (s)$,
\item[(ii)]
 for each jump point $u\in \Jcal(g)$, on $[\kappa^{-1}(u-), \kappa^{-1}(u)]$ 
define $g\tcirc \kappa$ 
to be the line segment through the points $(\kappa^{-1}(u-), g(u-))$ and $(\kappa^{-1}(u), g(u))$.
\end{itemize}
\end{definition}

From now on we refer to $\tcirc$ as ``smoothly composed with''. 
As discussed immediately after Remark \ref{rem:DJfu}, there are situations where 
``smoothly composed with'' and ``composed with'' differ on $(0,\infty)$, so 
$\,\tilde\circ\,$ is not a direct extension/generalization of $\circ$. If $\kappa=g^{-1}$ then it is easy to see that both \eqref{H1} and \eqref{H2} are satisfied and moreover that
$g\tcirc \kappa= g\tcirc g^{-1}=\id$.

\begin{figure}[h]
    \centering
    \includegraphics[width=.95\textwidth]{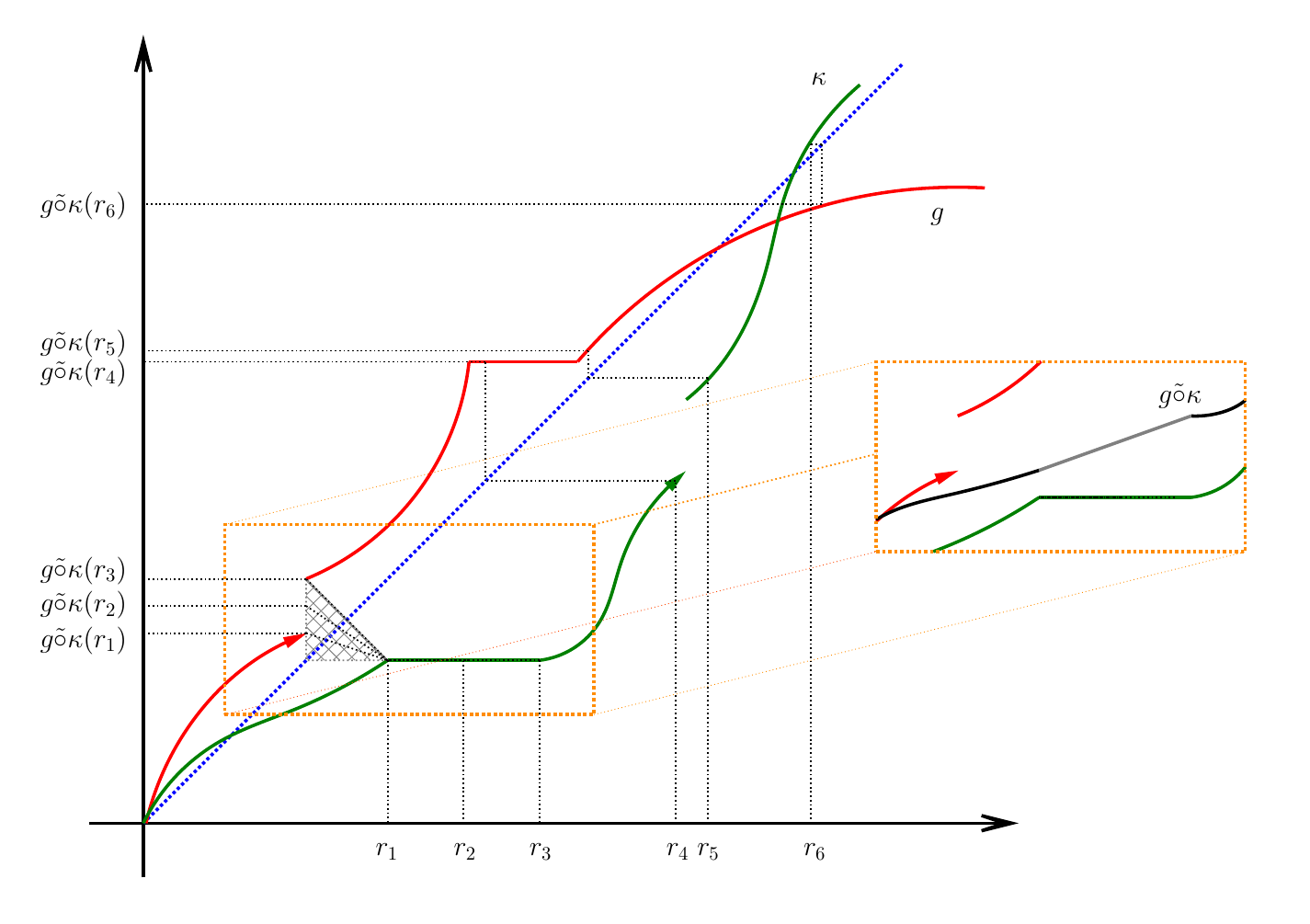}
    \caption{The red curve is the graph of $g$, while the green curve is the graph of $\kappa$. The rectangle with dotted orange borders is an interesting region, since it includes a jump $u$ of $g$, as well as the image of $\kappa$ on an interval which includes $[r_1, r_3]=[\kappa^{-1}(u-),\kappa^{-1}(u)]$. Its enlarged copy on the right depicts the graph of $\gamma = g\tcirc \kappa$ in blue.
    In particular, $\gamma$ is a continuous (linear interpolation) function on $[r_1,r_3]$. Also note that (since $g$ and $\kappa$ are compatible) $g$ must be constant on $[\kappa(s-),\kappa(s)]$, where $s$ is the jump of $\kappa$ in $[r_4,r_5]$.
    } 
    \label{fig:compositions}
\end{figure}

\subsubsection{Graphs of \texorpdfstring{$g\circ \kappa$}{gK} and \texorpdfstring{$g\tilde\circ\kappa$}{gK}}
\label{para:compositions}    

Suppose that $(g,\kappa)$ is a pair of compatible functions in $\Dfu$.
Let us assume that the graphs of $g$ and $\kappa$ are drawn in the same Cartesian system.
(see Figure \ref{fig:compositions} for an illustration). 
    Recall that in order to construct $g\circ\kappa (t)$ using one starts with the point $(t,0)$ on the abscissa, searches for $(t,\kappa(t))$ on the graph of $\kappa$, from there moves horizontally to  $(\kappa(t),\kappa(t))$ on the diagonal $y=x$;
    and finally searches along the vertical line $x=\kappa(t)$ for the point   $(\kappa(t),g(\kappa(t))$ on the graph of $g$.
    This was done in the constructions of  $g\tcirc\kappa(r_k) = g\circ \kappa(r_k)$ for $k\in\{4,5,6\}$ 
    in Figure \ref{fig:compositions} (at these points $\kappa$ is strictly increasing and continuous).

\noindent
     One can analogously construct $g\tcirc\kappa(r_j)$ for $j\in\{1,2,3\}$, knowing that
      $[\kappa^{-1}(u-),\kappa^{-1}(u)] = [r_1,r_3]$ for some jump point $u = \kappa(r_j)$ of $g$. 
      In this general setting the final vertical move from $(u,u) (= (\kappa(r_j),\kappa(r_j))$ depends on $j$, or more precisely on the position of $r_j$ within $[r_1, r_3]$. 
      Since
      $r_1=\min\{s:\kappa(s)=u\}$,
     the procedure ends at $(\kappa(r_1),g(u-))$, the ``lowest point''  with abscissa $u$ in the closure of the graph of $g$.
    Since $r_3=\max\{s:\kappa(s)=u\}$,
    the procedure ends at $(\kappa(r_3),g(u))$, the ``highest point''  with abscissa $u$ in the closure of the graph of $g$.
    In general, any $r_2\in[r_1,r_3]$ has representation $ r_2(\lambda)= r_1+\lambda(r_3-r_1)$ for some $\lambda \in [0,1]$.
    Given such $\lambda$, the procedure ends at 
    $(\kappa(r_2),g(u-) +\lambda (g(u) - g(u-))$.
    
\color{black}
    
\begin{remark}\label{rem_interpolation_formula_for_tilda_circ} Our proofs below do not rely on an explicit formula for $g\tcirc \kappa$; however, we include it here for readers' benefit. Note that the family of sets $\{\kappa^{-1}(\{ u \}),\ u\geq 0\}$ is a partition of $\R_+$. Then for every $s,u\in\R_+$, such that $s \in \kappa^{-1}(\{ u \})$
\begin{align*}
  (g\tcirc \kappa(s))&= g(u-) + \frac{ g(u)-g(u-) }{ \kappa^{-1}(u)-\kappa^{-1}(u-)}(s-\kappa^{-1}(u-))\\
  &= g(u) - \frac{ g(u)-g(u-) }{ \kappa^{-1}(u)-\kappa^{-1}(u-)}(\kappa^{-1}(u)-s).
\end{align*} 
In the expression above, we assume that $A:=\frac{ g(u)-g(u-) }{ \kappa^{-1}(u)-\kappa^{-1}(u-)}$ equals zero if $g$ has no jump at $u$. 
Due to~\eqref{H1} $\kappa^{-1}(u)-\kappa^{-1}(u-)>0$ if $u\in \Jcal(g)$, so $A$ is always well defined.

Note in addition that it could be $\kappa^{-1}(v)-\kappa^{-1}(v-)>0$  even if $g(v)=g(v-)$, and then  the restriction of 
$g\tilde\circ\kappa$ on $[\kappa^{-1}(v-),\kappa^{-1}(v))$ is again a linear (constant) function $y(s)\equiv g \circ \kappa(v)$.

\end{remark}

Even though the above remark is not necessary for the following lemma, using it makes the following straightforward:
\begin{lemma}
\label{lem:circ_tilde_circ_compare}
   Suppose that $(g,\kappa)$ is a compatible pair. Then $g(\kappa(t)-)\leq g\tilde{\circ}\kappa(t) \le g(\kappa(t))$ for all $t\geq 0$.
\end{lemma}

\begin{remark}
\label{rem:positiveslopeetc}
\begin{enumerate}
    \item[(a)] Note that each line segment in (ii) of Definition~\ref{def_tilde_compose} is increasing (its slope,  denoted by $A$ in Remark \ref{rem_interpolation_formula_for_tilda_circ}, is strictly positive).
    \vspace{-3mm}
    \item[(b)] If $v<u$ is another jump of $g$ (and therefore of $\kappa^{-1}$) then $g(v-)<g(v)\leq g(u-)< g(u)$.
We conclude that $g\tcirc\kappa$ restricted to 
 $\bigcup_{u: \Jcal(g)}  [\kappa^{-1}(u-), \kappa^{-1}(u)]$
is a (strictly) increasing function.
\vspace{-3mm}
    \item[(c)] From \eqref{H1} and the right-continuity of $\kappa$ at $\kappa^{-1}(u-)$ we conclude that at each jump point $u$ of $g$ it must be $\kappa(\kappa^{-1}(u-)) =u$, and so $g \circ \kappa(\kappa^{-1}(u-))=g(u)$.
    \vspace{-3mm}
    \item[(d)] The definition above does not rely on any additional properties of $\kappa$. In particular, it could be that $\kappa(\kappa^{-1}(u))> u$, and also that $g(\kappa(\kappa^{-1}(u)))> g(u)$. Similarly, it may not be possible to access $\kappa^{-1}(u-)$ from the left through the good set of points in Definition \ref{def_tilde_compose} (i).
\end{enumerate}

\end{remark}

\begin{remark}
\label{rem:consecjumps}
Condition \eqref{H1} is necessary for the interpolating line in Definition \ref{def_tilde_compose} (ii) to be well-defined.
Also note that it could happen that for two successive jumps  of $g$, occurring respectively at $u_1$ and $u_2>u_1$, we have $\kappa^{-1}(u_1)=\kappa^{-1}(u_2-)$. This means that $\kappa^{-1}$ is a constant function on $[u_1,u_2)$, and also that $\kappa$ jumps at
point $\kappa^{-1}(u_1)=\kappa^{-1}(u_2-)$ from value $u_1=\kappa( \kappa^{-1}(u_1)-)=\kappa( \kappa^{-1}(u_2-)-)$ to value 
$u_2= \kappa( \kappa^{-1}(u_1))=\kappa( \kappa^{-1}(u_2-))$.
Assumption \eqref{H2} with $s= \kappa^{-1}(u_1)$ is necessary for compatibility of the two different definitions of $g\tcirc \kappa$ at the point $\kappa^{-1}(u_1)=\kappa^{-1}(u_2-)$.
\end{remark}

\begin{lemma}
  \label{lem_first_property_of_tilde_circ}
 Let $g,\kappa$ be elements of $\Dfu$ such that $(g,\kappa)$ satisfies \eqref{H1} and \eqref{H2}. Then, provided that $u\in \Jcal(g)$ is a jump point of $g$,\\
(a) $g \tcirc  \kappa(\kappa^{-1}(u-)) =g(u-) \geq g \tcirc  \kappa (s)$,\, for all $s<\kappa^{-1}(u-)$,\\
(b) $g \tcirc  \kappa(\kappa^{-1}(u)) =g(u) \leq g \tcirc  \kappa (s)$,\, for all $s>\kappa^{-1}(u)$.
\end{lemma}
In the argument below we will use several times and without explicit mention the fact that $g$ (and $\kappa$) is ndrcll.
\begin{proof}
\noindent
(a) 
If $s < \kappa^{-1}(u-)$ then necessarily $\kappa(s) < u$ so that $g\circ \kappa(s) \leq g(u-)$, and this implies the claim provided $s$ is a ``good point'' from Definition \ref{def_tilde_compose} (i).
Otherwise, it must be $s \in [\kappa^{-1}(v-),\kappa^{-1}(v)]$ for some $v<u$ and then $g \tcirc  \kappa(s) \leq g(v)\leq g(u-)$.

\noindent
(b) If $s>\kappa^{-1}(u)$, then necessarily $\kappa(s)> u$ so that $g\circ \kappa(s) \geq g(u)$, and this again implies the claim provided $s$ is a ``good point''. 
Otherwise, it must be $s \in [\kappa^{-1}(v-),\kappa^{-1}(v)]$ for some $v>u$ and then $g \tcirc  \kappa(s) \geq g(v-)\geq g(u)$.
\end{proof}

The above partial monotonicity result can be easily improved as follows. 
\color{black}

\begin{lemma} 
  \label{lem_tilde_circ_ndrcll}
  Let $g,\kappa$ be as in Lemma \ref{lem_first_property_of_tilde_circ}.
  Then $\gamma=g \tcirc \kappa$ is ndrcll.
\end{lemma}

\begin{proof}
Since $g$ and $\kappa$ are elements of $\Dfu$, the same is true for $g \circ \kappa$.
To verify that $\gamma=g \tcirc \kappa$ is monotone non-decreasing, we use Lemmas \ref{lem:circ_tilde_circ_compare}, \ref{lem_first_property_of_tilde_circ}
when comparing the values of $\gamma(s_1)$ and $\gamma(s_2)$ if at least one, $s_1$ or $s_2$, is not a ``good point'' from Definition \ref{def_tilde_compose} (i), and we use the monotonoicity of $g\circ\kappa$ when 
both $s_1$ and $s_2$ are ``good points''. One can also try to derive this using Remark \ref{rem_interpolation_formula_for_tilda_circ} and a comparison argument.
Monotonicity implies that $\gamma$ has both limits from the right and from the left at every point.

We will next show the right continuity of $\gamma$. By the construction
of $\gamma$, it is trivial that $\gamma$ is right-continuous on
every interval $[\kappa^{-1}(u-),\kappa^{-1}(u))$, $u\ge0$ (see also Remark~\ref{rem_interpolation_formula_for_tilda_circ}). It remains to show that $\gamma$ is right continuous on $A:=\left(\bigcup_{u\ge0}[\kappa^{-1}(u-),\kappa^{-1}(u))\right)^{c}$. 

Let us fix $s\in A$. Then there exists $u\ge 0$ such that $s=\sup\kappa^{-1}(\{u\})$,
if and only if $s$ is the right limit of a positive length interval $\left[\kappa^{-1}(u-),\kappa^{-1}(u)\right)$,
 or there exists $ u'\ge 0$ such that $\{s\}  = \kappa^{-1}(\{u'\})$. If both happen, 
then $\kappa$ jumps at $s$ from $u= \kappa(s-)$ to $u'=\kappa(s)$, and  $s$ is the right end point of
$\left[\kappa^{-1}(u-),\kappa^{-1}(u)\right)$.

There are two cases to consider. In the first case we have  $\kappa^{-1}(\{u\})=\left[\kappa^{-1}(u-),\kappa^{-1}(u)\right]$, which includes the situation where
$\kappa^{-1}(\{u\}) = \{s\}$ is a one-point set. The right continuity of $\kappa$ at $s=\kappa^{-1}(u)$ implies existence of a sequence 
\begin{equation}
\{u_{n},n\ge1\}\subset\mathrm{Im}\text{\ensuremath{\kappa}}\ \mbox{such that}\ u_{n}\downarrow u.\label{eq:assumption_on_st_inc}
\end{equation}
In particular, for each $n$ we may choose $s_{n}\in\kappa^{-1}(\{u_{n}\})$. 
Note that $(s_{n})_{n\ge1}\subset (s,\infty)$ is  strictly decreasing. Recalling the definition of $\gamma$ and its monotonicity, together with  Lemma \ref{lem:circ_tilde_circ_compare} and the right continuity of $g$, we have
\[
g(u)=\gamma(s)\le\lim_{n\to\infty}\gamma(s_{n})=\lim_{n\to\infty}g\tilde{\circ}\kappa(s_{n})\le\lim_{n\to\infty}g(u_{n})=g(u).
\]
Due to the already established monotonicity of $\gamma$,
the equality
$\gamma(s)=\lim_{n\to\infty}\gamma(s_{n})$
is sufficient for concluding that $\gamma$ is right continuous  at $s$.

In the second case we have $\kappa^{-1}(\{u\})=\left[\kappa^{-1}(u-),\kappa^{-1}(u)\right)$.  In particular, 
$s= \kappa^{-1}(u)$ but
$\kappa(s)= u'>u=\kappa(s-)$, so  \eqref{eq:assumption_on_st_inc} is false. 
Recalling the definition of $\gamma$ and the assumption~\eqref{H2}, we have
\[
\gamma(s)=g(u)=g(\kappa(s-))=g(\kappa(s)-)=g(u'-).
\]
Thus, if $\left[\kappa^{-1}(u'-),\kappa^{-1}(u')\right)$
has positive length, we obtain the (right-)continuity of $\gamma$ at $s$ from the reasoning of Remark \ref{rem:consecjumps}. Otherwise, $\kappa^{-1}(\{u'\})$
is a single point. This implies (due to~\eqref{H1}) that 
$g(u'-)=g(u')$ and moreover that
$\gamma(s)=g\circ \kappa(s) = g(u')=g(u'-)$.
In this case we can again use \eqref{eq:assumption_on_st_inc}
with $u$ replaced by $u'$, and the subsequent reasoning to obtain the right continuity
of $\gamma$ at $s$.
\end{proof}

\begin{lemma}
  \label{lem_continuity_of_tilde_circ}
  Let $g,\kappa$ be as in Lemmas \ref{lem_first_property_of_tilde_circ} and \ref{lem_tilde_circ_ndrcll}.
  Then $\gamma=g \tcirc \kappa \in \Dfu$ is a continuous function.
\end{lemma}
\begin{proof}
Both $g(0)=0$ and $\kappa(0)=0$ so $g \circ \kappa(0) =g \tcirc  \kappa(0) =\gamma(0) = 0$. Lemma \ref{lem:circ_tilde_circ_compare} directly implies that $g\tcirc \kappa$ is both unbounded and strictly increasing at $0$.
Lemma \ref{lem_tilde_circ_ndrcll} gives monotonicity. Therefore $\gamma\in \Dfu$. 
\color{black}

We next fix $s>0$ and show that $\gamma$ is left-continuous at $s$.
Set $u:=\kappa(s)$ and note that if $s$ belongs to the interior
of $\kappa^{-1}(\{u\})$, which  equals $(\kappa^{-1}(u-),\kappa^{-1}(u))$ if non-empty,
then the left-continuity of $\gamma$ at $s$  follows directly from the construction
of $\tcirc$. 
Moreover, if $(\kappa^{-1}(u-),\kappa^{-1}(u))$ is non-empty then $\gamma$ is left-continuous at $s=\kappa^{-1}(u)$ due to the same observation.

We now assume that $s=\kappa^{-1}(u-)\leq \kappa^{-1}(u)$ and consider a sequence $(s_{n})_n$ of positive real numbers, which strictly increases to $s$.
Our goal is to show that $\lim_n \gamma(s_{n})=\gamma(s)$, as this  together with the monontonicity of $\gamma$ will imply left-continuity at $s$. We set $u_{n}:=\kappa(s_{n})$, $n\ge1$, and $\bar{u}:=\lim_{n\to\infty}u_{n}$.
Then note that $\kappa(s-)=\bar{u}\leq u=\kappa(s)$
and $\gamma(s)=g(u-)$ for our choice of $s$. Using the definition and monotonicity of $\gamma$ we have
\[
g(u_{n}-)\le\gamma(s_{n})\le\gamma(s).
\]
Passing to the limit as $n\to\infty$  we obtain
\begin{equation}
\label{E:ggammagamma}
g(\bar{u}-)\le\gamma(s-)\le\gamma(s)=g(u-).
\end{equation}
If $\bar{u}=u$ then clearly $\gamma(s)=\gamma(s-)=g(u-)$.
Otherwise, if $\bar{u}<u$, then~\eqref{H2} implies
\[
g(u-)=g(\kappa(s)-)=g(\kappa(s-))=g(\bar{u}).
\]
Consequently, if $g$ is continuous at $\bar{u}$, then again $\gamma(s-)=\gamma(s)$
due to \eqref{E:ggammagamma}. Otherwise, $g(\bar{u}-)<g(\bar{u})$
and therefore $\kappa^{-1}(\{\bar{u}\})$ has positive length
due to~\eqref{H1}. 
Since here $s= \kappa^{-1}(u-) = \kappa^{-1}(\bar{u})$, we have already covered this case above. This completes the proof of the
lemma.
\end{proof}

\begin{remark}
\label{rem:spline_options}
 The exact expression (formula) for the linear spline on conveniently chosen intervals $[\kappa^{-1}(u-),\kappa^{-1}(u)]$ was not important for the arguments above.
The linear interpolation is the simplest, and it is compatible with taking inverses. However another continuous or differentiable increasing interpolation with compatible boundary values would equally imply an analogue of Lemma \ref{lem_continuity_of_tilde_circ}.
Nevertheless, the following additivity result, necessary in the proof of an important property \eqref{P3} in Lemma \ref{lem_properties_of_kappa_and_gamma}, requires the spline to be linear.
\end{remark}

The following lemma is an easy consequence of our choice of linear spline. We leave the details to the reader.
\begin{lemma}
\label{lem_linearity_of_tilda_circle}
Let $h_1,h_2,\kappa\in \Dfu$ be such that both $(h_1,\kappa)$ and $(h_2,\kappa)$
satisfy \eqref{H1} and \eqref{H2}. Then $(h_1+h_2,\kappa)$ also satisfies \eqref{H1} and \eqref{H2}, and moreover 
\begin{equation*}
    h_1\tcirc  \kappa + h_2\tcirc \kappa = (h_1+ h_2)\tcirc \kappa.
\end{equation*}
\end{lemma}

\subsection{Solving for \texorpdfstring{$\bT(y)$}{T(y)} -- the general case}
\label{sub:construction_of_a__special_curve}
Inspired by the  analysis of the previous two sections, we now derive the general expression for the minimal solution to  \eqref{equ_equation_for_t} in terms of a solution to a $1$-dimensional optimization problem.
It turns out that $\circ$ can be replaced with  $\tcirc $ in \eqref{E:T_via_s_sc}, but arguing this rigorously is not as simple as one might guess.

\begin{lemma} 
  \label{lem_construction_of_a_curve_gamma}
  Let $g_i \in \Dfu$, $i \in [m]$, and $f,\kappa$ be defined by~\eqref{equ_f_kappa} using inverses. Then  for each $i \in [m]$ both~\eqref{H1} and~\eqref{H2} are satisfied for $(g_i^{-1},\kappa)$.
\end{lemma}

\begin{proof} 
 Fix some $i\in [m]$.
 Let us check \eqref{H1}.
 Suppose that $g^{-1}_i(u-)<g^{-1}_i(u)$ for some $u>0$. 
 Then in particular $f(u-)<f(u)$. 
 Since $\kappa$ is the (generalized right-continuous) inverse of $f$, it is also true 
that $f=\kappa^{-1}$ and in particular that
  \begin{equation} 
  \label{equ_property_of_f_and_kappa}
  \kappa(s)=u, \quad \mbox{for all}\ s \in [f(u-),f(u))=[\kappa^{-1}(u-),\kappa^{-1}(u)).
  \end{equation}

Let us now check \eqref{H2}.
Suppose that $\kappa(s-)<\kappa(s)$ for some $s>0$. Then it must be (see also Lemma~\ref{lem_properties_of_right_inverse}~iii)) that $f(u)=s$ for all $u \in [\kappa(s-),\kappa(s))=[f^{-1}(s-),f^{-1}(s))$. 
Recalling that 
$g_i^{-1}$, $i \in [m]$, are all non-decreasing, their sum being constant on any interval 
implies that each of them is constant on the same interval. So $g_i^{-1}(u)= g_i^{-1}(f^{-1}(s-))= g_i^{-1}(\kappa(s-))$ for all $u\in [\kappa(s-),\kappa(s))$, and this is equivalent to \eqref{H2} for $(g_i^{-1},\kappa)$.
\end{proof}

 Recall that all the {inverses are} considered to be {generalized} right-continuous inverses. Let $g_i \in \Dfu$, $i \in [m]$, and let $\kappa$ be defined by~\eqref{equ_f_kappa}.
Furthermore, define
\begin{equation} 
  \label{equ_definition_of_gamma_i}
  \gamma_i:=g^{-1}_i \tcirc \kappa,
\end{equation}
and $s\mapsto \vec\gamma(s)$ by
\begin{equation}
\label{D:gamma}
\vec\gamma(s):=(\gamma_1(s),\ldots,\gamma_m(s)), \quad s\geq 0.
\end{equation}
  Note that the next two results are stated and proved in greater generality, although we will apply them only in the setting where $g_i$ are given by \eqref{equ_function_g}.
\begin{lemma}
  \label{lem_properties_of_kappa_and_gamma}
  \begin{enumerate}
    \item [(\namedlabel{P1}{P1})] The curve $\vec\gamma$ is a continuous curve in $\R^m $,

    \item [(\namedlabel{P2}{P2})] for each $i$,  $\gamma_i$ is non-decreasing,

  \item [(\namedlabel{P3}{P3})] 
  $s \mapsto \|\vec\gamma(s)\|_1$ is the identity map on $\R_+$,

\item [(\namedlabel{P4}{P4})] for each $s\geq 0$
\[
  \kappa(s) \in \bigcap_{ i=1 }^{ m } \left[g_i(\gamma_i(s)-),g_i(\gamma_i(s))\right].
\]
In particular,  if $s\geq 0$ is such that $g_i$ is continuous at $\gamma_i(s)$ for each $i \in [m]$,
then 
$$
g_i(\gamma_i(s))=g_{i+1}(\gamma_{i+1}(s))(=\kappa(s)), \ \forall i \in [m-1].
$$
\end{enumerate}
\end{lemma}
\begin{proof} 
Properties~\eqref{P1} and~\eqref{P2} are both clearly satisfied due to Lemmas~\ref{lem_tilde_circ_ndrcll},~\ref{lem_continuity_of_tilde_circ} and~\ref{lem_construction_of_a_curve_gamma}.

  Let us show~\eqref{P3}. 
Recalling that $f=\sum_{ i=1 }^{ m } g^{-1}_i$, we have
  \[
    \|\vec\gamma(s)\|_1= \sum_{ i=1 }^{ m } \left(g_i^{-1}\tcirc \kappa\right)(s)=(f \tcirc \kappa)(s) =s, \quad s \geq 0,
  \]
where the second identity is due to Lemma~\ref{lem_linearity_of_tilda_circle}, and the third one is due to $f=\kappa^{-1}$ (see \eqref{equ:inverse} and Remark \ref{rem:spline_options}).

  Let us next verify~\eqref{P4}. Fix some $i \in [m]$ and recall that, due to the construction of $\tcirc$, if 
$\kappa(s)=u$ is not a jump point of $g_i^{-1}$ then 
\[
\gamma_i(s) = g_i^{-1} \circ \kappa(s) =  g_i^{-1} (u) = g_i^{-1} (u-).
\]
Applying $g_i$ to the above identity to get that
$
g_i(\gamma_i(s) )=g_i (g_i^{-1} \circ \kappa(s) ) \geq \kappa(s),
$
and also that (see also Lemma \ref{lem_properties_of_right_inverse} part iii)) 
$
g_i(\gamma_i(s)-) = g_i(g_i^{-1} (u)-)\leq u = \kappa(s).
$

Alternatively, if $\kappa(s)=u$ is a jump point of $g_i^{-1}$,
then $s \in [\kappa^{-1}(u-), \kappa^{-1}(u)]=[f(u-),f(u)]$,
and $\gamma_i= g_i^{-1}\tcirc \kappa$
is specified at $s$ by linearly interpolating 
between
$(f(u-),g_i^{-1}(u-))$ and $(f(u),  g_i^{-1}(u))$. Therefore
$
g_i^{-1}(u-) \leq \gamma_i(s) \leq g_i^{-1}(u),
$
We can now use the monotonicity of $g_i$, together with 
 the RHS (resp.~LHS) inequality to get
 \begin{align*}
&g_i(\gamma_i(s)-) \leq g_i(g_i^{-1}(u) -) \le u=\kappa(s),\\
 \big(\text{resp.~}\ &\kappa(s)=u \le g_i(g_i^{-1}(u-) ) \leq g_i(\gamma_i(s))\big),
\end{align*}
where we used Lemma~\ref{lem_properties_of_right_inverse} iii) for both estimates.
\end{proof}

\begin{proposition}
\label{pro_properties_of_kappa_and_gamma}
Let  $\vec g,f,\kappa,\vec\gamma$ be as in Lemma \ref{lem_properties_of_kappa_and_gamma},
and suppose that we are given some $\vec t=(t_1,\ldots,t_m) \in \R^m _+$ such that $g_i(t_i)=g_{i+1}(t_{i+1})$, $i \in [m-1]$. 
Then 
 $$\vec t=\vec\gamma(\|\vec t\|_1),$$
provided that
\begin{enumerate}
\item [(\namedlabel{P5}{a})] for each $i \in [m]$, $g_i$ is continuous at $t_i$ and strictly increasing from the left at $t_i$, or 
\item [(\namedlabel{P5'}{b})]  for each $i \in [m]$, $g_i$ is strictly increasing from the right at $t_i$.
  \end{enumerate}
\end{proposition}
  \begin{proof}
Let $\vec t\,$ be given as stated. 
Define
$s:=\|\vec t\|_1$ and  $\vec t'=(t_1',\ldots, t_m'):=\vec\gamma(s)$.
Our goal is to prove that $\vec t'=\vec t$ 
if \eqref{P5} or \eqref{P5'} (or both).

Note initially that \eqref{P3} implies that $ \|\vec t'\|_1 = s = \|\vec t\|_1$.
So if $\vec t\neq \vec t'$ then there must exist some (minimal) index $i$ such that both
 $t_i>t_i'$ and $t_{i+1}\leq t_{i+1}'$, otherwise $t_m>t_m'$ and $t_1\leq t_1'$.
Assuming \eqref{P5}, we would get from \eqref{P4} and monotonicity of $g_i$ that 
$$
g_i(t_i)>g_i(t_i')\geq \kappa(s) \geq g_{i+1}(t_{i+1}'-) \geq g_{i+1}(t_{i+1}-)=g_{i+1}(t_{i+1})=g_i(t_i),
$$
a contradiction.
For the same reason as above, $\vec t\neq \vec t'$ implies that there must exist some (minimal) index $j$ such that both
 $t_j<t_j'$ and $t_{j+1}\geq t_{j+1}'$, otherwise $t_m<t_m'$ and $t_1\geq t_1'$.
Assuming \eqref{P5'}, we would get from \eqref{P4} and monotonicity of $g_j$ that 
$$
g_j(t_j)<g_j(t_j'-)\leq \kappa(s) \leq g_{j+1}(t_{j+1}') \leq g_{j+1}(t_{j+1})=g_j(t_j),
$$
which is again impossible.\\
If the the discrepancy is at $m$ and $1$ instead of at $i$ and $i+1$ (resp.~$j$ and $j+1$), we would  analogously arrive to a contradiction in the above argument under assumption (a) (resp.~(b)). 
\end{proof}

Recall that $\bT(y)= \bT^{\vec \rho}(\bbx;y)$ is the minimal solution of \eqref{equ_equation_for_t}.
Equivalently, $\bT(y)$ is the solution of
\eqref{E:important_identity_two} and $(\bT(y),s(y))$ (where
$s(y)$ is defined in \eqref{D:s_of_y}) is the solution of \eqref{E:important_identity_four}.

The following theorem establishes Theorem \ref{thm:gammaExistence}(2).

\begin{theorem} 
  \label{the_t_and_gamma}
  Let $\vec\gamma$ be the continuous curve defined in~\eqref{equ_definition_of_gamma_i},
where $g_i$, $i \in [m]$, are given by~\eqref{equ_function_g}. Then 
  \[
    \vec\gamma(\|\bT(y)\|_1)=\bT(y), \quad \forall y\geq 0\qquad \text{s.t. } \bT(y)\in\R_+^m.
  \]
\end{theorem}

\begin{proof} 
Each $x_{i,j}$ is a {\em rcll} function,
so it can have at most countably many discontinuities (and they are all jumps).
On the set $C$ defined by 
  \[
    C:=\bigcap_{ i, j=1 }^m \{ u\geq 0:\ x_{i,j}(u-)=x_{i,j}(u) \}= \bigcap_{ l=1 }^m \{u \geq 0: x_{*,l}(u-)=x_{*,l}(u) \}
  \]
clearly all $\underline x_i$ and (therefore) all $g_i$ are (left-)continuous.
The complement of $C$ is either a finite (possibly empty) or a countable subset of $\R_+$. 

Recall that $y \mapsto  \bT(y)$ is (component-wise) strictly increasing (and left-continuous, see Lemma \ref{lem_T_is_lcrl}).
Therefore, 
$ T_i^{-1}(C^c):=\{ y\geq 0:\ T_i(y) \in C^c \}$ is either a finite (possibly empty) or a countably infinite set, for each $i$.
Define
\begin{equation} 
  \label{equ_definition_of_y}
    Y:=\bigcap_{ i=1 }^{ m } T_i^{-1}(C)=\bigcap_{ i=1 }^{ m } \{ y\geq 0:\ T_i(y) \in C \} \subset [0,\infty).
  \end{equation}
The above considerations imply that $Y^c= \bigcup_{ i=1 }^{ m } T_i^{-1}(C^c)$ is a subset of a countable set.
In particular, $Y$ is dense in $[0,\infty)$.

We already know that $\bT(y)$ satisfies 
\eqref{E:important_identity_two} for any $y \geq 0$.
Furthermore, if $y \in Y$, then the identities \eqref{A1} are fulfilled at $\vec t=\bT(y)$. In other words, $g_i(t_i-)=g_i(t_i)$ for each $i$, and in particular $g_i(t_i)=g_{i+1}(t_{i+1})$, for all $i\in [m-1]$.
  
In addition, observe that Lemma~\ref{lem_minimal_solution_via_underline_x}~ii) and the definition of $g_i$ in (\ref{equ_function_g}) imply that
$g_i$ is strictly increasing from the left at $t_i=T_i(y)$ for each $i\in[m]$, so that the additional hypotheses \eqref{P5} of Proposition~\ref{pro_properties_of_kappa_and_gamma} is satisfied at $\bT(y)$ for each $y\in Y$. 
Applying Proposition~\ref{pro_properties_of_kappa_and_gamma} for each $y\in Y$ separately, we conclude that
  \begin{equation} 
  \label{equ_gamma_and_t_for_y}
    \vec\gamma(\|\bT(y)\|_1)=\vec\gamma(s(y))=\bT(y), \quad \forall y \in Y.
  \end{equation}
Lemmas \ref{lem_T_is_lcrl} and \ref{lem_continuity_of_tilde_circ}, joint with the fact that $Y$ is dense in $\R_+$,
now imply the stated claim. 
\end{proof}

{ Recall that here and above $\bbx$ and $\vec\rho$ satisfy \eqref{equ_symmetry_assumption}, and that $y\mapsto \bT(y)$ depends on $\bbx$ and $\vec\rho$, while the map $s \mapsto \vec{\gamma}(s)$ is determined by $\bbx$. Our next goal is to prove Theorem \ref{thm:gammaExistence}(3) which includes the hypothesis that $\bT(y)\in\R_+^m$. We therefore fix a $y$ such that $\bT(y)\in\R_+^m$. This simplifies our analysis of the optimization problem \eqref{E:important_identity_four}. 
}
{  

Indeed, we can now proceed in a way analogous to that in Section \ref{sub:solving_for_T}, relying on the  powerful Theorem \ref{the_t_and_gamma}.
We can { now append an additional condition $\vec{t}=\vec{\gamma}(\|\vec{t}\|_1)$}
to our optimization problem \eqref{E:important_identity_four}, with any given $j\in [m]$ as the reference index.
{The new and equivalent optimization problem is}
\begin{equation}
\label{E:optimization_in_s}
\left\{
\begin{array}{l}
\vec t=\vec\gamma(s),\\
s=\| \vec t\|_1 = \|\vec\gamma(s)\|_1,\\
x_j(\vec\gamma(s)-)=- {\rho_j} y,\\
s \to \min.
\end{array}
\right. 
\end{equation}
{
The first two lines in \eqref{E:optimization_in_s} rely on  
\eqref{P3} and Theorem \ref{the_t_and_gamma}}, the third line comes from 
\eqref{equ_equation_for_t}, and the final { line specifies the optimization rule. After solving for $s(y)\equiv s(y,\bbx,\rho)$, we will use it} to find 
\begin{equation}
\label{E:T_via_s}
\bT(y)=\vec\gamma(s(y)).
\end{equation}
{
It is important to note that}
$s(y)$ solves simultaneously each and every optimization problem
$$
\left\{
\begin{array}{l}
x_j(\vec\gamma(s)-)=- {\rho_j} y,\\
s \to \min,
\end{array}
\right. , \quad j \in [m],
$$
in complete analogy to \eqref{E:T_via_s_sc}--\eqref{E:s_via_x_j_circ}.
The above can be rewritten as $\bT(y)=\vec\gamma(s(y))$, where for all $y\in \R_+$
\begin{equation}
\label{E:s_as_inf}
s(y)=\min\{s\geq 0: x_j(\vec\gamma(s)-)=- {\rho_j} y\}, \quad \forall j \in[m].
\end{equation}
We note that $y\mapsto s(y)$ is a left-continuous non-decreasing function.
{ Indeed, $x_j$ can be replaced in \eqref{E:s_as_inf} with $\underline{x}_j$ from \eqref{equ_definition_of_underline_x_i}, and since $\underline{x}_j$ is non-increasing and continuous, we have that 
\[
s(y)=\min\left\{s\ge 0:-\frac{1}{\rho_j} 
\underline{x}_j(\vec\gamma(s))\ge y\right\}= \inf\left\{s\ge 0:-\frac{1}{\rho_j} 
\underline{x}_j(\vec\gamma(s))\ge y\right\},\quad y\ge 0,
\]
so that $s$ is the left}-continuous generalised inverse
of 
{
$-\frac{1}{\rho_j} 
\underline{x}_j\circ \vec\gamma$.
Furthermore, 
if $y \in Y$ (where $Y$ is the ``good set'' from the proof of Theorem \ref{the_t_and_gamma}) then it is easy to see that $s(y)=\tilde{s}(y)$ where}
\begin{equation}
\label{E:s_as_inf_on_Y}
\tilde{s}(y):=\min\{s\geq 0: x_j(\vec\gamma(s))=- {\rho_j} y\}, \quad \forall j \in[m].
\end{equation}
{ Indeed, the condition in \eqref{E:s_as_inf} is satisfied earlier than the condition in \eqref{E:s_as_inf_on_Y}, so that $s(y)\leq \tilde{s}(y)$ for all $y$. However, if $y\in Y$, then $s(y)$ also solves \eqref{E:s_as_inf_on_Y}, implying the reversed inequality.}

Now define for each $i\in[m]$
\begin{equation}
\label{E:def_S_i}
  S_i(y)=\inf\left\{ s\geq 0:\,{ x_i\circ \vec\gamma(s-)} =\lim_{u\uparrow s}x_i\circ \gamma(u)
  =-\rho_i y \right\}, \quad y\geq 0.
\end{equation}

Note that { $s\mapsto x_i\circ \vec\gamma (s)$} is again a rcll (or \cdl) function { on $[0,+\infty)$, with} no negative jumps. Therefore, when $S_i(y)<\infty$, $\inf$ could be replaced by $\min$. Moreover, { it} is easy to see that $S_i$ is again a left-continuous and non-decreasing function.

\begin{corollary} \label{cor_relation_between_S_and_T}
  The maps $y\mapsto S_i(y)$ and $y\mapsto s(y)$ are identical for each $i\in[m]$. 
In particular, whenever $\bT(y+)\in\R_+^m$ then $S_i(y+)<\infty$ and $$S_i(y+)-S_i(y)=\|\bT(y+)-\bT(y)\|_1$$
for all $i\in[m]$.
\end{corollary}

\begin{proof} 
For the first part of the statement, it is enough to show that $s$ and $S_i$ coincide on the dense set $Y$, because both $S_i$ and $s$ are left-continuous functions.

{
Comparing the conditions in~\eqref{E:s_as_inf_on_Y} and in~\eqref{E:def_S_i},
it is clear that $S_i(y)\leq \tilde{s}(y)$ for any $y$.
However if $y\in Y$, then $s(y)=\tilde{s}(y)$ implying
$$
S_i(y)\leq s(y),\ y\in Y.
$$

To prove the reversed inequality, we use the fact that $x_i$ (and therefore $x_i\circ \vec\gamma$) has no non-negative jumps for each $i\in [m]$ and the monotonicity and continuity of $\vec\gamma$. 
More precisely, $\vec\gamma$ can be either strictly increasing from the left at $r$, or constant on some interval $(r-\eps,r]$ of positive length.
In the former case, $x_i\circ \vec\gamma(r-)$ equals $x_i\circ \vec(\gamma(r)-)$, while in the latter case $x_i\circ \vec\gamma(r-) = x_i\circ \vec\gamma(r)\geq x_i(\vec\gamma(r)-)$.  
The hereby verified inequality 
\[
x_i\circ \vec\gamma(r-)\ge x_i(\vec\gamma(r)-),  \ r\geq 0,
\]
implies that
the condition in \eqref{E:s_as_inf} is satisfied earlier than that in \eqref{E:def_S_i}, therefore 
$$
s(y) \leq S_i(y), \ y \in [0,\infty),
$$
which concludes the argument for $s\equiv S_i$.
}

The second part of the statement directly follows from the (strict) monotonicity of $y \mapsto \bT(y)$ yielding  $\|\bT(y+)-\bT(y)\|_{1}=\|\bT(y+)\|_1-\|\bT(y)\|_1$.
\end{proof}}

\subsection*{Acknowledgements}
DC was partially supported by NSF DMS 2023239. He would also like to thank the Institute for Foundation of Data Science and the University of Washington, where part of this research was conducted.
VK was partially supported by the Deutsche Forschungsgemeinschaft (DFG, German Research Foundation) – SFB 1283/2 2021 – 317210226. 
The second and the third author thank the Max Planck Institute for Mathematics in the Sciences for its warm hospitality, where a part of this research was carried out. The second author is also grateful to IRMA, University of Strasbourg, where a part of this research was started.

\setlength{\bibsep}{0pt}

\providecommand{\bysame}{\leavevmode\hbox to3em{\hrulefill}\thinspace}
\providecommand{\MR}{\relax\ifhmode\unskip\space\fi MR }
\providecommand{\MRhref}[2]{%
	\href{http://www.ams.org/mathscinet-getitem?mr=#1}{#2}
}
\providecommand{\href}[2]{#2}


\begin{thebibliography}{10}
	
	\bibitem{Abbe:2017}
	Emmanuel Abbe, \emph{Community detection and stochastic block models: recent
		developments}, J. Mach. Learn. Res. \textbf{18} (2017), Paper No. 177, 86.
	\MR{3827065}
	
	\bibitem{ABBG.10}
	L.~Addario-Berry, N.~Broutin, and C.~Goldschmidt, \emph{Critical random graphs:
		limiting constructions and distributional properties}, Electron. J. Probab.
	\textbf{15} (2010), no. 25, 741--775. \MR{2650781}
	
	\bibitem{ABBG.12}
	\bysame, \emph{The continuum limit of critical random graphs}, Probab. Theory
	Related Fields \textbf{152} (2012), no.~3-4, 367--406. \MR{2892951}
	
	\bibitem{Aldous.97}
	David Aldous, \emph{Brownian excursions, critical random graphs and the
		multiplicative coalescent}, Ann. Probab. \textbf{25} (1997), no.~2, 812--854.
	\MR{1434128}
	
	\bibitem{AL.98}
	David Aldous and Vlada Limic, \emph{The entrance boundary of the multiplicative
		coalescent}, Electron. J. Probab. \textbf{3} (1998), no. 3, 59. \MR{1491528}
	
	\bibitem{Hernandez:2020}
	Osvaldo {Angtuncio Hern{\'a}ndez}, \emph{{On Multitype Random Forests with a
			Given Degree Sequence, the Total Population of Branching Forests and
			Enumerations of Multitype Forests}}, arXiv e-prints (2020), arXiv:2003.03036.
	
	\bibitem{BBBSW.23}
	Jnaneshwar Baslingker, Shankar Bhamidi, Nicolas Broutin, Sanchayan Sen, and
	Xuan Wang, \emph{Scaling limits and universality: Critical percolation on
		weighted graphs converging to an $ {L}^3$ graphon}, arXiv preprint
	arXiv:2303.10082 (2023).
	
	\bibitem{BerzunzaOjeda.18}
	Gabriel~Hern\'{a}n Berzunza~Ojeda, \emph{On scaling limits of multitype
		{G}alton-{W}atson trees with possibly infinite variance}, ALEA Lat. Am. J.
	Probab. Math. Stat. \textbf{15} (2018), no.~1, 21--48. \MR{3748121}
	
	\bibitem{BBSW.14}
	Shankar {Bhamidi}, Nicolas {Broutin}, Sanchayan {Sen}, and Xuan {Wang},
	\emph{{Scaling limits of random graph models at criticality: Universality and
			the basin of attraction of the Erd{\H{o}}s-R{\'e}nyi random graph}}, arXiv
	e-prints (2014), arXiv:1411.3417.
	
	\bibitem{BSW.17}
	Shankar Bhamidi, Sanchayan Sen, and Xuan Wang, \emph{Continuum limit of
		critical inhomogeneous random graphs}, Probab. Theory Related Fields
	\textbf{169} (2017), no.~1-2, 565--641. \MR{3704776}
	
	\bibitem{Blanc_Renaudie:2024}
	Arthur {Blanc-Renaudie}, Nicolas {Broutin}, and Asaf {Nachmias}, \emph{{The
			scaling limit of critical hypercube percolation}}, arXiv e-prints (2024),
	arXiv:2401.16365.
	
	\bibitem{BJR.07}
	B\'{e}la Bollob\'{a}s, Svante Janson, and Oliver Riordan, \emph{The phase
		transition in inhomogeneous random graphs}, Random Structures Algorithms
	\textbf{31} (2007), no.~1, 3--122. \MR{2337396}
	
	\bibitem{BDW.21}
	Nicolas Broutin, Thomas Duquesne, and Minmin Wang, \emph{Limits of
		multiplicative inhomogeneous random graphs and {L}\'{e}vy trees: limit
		theorems}, Probab. Theory Related Fields \textbf{181} (2021), no.~4,
	865--973. \MR{4344135}
	
	\bibitem{Chatterjee.17}
	Sourav Chatterjee, \emph{Large deviations for random graphs}, Lecture Notes in
	Mathematics, vol. 2197, Springer, Cham, 2017, Lecture notes from the 45th
	Probability Summer School held in Saint-Flour, June 2015, \'{E}cole
	d'\'{E}t\'{e} de Probabilit\'{e}s de Saint-Flour. [Saint-Flour Probability
	Summer School]. \MR{3700183}
	
	\bibitem{Chaumont:2016}
	Lo{\"\i}c Chaumont and Rongli Liu, \emph{Coding multitype forests: application
		to the law of the total population of branching forests}, Transactions of the
	American Mathematical Society \textbf{368} (2016), no.~4, 2723--2747.
	
	\bibitem{Chaumont:2020}
	Lo\"{\i}c Chaumont and Marine Marolleau, \emph{Fluctuation theory for
		spectrally positive additive {L}\'{e}vy fields}, Electron. J. Probab.
	\textbf{25} (2020), Paper No. 161, 26. \MR{4193902}
	
	\bibitem{Chaumont:2021}
	Lo{\"\i}c Chaumont and Marine Marolleau, \emph{Extinction times of multitype
		continuous-state branching processes}, Annales de l'Institut Henri Poincare
	(B) Probabilites et statistiques, vol.~59, Institut Henri Poincar{\'e}, 2023,
	pp.~563--577.
	
	\bibitem{Clancy.24+}
	David {Clancy, Jr.}, \emph{{Component sizes of rank-2 multiplicative random
			graphs}}, In preparation (2024+).
	
	\bibitem{CKL.24+}
	David {Clancy, Jr.}, Vitalii Konarovskyi, and Vlada Limic, \emph{Degree
		corrected stochastic block models: limit theorems}, In preparation (2024+).
	
	\bibitem{CKG.20}
	Guillaume Conchon-Kerjan and Christina Goldschmidt, \emph{The stable graph: the
		metric space scaling limit of a critical random graph with iid power-law
		degrees}, The Annals of Probability \textbf{51} (2023), no.~1, 1--69.
	
	\bibitem{deRaphelis.17}
	Lo\"{\i}c de~Raph\'{e}lis, \emph{Scaling limit of multitype {G}alton-{W}atson
		trees with infinitely many types}, Ann. Inst. Henri Poincar\'{e} Probab.
	Stat. \textbf{53} (2017), no.~1, 200--225. \MR{3606739}
	
	\bibitem{DLV.19}
	Amir Dembo, Anna Levit, and Sreekar Vadlamani, \emph{{Component sizes for large
			quantum Erd{\H{o}}s--R{\'e}nyi graph near criticality}}, The Annals of
	Probability \textbf{47} (2019), no.~2, 1185--1219.
	
	\bibitem{DvdHvLS.20}
	Souvik Dhara, Remco van~der Hofstad, Johan S.~H. van Leeuwaarden, and Sanchayan
	Sen, \emph{Heavy-tailed configuration models at criticality}, Ann. Inst.
	Henri Poincar\'{e} Probab. Stat. \textbf{56} (2020), no.~3, 1515--1558.
	\MR{4116701}
	
	\bibitem{DvdHvLS.17}
	Souvik Dhara, Remco van~der Hofstad, Johan~SH Van~Leeuwaarden, and Sanchayan
	Sen, \emph{Critical window for the configuration model: finite third moment
		degrees}, Electronic Journal of Probability \textbf{22} (2017), 1--33.
	
	\bibitem{DL.02}
	Thomas Duquesne and Jean-Fran\c{c}ois Le~Gall, \emph{Random trees, {L}\'{e}vy
		processes and spatial branching processes}, Ast\'{e}risque (2002), no.~281,
	vi+147. \MR{1954248}
	
	\bibitem{Ethier:1986}
	Stewart~N. Ethier and Thomas~G. Kurtz, \emph{Markov processes: Characterization
		and convergence}, Wiley Series in Probability and Mathematical Statistics:
	Probability and Mathematical Statistics, John Wiley \& Sons, Inc., New York,
	1986. \MR{838085}
	
	\bibitem{Federico.19}
	Lorenzo Federico, \emph{Critical scaling limits of the random intersection
		graph}, arXiv preprint arXiv:1910.13227 (2019).
	
	\bibitem{Hennig:2016}
	Christian Hennig, Marina Meila, Fionn Murtagh, and Roberto Rocci (eds.),
	\emph{Handbook of cluster analysis}, Chapman \& Hall/CRC Handbooks of Modern
	Statistical Methods, CRC Press, Boca Raton, FL, 2016. \MR{3645404}
	
	\bibitem{Janson.10}
	Svante Janson, \emph{Asymptotic equivalence and contiguity of some random
		graphs}, Random Structures \& Algorithms \textbf{36} (2010), no.~1, 26--45.
	
	\bibitem{Joseph.14}
	Adrien Joseph, \emph{The component sizes of a critical random graph with given
		degree sequence}, Ann. Appl. Probab. \textbf{24} (2014), no.~6, 2560--2594.
	\MR{3262511}
	
	\bibitem{Karrer:2011}
	Brian Karrer and Mark~EJ Newman, \emph{Stochastic blockmodels and community
		structure in networks}, Physical review E \textbf{83} (2011), no.~1, 016107.
	
	\bibitem{KL.21}
	Vitalii Konarovskyi and Vlada Limic, \emph{Stochastic block model in a new
		critical regime and the interacting multiplicative coalescent}, Electron. J.
	Probab. \textbf{26} (2021), Paper No. 30, 23. \MR{4235481}
	
	\bibitem{Limic.19}
	Vlada Limic, \emph{The eternal multiplicative coalescent encoding via
		excursions of {L}\'{e}vy-type processes}, Bernoulli \textbf{25} (2019),
	no.~4A, 2479--2507. \MR{4003555}
	
	\bibitem{Martin:2017}
	James~B. Martin and Bal\'{a}zs R\'{a}th, \emph{Rigid representations of the
		multiplicative coalescent with linear deletion}, Electron. J. Probab.
	\textbf{22} (2017), Paper No. 83, 47. \MR{3718711}
	
	\bibitem{Miermont.08}
	Gr\'{e}gory Miermont, \emph{Invariance principles for spatial multitype
		{G}alton-{W}atson trees}, Ann. Inst. Henri Poincar\'{e} Probab. Stat.
	\textbf{44} (2008), no.~6, 1128--1161. \MR{2469338}
	
	\bibitem{NP.07}
	Asaf Nachmias and Yuval Peres, \emph{Component sizes of the random graph
		outside the scaling window}, ALEA Lat. Am. J. Probab. Math. Stat. \textbf{3}
	(2007), 133--142. \MR{2349805}
	
	\bibitem{RY.99}
	Daniel Revuz and Marc Yor, \emph{Continuous martingales and {B}rownian motion},
	third ed., Grundlehren der mathematischen Wissenschaften [Fundamental
	Principles of Mathematical Sciences], vol. 293, Springer-Verlag, Berlin,
	1999. \MR{1725357}
	
	\bibitem{vanderHofstad.17}
	Remco van~der Hofstad, \emph{Random graphs and complex networks. {V}ol. 1},
	Cambridge Series in Statistical and Probabilistic Mathematics, [43],
	Cambridge University Press, Cambridge, 2017. \MR{3617364}
	
	\bibitem{Wang.23}
	Minmin Wang, \emph{Large random intersection graphs inside the critical window
		and triangle counts}, arXiv preprint arXiv:2309.13694 (2023).
	
\end{thebibliography}
\end{document}